\documentclass[12pt]{article}
\usepackage{latexsym,amssymb,amsgen,amsmath,amsxtra, epsfig,
amsfonts, amscd, color, amsthm}
\usepackage{fullpage}
\usepackage{mathbbol}
\usepackage{eucal}
\usepackage{authblk}
\usepackage{tikz}

\numberwithin{equation}{section}

\usepackage[linktocpage=true]{hyperref} 

\usepackage{comment} 
\excludecomment{TOOLONG} 

\usepackage{graphicx}
\definecolor{mygray}{gray}{0.6} %

\makeatletter
\def\fullpath{\begingroup\everyeof{\noexpand}\@sanitize
  \edef\x{\@@input|"find `pwd` -name \jobname.tex" }%
  \edef\x{\endgroup\noexpand\zap@space\x\noexpand\@empty}\x}
\makeatother

\newcounter{mnotecount}[section]

\newcommand{\R}{{\mathbb{R}}}
\newcommand{\Z}{{\mathbb{Z}}}

\newcommand{\dbar}{{\mathchar'26\mkern-11mud}}

\def\be{\begin{equation}}
\def\ee{\end{equation}}
\def\ba{\begin{array}}
\def\ea{\end{array}}
\def\vp{\varphi}

\def\tV{{\tilde V}}

\newcommand{\curl}{\mathop{\rm curl}}

\newtheorem{theorem}{Theorem}[section]
\newtheorem{proposition}[theorem]{Proposition}
\newtheorem{corollary}[theorem]{Corollary}
\newtheorem{lemma}[theorem]{Lemma}
\newtheorem{remark}[theorem]{Remark}

\begin{document}

\title{Cauchy Problem for Incompressible Neo-Hookean materials}

\author[1]{Lars Andersson \thanks{laan@aei.mpg.de}}
\affil[1]{Max Planck Institute for Gravitational Physics (Albert Einstein Institute), Am M\"uhlenberg 1, D-14476 Potsdam, Germany }

\author[2]{Lev Kapitanski \thanks{l.kapitanski@math.miami.edu}}%
\affil[2]{Department of Mathematics, University of Miami, Coral Gables, FL 33124, USA} 

\date{November 6, 2021}

\maketitle

\begin{abstract}
In this paper we consider the Cauchy problem for neo-Hookean incompressible elasticity in spatial dimension $d \geq 2$. The Cauchy problem  can be formulated in terms of maps $x(t,\cdot): \R^d_\xi \to \R^d_x$ with domain a reference space $\R^d_\xi$, and with values in space $\R^d_x$. Initial data consists of initial deformation $\phi(\xi) = x(0,\xi)$ and velocity $\psi(\xi) = \partial x(t,\xi)/\partial t |_{t=0}$, which we assume are in Sobolev spaces $(\phi, \psi) \in H^{s+1}(\R^d)\times H^{s}(\R^d)$. If $s>s_{crit}=  d/2+1$, well-posedness is well-known.  We are here interested primarily in the low regularity case, $s \le s_{crit}$.   
For $d = 2, 3$, we prove existence and uniqueness  for $s_0 < s\le s_{crit}$, and we can prove  well-posedness, but  for a smaller range, $s_1 < s \le  s_{crit}$, 
\begin{align*} 
\text{if $d=2$}{}&, \quad s_0 = \frac74, \quad s_1= \tfrac74 + \tfrac{\sqrt{65}-7}{8} \\  
\text{if $d=3$}{}&, \quad s_0 = 2, \quad s_1 = 1 + \sqrt{\tfrac32} 
\end{align*} 
We consider the initial deformations of the form $x(0, \xi) = A \xi + \varphi(\xi)$, where $A$ is a constant $SL(d, \R)$ matrix. For the full range (in $s$) results, as indicated above, we need  additional H\"older 
regularity assumptions on certain combinations of second order derivatives of $\varphi$. 
 A key observation in the proof is that the equations of evolution for the vorticities decomposes into a first-order hyperbolic system, for which a Strichartz estimate holds, and a coupled transport system. This allows one to set up a bootstrap argument to prove local existence and uniqueness. Continuous dependence on initial data is proved using an argument inspired by Bona and Smith, and Kato and Lai, with a modification based on new estimates for Riesz potentials. The results of this paper should be compared to  what is known for the ideal fluid equations, where, as shown by Bourgain and Li, the requirement $s > s_{crit}$ is necessary. 
\end{abstract}

\tableofcontents

\section{Introduction}\label{intro}

We are concerned here with the motion of the incompressible neo-Hookean medium (material). We assume  that the medium occupies the whole space $\R^d$ with coordinates $\xi$. 
A generic point $\xi\in\R^d$ will find itself at the location $x(t, \xi)$ at time $t$. Incompressibility means that 
\be\label{det=1}
\det \frac{\partial x(t, \xi)}{\partial \xi} = 1\,.
\ee
The action functional 
\[
\int \int \frac12\,\left|\frac{\partial x(t, \xi)}{\partial t}\right|^2  + p(t, \xi)\,\left(\det \frac{\partial x(t, \xi)}{\partial \xi} - 1\right)\;d\xi\,dt\,,
\]
with the Lagrange multiplier $p$, gives rise to the Euler-Lagrange equations 
\be\label{idealfluid}
\frac{\partial^2 x(t, \xi)}{\partial t^2}  + \nabla_x p(t,\xi) = 0
\ee
describing the ideal (incompressible and inviscid)  fluid, cf. \cite{MR1784335}.
An ideal fluid has only kinetic energy. 
If there are internal forces controlling the deformation $x(t, \xi)$, one has to add the potential energy. 
If this potential energy depends on the deformation gradient $\partial x/\partial\xi$ only, the material is hyperelastic 
and the action takes form  
\[
\int \int \frac12\,\left|\frac{\partial x(t, \xi)}{\partial t}\right|^2 - W\left(\frac{\partial x(t, \xi)}{\partial \xi}\right) + p(t, \xi)\,\left(\det \frac{\partial x(t, \xi)}{\partial \xi} - 1\right)\;d\xi\,dt\,.
\]
(Of course, there are conditions on the meaningful stored energy functions $W$, cf. \cite{MR1262126}.) The corresponding Euler-Lagrange equations are 
\be\label{gen}
\frac{\partial^2 x^i}{\partial t^2} -  D^j_bD^i_aW\left(\frac{\partial x}{\partial\xi}\right)\;\frac{\partial^2 x^j}{\partial\xi^a \partial \xi^b} + \frac{\partial p}{\partial x^i} = 0\,,\quad i = 1,\dots, d,
\ee
where $D^i_a$ stands for the derivative $\partial/\partial(\partial x^i/\partial\xi^a)$. The summation over the repeated 
indices is assumed\footnote{We use lower case latin indices  $a,b,c$ for Lagrangian coordinates, and  $i,j,k,\dots$ for Eulerian coordinates. In the summation the repeated indices run from $1$ to $d$.}. The simplest  choice  
\[
W\left(\frac{\partial x}{\partial \xi}\right) = \frac12\,\left|\frac{\partial x}{\partial \xi}\right|^2 = \frac12\,\sum_{i, a = 1}^d \left|\frac{\partial x^i}{\partial \xi^a}\right|^2\,
\] 
corresponds to the neo-Hookean material, or solid, as it is sometimes  called. The field equations \eqref{gen} then simplify to
\be\label{NH-L}
\frac{\partial^2 x^i}{\partial t^2} -  \frac{\partial^2 x^i}{\partial\xi^a\partial\xi^a} + \frac{\partial p}{\partial x^i} = 0\,,\quad i = 1,\dots, d.
\ee
Equations \eqref{NH-L} complemented by the constraint \eqref{det=1}
describe the motion of the incompressible neo-Hookean  material {\it in the Lagrange coordinates} $ (t, \xi)$. 

\begin{remark}[Initial conditions]\label{rem-1} 
We shall consider the Cauchy problem for \eqref{idealfluid} and \eqref{NH-L}. Therefore, we prescribe the initial deformation $\phi(\xi)$ and the initial velocity $\psi(\xi)$  by
\begin{align*} 
\phi(\xi) ={}& x(0, \xi) \,,\\
\psi(\xi) ={}& \frac{\partial x(t, \xi)}{\partial t} \bigg|_{t=0}\,. 
\end{align*}
Of course, $\phi$ 
should define a volume preserving (and orientation preserving) diffeomorphism   of $\R^d$. In addition, there is a compatibility condition for $\phi$ and $\psi$ originating in \eqref{det=1}. Indeed, if \eqref{det=1} holds,
\begin{align*} 
0 =\frac{\partial\hfil}{\partial t} \left(\det \frac{\partial x(t, \xi)}{\partial \xi}\right) ={}& 
D^i_a \left(\det \frac{\partial x(t, \xi)}{\partial \xi}\right)\;\frac{\partial\hfil}{\partial t}\frac{\partial x^i(t, \xi)}{\partial \xi^a} \\
={}& \hbox{trace}\left\{\left[\left(\frac{\partial x(t, \xi)}{\partial \xi}\right)^{-1}\right]^T\,\frac{\partial\hfil}{\partial t}\frac{\partial x(t, \xi)}{\partial \xi}\right\}
\end{align*}
Thus, if $y = \phi(\xi)$ and $\xi = \phi^{-1}(y)$, we must have
\be\label{compat=1}
\frac{\partial \psi^i(\xi)}{\partial \xi^a}\;\frac{\partial \xi^a}{\partial y^i} = 0
\ee
for all $\xi\in\R^d$. In the case of the ideal fluid, one can change variables 
from $\xi$ to $\eta = \phi(\xi)$, set $x(0, \eta) = \eta$, and recalculate $\psi(\xi(\eta)) = \partial x(t, \eta)/\partial t \big|_{t=0}$. Such a change does not affect equation \eqref{idealfluid} but simplifies the initial conditions. In the case of equation \eqref{NH-L}, such a change of variables, in general, changes the form of the equation and will not be done. 
\end{remark}
\medskip

In {\it the Euler picture}, the independent variables are $(t, x)$, and equation \eqref{idealfluid} for the ideal incompressible fluid takes the form
\be\label{Euler}
\partial_t v^i + v^j\,\partial_j v^i + \partial_i p = 0\,,\quad i = 1,\dots, d,
\ee
where
\be
v(t, x) = \frac{\partial x}{\partial t}
\ee
is the vector of velocity in the Eulerian coordinates. The incompressibility condition \eqref{det=1} translates into 
\be\label{incomp}
\hbox{div}\,v \equiv \partial_i v^i = 0\,.
\ee
To describe the Eulerian form of equation \eqref{NH-L} for the incompressible neo-Hookean material, in addition to the true velocity $v(t,x)$ we introduce $d$ ``fake"  velocities $(v_1^i), \dots, (v_d^i)$  which represent the deformation gradient, 
\be
v_a^i (t, x) = \frac{\partial x^i}{\partial \xi^a}\,,\quad a = 1, \dots, d\,.
\ee
To distinguish between the derivatives in the two  pictures, we denote by $\partial_t$ and $\partial_k$ the partial derivatives in the Euler picture and by $\partial/\partial t$, $\partial/\partial \xi^a$ the derivatives in the Lagrange picture.  The dynamic equations  are 
\begin{align}
 \partial_t  v^i + v^j\, \partial_j v^i -  v^k_b\, \partial_k  v^i_b + 
\partial_i p ={}& 0\,,\label{NH-E-1}\\ 
\partial_t v^i_a + v^k\,\partial_k v^i_a - v^k_a\,\partial_k v^i ={}& 0\,, \;\;a=1,\dots, d\,, \label{NH-E-2} 
\end{align}
for each $i = 1, \dots, d$.
The first equation, \eqref{NH-E-1}, is the Eulerian form of \eqref{NH-L}, while \eqref{NH-E-2} expresses the fact that the partial derivatives $\partial/\partial t$ and $\partial/\partial \xi^a$ commute. 
In addition, we must have
\be\label{NH-E-3}
\hbox{div} \, v = 0, \quad \hbox{div} \, v_a \equiv \partial_i v^i_a = 0,\;\;a=1,\dots, d\,,
\ee
where the equation for $v_a$'s is the Piola identity (in the incompressible case), cf. \cite[\S 1.7]{MR1262126}. There are additional compatibility conditions 
\be\label{compat_ab}
v^k_a\;\partial_k  v^i_b = v^k_b\;\partial_k v^i_a\,,\quad a, b=1,\dots, d, 
\ee
which represent the fact that the Lagrangian derivatives $\partial/\partial \xi^a$ and $\partial/\partial \xi^b$ 
commute. 

If $x(t, \xi)$ and $p(t, \xi)$ is a solution of \eqref{NH-L}, \eqref{det=1}, then 
\be
v^i = \frac{\partial x^i(t, \xi)}{\partial t} \,,\quad v_a^i = \frac{\partial x^i(t, \xi)}{\partial \xi^a}\,,
\ee
and the pressure $p$, expressed as functions of $t$ and $x$, solve equations \eqref{NH-E-1}, \eqref{NH-E-2}, \eqref{NH-E-3}, and \eqref{compat_ab}.  Conversely, if ${ v}$, ${  v}_a$, and $p$ solve 
\eqref{NH-E-1}, \eqref{NH-E-2}, \eqref{NH-E-3}, and \eqref{compat_ab}, then  
the solution $x(t, \xi)$ of the ODE
system
\be
\frac{dx}{dt} = v(t,x)\,,\,\quad x(0, \xi) = \phi(\xi)\,,
\ee
(and $p$ with the corresponding coordinate change $(t, x) \to (t, \xi)$) solves the Lagrange equations \eqref{NH-L}; see also section \ref{sec:vort-lag}.

In the mechanics/engineering literature, the array $(\partial x^i(t, \xi)/\partial \xi^a)$ (in Lagrangian coordinates $t, \xi$) is denoted $F = (F^i_{a})$ and is called the deformation gradient. 
We use the notation $v^i_a$ in the Euler coordinates instead to emphasize the similarity of the subsequent treatment of the vectors $v_a$ and the velocity $v$. 

\bigskip

In this paper we study the Cauchy problem for the equations of motion of the neo-Hookean material, in  Eulerian form \eqref{NH-E-1}, \eqref{NH-E-2}, \eqref{NH-E-3}, and \eqref{compat_ab}. At the same time, we obtain corresponding results for the equations in the Lagrangean  picture. 
Energy estimates 
for equations \eqref{NH-E-1} and \eqref{NH-E-2} are essential for our study. Thus we use the Sobolev spaces $H^s(\R^d)$ to measure regularity and integrability of $v$ and $v_a$. In particular, we will have the gradients $\nabla_x v_a(t, x)$, changing continuously with $t$, in $L^2(\R^d, dx)$.  In dimensions $d\ge 3$, this alone imposes 
a restriction on the behavior at infinity of the initial diffeomorphism $\phi(\xi)$: as $|\xi|\to \infty$,  $\phi(\xi) \to A\,\xi$, where $A$ is a (any) matrix in $SL(d, \R)$. In the two dimensional case we impose 
this condition on $\phi(\xi)$. 
Once the matrix $A$ is fixed, we  split the deformation gradient $v^i_a(t, x)$ as 
\be\label{A-split}
v^i_a(t, x) = A^i_a + u^i_a(t, x)
\ee
and  work with the vectorfields $u_a$ instead of $v_a$. 
\begin{remark}It appears that in the existing literature, cf. e.g.
\cite{Schochet,ES,MR2358646}, only the case $A^i_a = \delta^i_a$ has been considered. Writing the deformation gradient tensor as the identity matrix plus the displacement gradient tensor is 
a tradition in elasticity theory.
\end{remark}
Denote by $V$ the collection of all the components of $v$ and $u_1, \dots, u_d$. 
We work in the scale of Sobolev spaces $H^s(\R^d)$ and compare with the existing results for the fluid equations 
\eqref{Euler}, \eqref{incomp}.
The regularity $H^s$ now refers to the $H^s$ norm of $V$ in the neo-Hookean case and the norm of $v$ in the fluid case. 
We are interested in local well-posedness in $H^s$, by which we mean local in time existence, uniqueness, and continuous dependence on the initial conditions (in $H^s$).  
In the subcritical case $s > d/2 + 1$, the well-posedness for the fluid equations is known, cf. \cite{Kato-Lai,KP}. 
According to the recent results of Bourgain and Li \cite{BL-0,BL-2}, well-posedness fails to hold for $s\le d/2+1$. 

There are results on local well-posedness in the case $s > d/2 + 1$ for a class of hyperelastic systems which includes  the neo-Hookean case, cf. \cite{ES}.  
Here we consider specifically the neo-Hookean case and 
 show that  then the system has smoothing properties that allow us to lower the 
regularity requirements for well-posedness compared to the fluid case. In particular, we can prove well-posedness 
below the critical regularity $s_{crit} = d/2 + 1$.  We believe that an analogous result holds for more general hyperelastic systems. 

\bigskip

The outline of the paper is as follows. We work with equations \eqref{NH-E-1}, \eqref{NH-E-2}, and \eqref{NH-E-3}, 
with initial data satisfying 
$v(0,\cdot)\in H^s$ and $u_a^i(0, x) = v_a^i(0, x) - A_a^i$. The compatibility condition 
\eqref{compat_ab} is satisfied for the initial data and will be propagated by the flow. 
 The basic {\it a priori} energy estimates are 
\be\label{ener} 
\|V(t)\|_{H^s} \lesssim \|V(0)\|_{H^s}\,\exp{\displaystyle{\left(c\,\int_0^t \|\nabla V(t^\prime)\|_{L^\infty}\,\,dt^\prime\right)}}
\ee
which are valid for any $s \ge 0$. Here and below, we write $\nabla V$ for $\nabla_x V$, when there is no room for confusion. If $s=0$, the energy is conserved:
\be\label{ener-0}
\|V(t)\|_{L^2} = \|V(0)\|_{L^2}\,.
\ee
As in the fluid case, it is clear that the solution exists as long as  the integral
\be\label{integral}
\int_0^t \|\nabla V(t^\prime)\|_{L^\infty}\,\,dt^\prime
\ee
is finite. 
If $s > d/2+1$, the Sobolev embedding bounds $\|\nabla V(t^\prime)\|_{L^\infty}$ by $\|V(t^\prime)\|_{H^s}$, and 
then the differential inequality resulting from \eqref{ener}  provides a bound on  the norm $\|V(t)\|_{H^s}$ for a short time interval 
depending on $\|V(0)\|_{H^s}$. To prove well-posedness, we use a modified Kato-Lai approach, cf. \cite{Kato-Lai}.

Recall that in the fluid case it has proved 
profitable to move from velocity to vorticity.
With the help of the Beale-Kato-Majda (BKM) estimate \cite{BKM}, 
it is possible to replace $\int_0^t \|\nabla v(t^\prime)\|_\infty\,dt^\prime$ by the integral 
$\int_0^t \|\omega(t^\prime)\|_\infty\,dt^\prime$. Note that the application of the  BKM estimate uses the fact that 
$s > d/2 + 1$. In the two-dimensional case, since the vorticity is transported by the flow, the solutions exist for all time. In the neo-Hookean case, the vorticity equations are not so nice, but they have some redeeming features. 
We do not obtain the global existence in the two-dimensional case,\footnote{There is a huge literature on global existence for small initial data, but here we concentrate on low regularity.} but we can lower the regularity requirements 
to a certain range of $s\le d/2 + 1$. 
Instead of the BKM estimate we use a different approach.  
 
 In our case we have two types of vorticities: the true vorticity, and $d$ ``fake" vorticities  
\be\label{vortic}
\omega^{mn} = \partial_m v^n - \partial_n v^m, \quad \omega_a^{mn} = \partial_m v_a^n - \partial_n v_a^m, \;\;a = 1, \dots, d\,.  
\ee
We combine them into an aggregate $\Omega$. In dealing with the integral \eqref{integral}, the following estimate 
will play an important role:
\be\label{nablavinfty}
\|\nabla V\|_\infty \lesssim \|V\|_{L^2}^{\gamma_1}\;\|\Omega\|_{{\dot B}^r_{p, p}}^{1 - \gamma_1}\,,
\ee
where ${\dot B}^r_{p, p}$ is the homogeneous Besov space, and $\gamma_1$  is function of $d$, $r$, and $p$, for a certain range of those parameters, see Lemma~\ref{v-r}.   Thanks to \eqref{ener-0}, we may replace \eqref{integral} with 
\be\label{integral2}
\int_0^t \|\Omega(t^\prime)\|_{{\dot B}^r_{p, p}}^{1 - \gamma_1}\,dt^\prime\,.
\ee
It turns out that the equations for $\omega^{mn}$ and $\omega_a^{mn}$ in the Lagrangean coordinates have the form 
(the superscript $L$ refers to the Lagrangean coordinates)
\begin{subequations}\label{Lom}
\begin{align}
\frac{\partial { \omega}^{L}}{\partial t} - \frac{\partial { \omega}_a^{L}}{\partial \xi^a} =& { F}^{L}[\nabla V]\,
\label{Lom-1}\\ 
\frac{\partial {\omega}_b^{L}}{\partial t} - \frac{\partial { \omega}^{L}}{\partial \xi^b} =& { F}_b^{L}[\nabla V]\,\label{Lom-2}
\end{align}
\end{subequations}
where the right hand sides are quadratic in $\nabla V$.  Via the Fourier transform (in Lagrangean coordinates) this system splits into a pair of wave equations 
with the principal linear parts $\frac{\partial\hfil}{\partial t} \pm i \sqrt{-\Delta_\xi}$ and a simple transport equation 
driven by the wave system. We use the Strichartz estimates involving the ${\dot B}^r_{p,p}$-norms  for the solutions of the wave equations (see \cite{Kap-2}), and the transport equations allow us to estimate the remaining components. 
In addition, we need to estimate the corresponding Sobolev norms of $\Omega^L$. 
The interplay between the various norms in the Lagrangean and the Eulerian coordinates is quite subtle, bringing in further powers of various norms. However, we are able to  close the argument by using nonlinear Gronwall-type inequalities. As a result, we show a stronger estimate of the form 
\be
\int_0^{T} \|\nabla V(t)\|_\infty^{1+\delta}\,dt \le C\,,
\ee
where $C>0$ and $T > 0$ are determined by the $H^s$ norms of the initial conditions, and $\delta$ is a dimension dependent positive constant. The estimates close for the range of $s$ larger than some $s_0 < d/2 + 1$, thus  lowering 
the classical regularity. 
Continuous dependence brings additional restrictions $s > s_1$, where $s_1 > s_0$, but $s_1 < d/2 + 1$.  We follow the general plan suggested by Bona and Smith \cite{BS}  and, in the case of ideal fluid, by Kato and Lai \cite{Kato-Lai}. 
The key step is to show that the family of solutions 
$V^\epsilon(t)$ corresponding to the mollified initial condition $V^\epsilon(0)$ is Cauchy, as $\epsilon\to 0$,  in $C([0, T]\to H^s)$ uniformly for $V(0)$ in a compact subset of $H^s$.  
The main difference from \cite{Kato-Lai} is that we use estimate \eqref{nablavinfty} and a similar estimate 
\be\label{eq:vinfty}
\| V\|_\infty \lesssim \|V\|_{L^2}^{\gamma_2}\;\|\Omega\|_{{\dot B}^r_{p, p}}^{1 - \gamma_2}\,, 
\ee
to control the difference $V^\epsilon(t) - V^\delta(t)$, where $\epsilon > \delta \searrow 0$. The restriction on regularity comes, in the end, through the parameter $r$, whose range is dictated by the Strichartz estimates. 

We state and prove the main theorem in the physical cases $d=2$ and $d=3$.  The statement is easiest to present when the curl of the displacement gradient at $t=0$,  $
\curl u_a(0, x)$, is H\"older continuous with an appropriate H\"older index $\varkappa > 0$.
Denote
\be\label{s0}
s_0 = \begin{cases} \frac74 & \text{if}\; d = 2 \\ 
2 & \text{if}\; d = 3
\end{cases}\,,\quad 
s_1 = \begin{cases} \frac74 + \frac{\sqrt{65}-7}{8} & \text{if}\; d = 2 \\ 
1 + \sqrt{\frac32} & \text{if}\; d = 3
\end{cases},\quad 
\varkappa = \begin{cases} \frac{\sqrt{65} - 7}{8} & \text{if}\;d = 2\\ 
\sqrt{\frac{3}{2}} - 1 & \text{if}\; d = 3
\end{cases} .
\ee
\begin{theorem}\label{main}  Assume the initial conditions  $V(0) = (v(0), u(0))$ are such that $v(0)\in H^s(\R^d)$ and $\curl u_a(0) \in{C^\varkappa}(\R^d)$. 
 If $s > s_0$, then there exists a unique solution $V\in C([0, T]\to H^s(\R^d))$ of the system \eqref{NH-E-1}, \eqref{NH-E-2}, and 
 \eqref{NH-E-3}, for some $T > 0$ depending continuously on $\|v(0)\|_{H^s}$. 
 
If $s > s_1$ (if $d = 2$ then $s = s_1$ is allowed), then the solution depends continuously in $C([0, T^\prime]\to H^s(\R^d))$ on the initial data: if $v_n(0)\to v(0)$ in $H^s$ then $V_n\to V$ in  $C([0, T^\prime]\to H^s(\R^d))$, where 
$[0, T^\prime]$ is the common interval of existence, $T^\prime > 0$. 
\end{theorem}
When the initial value of the displacement gradient is less smooth, the regularity of the ``vorticities" $\curl u_a(0) = \omega_a(0)$ will control the values of $s_0$ and $s_1$ (they will increase). A careful analysis of this situation is done in Sections~\ref{sec:low2} and \ref{sec:low3}. 
The proofs in the cases $d=2$ and $d=3$ are somewhat different. 
We believe that the proof in the case $d = 3$ can be generalized to $d > 3$.

\section{%
Above the critical regularity: $s > \frac{d}{2}+1$} \label{sec:Part1}

\subsection{Cauchy problem in Euler coordinates}\label{sec:high}

\subsubsection{Preliminary notes}\label{subsec:prelim} 

A few words on function spaces (for more see Appendix \ref{sec:spaces}).  The notation $\|f\|_p$ is used for the $L^p(\R^d)$ norm, 
$1\le p\le\infty$. Also, $\|f\|$ is the $L^2$ norm. 
We work in the scale of standard Sobolev spaces $H^s = H^s(\R^d)$, $s\in \R$. 
These are Hilbert spaces 
with the norms $\|f\|_{H^s} = \|J^s f\|$, where $J^s = (1 - \Delta)^{s/2} = {\cal F}^{-1}(1 + |k|^2)^{s/2}\,{\cal F}$ for real $s$ are the Bessel potentials.
We use the same notation $H^s$ for the spaces of $\R$-valued, $\R^d$-valued, or matrix-valued functions.  In the case of $\R^d$-valued functions, $H^s$ splits into an orthogonal sum (the Hodge decomposition) 
$H^s = H^s_\sigma\oplus H^s_\nabla$, where $H^s_\sigma$ is the space of divergence-free vector fields  
and $H^s_\nabla$ is the space of gradients of $H^{s+1}$ scalar functions.

We are going to work with the dynamic equations \eqref{NH-E-1} and \eqref{NH-E-2} (and \eqref{NH-E-3}) ignoring for the most part the compatibility condition \eqref{compat_ab}. However, at some point we will have to justify \eqref{compat_ab}. Also, we will have to explain why $v^i_a$ is the deformation gradient. The proof of these facts 
makes use of the estimate
\be\label{int-v-infty}
\int_0^T \|\nabla v(t)\|_\infty\,dt < \infty\,
\ee 
for the solutions on the interval $[0, T]$. 
This will be shown to hold for the solutions that we consider. Assuming \eqref{int-v-infty}, 
consider first the quantities 
$q^i_{ab} = v^k_a\;\partial_k v^i_b - v^k_b\;\partial_k v^i_a$. According to equations \eqref{NH-E-2}, these quantities satisfy the following equations:
\be
(\partial_t + v^\ell\partial_\ell) q^i_{ab} = q^r_{ab}\,\partial_r v^i\;.
\ee 
Multiply each equation by $q^i_{ab}$, sum over $i$, $a$, and $b$, and integrate in $x$ over $\R^d$. This yields the inequality 
\[
\frac12 \frac{d\hfill}{dt}\,\int \sum |q^i_{ab}|^2\,dx \le \|\nabla v(t)\|_\infty\,\int \sum |q^i_{ab}|^2\,dx\,.
\]
Therefore, if all $q^i_{ab} = 0$ at $t = 0$, then $q^i_{ab}$ will vanish for all $t \in [0, T]$. 

Similarly, if $v(t), v_a(t)$ is a solution of  \eqref{NH-E-1}, \eqref{NH-E-2}, and \eqref{NH-E-3}, for which \eqref{int-v-infty} is true, 
consider the solution of the ODE system $\dot x = v(t, x)$ with the initial condition 
$x(0, \xi) = A\xi + \varphi(\xi)$. Consider the quantity $q^i_a = v^i_a - \partial x^i/\partial \xi^a$ and compute 
its full time derivative:
\[
(\partial_t + v^\ell\partial_\ell) q^i_a = (\partial_t + v^\ell\partial_\ell) v^i_a - \frac{\partial^2 x^i}{\partial t\partial \xi^a}\,. 
\]
The first term on the right equals $v^k_a\partial_k v^i$ by equation \eqref{NH-E-2}, while 
\[
\frac{\partial^2 x^i}{\partial t\partial \xi^a} = \frac{\partial v^i}{\partial \xi^a} = 
\frac{\partial x^k}{\partial \xi^a}\,\partial_k v^i\,.
\]
Thus, 
\[
(\partial_t + v^\ell\partial_\ell) q^i_a = q^k_a \partial_k v^i\,. 
\]
As with the quantity $q^i_{ab}$, we obtain the inequality
\[
\frac12\,\frac{d\hfill}{dt}\,\int \sum |q^i_{a}|^2\,dx \le \|\nabla v(t)\|_\infty\,\int \sum |q^i_{a}|^2\,dx\,.
\]
Thus, if all $q^i_a$ are $0$ at $t = 0$, they remain $0$ for $t \in [0, T]$. 
\bigskip

We emphasize that the results in this subsection do not require $s > d/2 + 1$.

\subsection{The Cauchy problem}

In the Euler picture, we assume that the deformation gradient $(v^i_a)$ is split as in \eqref{A-split}. Thus, we fix an $SL(d, \R)$ matrix $A = (A^i_a)$ and define $d$ vectorfields $u_1, \dots, u_d$ with the components $u^i_a(t, x) = v^i_a(t,x) - A^i_a$, $i=1,\dots, d$.  Equations  \eqref{NH-E-1}, \eqref{NH-E-2}, and \eqref{NH-E-3} now look as follows: 
\begin{subequations}\label{eq:A-u-system}
\begin{align}
& \partial_t v^i + v^j\, \partial_j v^i -  u^k_b\, \partial_k u^i_b  - A^k_b\,\partial_k u^i_b + 
\nabla^i p = 0\,,\label{A-u-1}\\ 
& \partial_t u^i_a + v^k\,\partial_k u^i_a - u^k_a\,\partial_k v^i -  A^k_a\,\partial_k  v^i = 0\,, \label{A-u-2}\\ 
& \hbox{div}\,v = 0,\quad  \hbox{div}\, u_a = 0\,. \label{A-u-3}
\end{align}
\end{subequations}
What we do next in this paper does not depend on the particular choice of matrix $A$. The important thing is that the terms containing  the elements of $A$  disappear in the derivation of energy estimates.  
Here is a typical calculation showing this: 
\[
\int - A^k_b\,\partial_k u^i_b\,v^i -  A^k_a\,\partial_k  v^i\,u^i_a\;dx = 0\,,
\]
Because of that, we choose the simplest $A = I$, the identity matrix, and work with the resulting equations: 
\begin{subequations}\label{eq:u-system}
\begin{align}
& \partial_t v^i + v^j\, \partial_j v^i -  u^k_b\, \partial_k u^i_b  - \partial_b u^i_b + 
\nabla^i p = 0\,,\label{u-1}\\ 
& \partial_t u^i_a + v^k\,\partial_k u^i_a - u^k_a\,\partial_k v^i -  \partial_a  v^i = 0\,, \label{u-2}\\ 
& \hbox{div}\,v = 0,\quad  \hbox{div}\, u_a = 0\,. \label{u-3}
\end{align}
\end{subequations} 
We shall use $u$ to denote the whole collection $u_1,\dots, u_d$, and any norm of $u$ will be understood as the maximum over $a$ of the norm of $u_a$. 
 We abbreviate $v(t)$ for $v(t, \cdot)$ etc. Recall that $V(t) = (v(t), u(t))$.

\begin{theorem}\label{Thm-1}
Assume $d \geq 2$. 
Let $s$ be a real number greater than $1 + d/2$. Assume that the initial conditions $v(0), u(0)$ for equations \eqref{u-1}, 
\eqref{u-2}, and \eqref{u-3} all belong to $H^s_\sigma$. Then there exists a local in time solution of those equations such that 
\be
v, u\in C([0, T]\to H^s_\sigma)\,.
\ee
The solution is unique and depends continuously on the initial conditions.  
The lifespan, $T$, of the solution is determined by the $H^s_\sigma$-norms of the initial conditions and is characterized by the condition 
\be
\int_0^t \|\nabla V(\tau)\|_\infty \,d\tau < \infty\,,\quad\forall t < T\,.
\ee
\end{theorem}

The general plan of the proof will be the same as in  the proof of the corresponding result for the Euler equations 
(describing an ideal fluid) by Kato and Lai, \cite{Kato-Lai},  complemented by some nice observations from \cite{FMRR-1}. We will sketch it anyway to emphasize certain points  that will be used later.
Denote the orthogonal 
projection $H^s\to H^s_\sigma$ by 
${\cal P}$. In the Fourier space,   
\[
\widehat{\left({\cal P} w\right)^i}(k) = \left(\delta^{ij} - \frac{k^i k^j}{|k|^2}\right){\hat w}^j(k)\,.
\]
Since Leray's work \cite{MR1555394}, projection 
of the equations onto the space of divergence-free vectors is standard in hydrodynamics, and we use it as well to get rid of the pressure term. Thus, we deal with the following equations:
\begin{align}
& \partial_t v^i + {\cal P}\left\{v^j\, \partial_j v^i -  u^k_b\, \partial_k u^i_b  - \partial_b u^i_b\right\} = 0\,,\label{u-11}\\ 
& \partial_t u^i_a + v^k\,\partial_k u^i_a - u^k_a\,\partial_k v^i -  \partial_a v^i = 0\,. \label{u-21}
\end{align}
We shall not need equations \eqref{u-3} since $\partial_t \hbox{div}\, v = 0$ follows from equation \eqref{u-11} and, from \eqref{u-21},  
\begin{align*} 
\partial_t \hbox{div}\,u_a + v^k\,\partial_k \hbox{div}\,u_a ={}&  u^k_a \partial_k \hbox{div}\, v + \partial_a  \hbox{div}\,v \,.
\end{align*} 
This shows that if equations \eqref{u-3} are satisfied at $t=0$, they will be satisfied for all times $t > 0$. 

Let $v(0)\in H^s_\sigma$ and $u_a(0)\in H^s_\sigma$  be given. The solution of \eqref{u-11}, \eqref{u-21} with these initial conditions will be obtained as a limit of approximate solutions 
$v^\epsilon, u_a^\epsilon$ whose Fourier transforms are supported in the ball $\{|k|\le \epsilon^{-1}\}$, cf. \cite{FMRR-1}. 

Denote by $\rho_\epsilon$ the  Friedrichs' mollifiers $\rho_\epsilon(x) = \epsilon^{-d} \rho(x/\epsilon)$, where $\rho(x)$ is the inverse Fourier transform 
of the characteristic function of the unit ball $\{|k|\le 1\}$. In other words, 
\[
\widehat{\rho_\epsilon * f}(k) = \chi_{\{|k|\le 1/\epsilon\}}(k)\,{\hat f}(k)\,.
\]
To save space, we write $\rho_\epsilon[f]$ instead of $\rho_\epsilon * f$,
\[
\rho_\epsilon[f](x) = \int \rho_\epsilon(x - y) f(y)\,dy\,.
\]
The mollifiers $\rho_\epsilon$ have the usual properties:
\begin{lemma}\label{mollifiers}
For any $\phi\in H^s$,   
\begin{enumerate}
\item $\rho_\epsilon[\phi]\to \phi$ in $H^s$ as $\epsilon\to 0$;
\item for any $m \ge 0$, $\|\rho_\epsilon[\phi]\|_{H^{s+m}} \le \left(1+\frac{1}{\epsilon^2}\right)^{m/2}\,\|\phi\|_{H^s}$. In particular, if $\epsilon\in (0, 1)$,
\be\label{epsilon-}
\|\rho_\epsilon[\phi]\|_{H^{s+1}} \le \frac{\sqrt{2}}{\epsilon}\,\|\phi\|_{H^s}\,,
\ee
\item for $m \ge 0$ and $\epsilon\in (0, 1)$, 
\be\label{epsilon+}
\|\rho_\epsilon[\phi] - \phi\|_{H^{s-m}}\le 2^{-m/2}\,\epsilon^m \,\|\phi\|_{H^s}\,.
\ee
\item If $\mathfrak C$ is a compact subset of $H^s$, then, for all $m\ge 0$, 
\be\label{epsilon+2}
\|\rho_\epsilon[\phi] - \phi\|_{H^{s-m}} = \epsilon^m\,o(\epsilon)
\ee 
uniformly in $\phi\in {\mathfrak C}$.
\end{enumerate}
\end{lemma}
\begin{proof} 
The first claim is well-known. The second one follows from  
\begin{align*} 
\|\rho_\epsilon[\phi]\|_{H^{s+m}}^2 ={}& \int_{k \le 1/\epsilon} (1 + |k|^2)^{s + m}\,|\hat \phi(k)|^2\,\dbar k  \\
 \le{}& \left(\frac{1 + \epsilon^2}{\epsilon^2}\right)^m\,\int_{k \le 1/\epsilon} (1 + |k|^2)^{s}\,|\hat \phi(k)|^2\,\dbar k\,,
\end{align*} 
where $\dbar k = (2\pi)^{-d} dk$. 
To prove the third and the fourth claims, observe that
\begin{align*} 
\|\rho_\epsilon[\phi] - \phi\|_{H^{s-m}}^2 ={}& \int_{k > 1/\epsilon} (1 + |k|^2)^{s - m}\,|\hat \phi(k)|^2\,\dbar k \\
={}& \int_{k > 1/\epsilon} \frac{1}{(1 + |k|^2)^m}(1 + |k|^2)^{s }\,|\hat \phi(k)|^2\,\dbar k \\ 
\le{}& 2^{-m}\,\epsilon^{2m}\,\int_{k > 1/\epsilon} (1 + |k|^2)^{s }\,|\hat \phi(k)|^2\,\dbar k\,.
\end{align*} 
By  the Kolmogorov-Riesz compactness criterion, cf. \cite{MR2734454},  
$\int_{k > 1/\epsilon} (1 + |k|^2)^{s }\,|\hat \phi(k)|^2\,\dbar k\to 0$ as $\epsilon\to 0$, uniformly in $\phi\in {\mathfrak C}$, a compact subset of $H^s$. 

\end{proof} 

\bigskip

To solve \eqref{u-11},\eqref{u-21} we construct approximate solutions $v^\epsilon, u_a^\epsilon$ by considering 
the following truncated system (cf. \cite{FMRR-1})  
\begin{align}
\partial_t v^i + {\cal P}\,\rho_\epsilon\left[v^j\, \partial_j v^i -  u^k_b\, \partial_k u^i_b  - \partial_b u^i_b\right] ={}& 0,\label{u-12}\\ 
\partial_t u^i_a + \rho_\epsilon \left[v^k\,\partial_k u^i_a - u^k_a\,\partial_k v^i -  \partial_a  v^i \right] ={}& 0 . \label{u-22}
\end{align}
We let $v^\epsilon$ and $u_a^\epsilon$ solve this system with the initial conditions
\be\label{ic-approx}
v^\epsilon(0) = \rho_\epsilon\left[v(0)\right],\quad u_a^\epsilon(0) = \rho_\epsilon\left[u_a(0)\right].
\ee
The solutions of \eqref{u-12} and \eqref{u-22} will automatically have their Fourier transforms supported in the ball $\{|k|\le 1/\epsilon\}$. That  local in time  and unique solutions exist follows from the fact that 
equations \eqref{u-12}, \eqref{u-22} can be viewed 
 as an ODE in the Hilbert space $\rho_\epsilon[H^s_\sigma]$, 
\be\label{Lip}
\frac{dV}{dt} = F(V)\,,
\ee
with locally Lipschitz right hand side. The Lipschitz property of $F(V)$ in our case is not hard to verify. To avoid clutter, we shall drop the index on $u_a$ when convenient. 
For example, the $H^s$ norm of $u\;\partial u - {\underline{u}}\;\partial {\underline{u}}$, where $u, \underline{u}\in \rho_\epsilon[H^s_\sigma]$, is estimated as follows:
\[
\begin{aligned}
& \|u\;\partial u - {\underline{u}}\;\partial {\underline{u}}\|_{H^s} \le 
\|u - \underline{u}\|_{H^s}\|\partial u\|_{H^s} + 
\|\underline{u}\|_{H^s}\|\partial (u - \underline{u})\|_{H^s} \le 
& \frac{1}{\epsilon}\,\left(\|u\|_{H^s} + \|{\underline{u}}\|_{H^s}\right)\;\|u - \underline{u}\|_{H^s}\,.
\end{aligned}
\]
We have used  that $H^s$ is an algebra when $s > d/2$ and that 
\[
\|\nabla w\|_{H^s} \le \frac{1}{\epsilon} \|w\|_{H^s}\,
\]
for any $w\in \rho_\epsilon[H^s]$. 

By the Cauchy-Picard theorem, cf. e.g., \cite[Theorem 3.1]{MB}, for every initial condition 
$(v^\epsilon(0), u^\epsilon(0))\in \rho_\epsilon[H^s_\sigma]$, there exists a $T_*(\|(v^\epsilon(0), u^\epsilon(0))\|_{\rho_\epsilon[H^s_\sigma]}, \epsilon) > 0$ and a unique solution $(v^\epsilon(t), u^\epsilon(t))$ of
the problem \eqref{u-12}, \eqref{u-22}, \eqref{ic-approx} on the time interval $(- T_*, T_*)$, such that  $(v^\epsilon, u^\epsilon)\in C([- T, T]\to \rho_\epsilon[H^s_\sigma])$ for every $0 < T < T_*$. The energy estimates discussed in the next section will show that 
$T_*$ can be chosen the same for all  $\epsilon > 0$. 

\subsection{Energy estimates}\label{subsec:energyestimates} 

For the basic $L^2$ estimate, let $(v^\epsilon, u^\epsilon)$ be a solution of \eqref{u-12}, \eqref{u-22}.  Multiply \eqref{u-12} by ${v^\epsilon}$, multiply \eqref{u-22} by $u^\epsilon_a$, sum over $a$ and integrate over $\R^d$ to obtain (after integration by parts and cancellations due to the divergence-free nature of $v^\epsilon$ and $u^\epsilon_a$):
\[
\frac{d\hfill}{dt}\,\int |v^\epsilon(t)|^2 + |u^\epsilon(t)|^2\,dx = 0\,
\]
(unless otherwise specified, the expressions such as $|u|^2$ are understood as $\sum_a u_a\cdot u_a$).  
Thus,
\be\label{L2}
\int |v^\epsilon(t)|^2 + |u^\epsilon(t)|^2\,dx = \int |v^\epsilon(0)|^2 + |u^\epsilon(0)|^2\,dx  \le \int |v(0)|^2 + |u(0)|^2\,dx \,.
\ee
The higher energy estimates are standard. We present them in a schematic form. Schematically, the system 
\eqref{u-12}, \eqref{u-22} is 
\be\label{sch-e}
\partial_t V + {\cal P}\rho_\epsilon [ V\cdot \nabla V ]= 0\,.
\ee
Let $r$ be any positive real number and note that $J^r = (1 - \Delta)^{r/2}$ commutes with 
${\cal P}$, with $\rho_\epsilon*$, and with all partial derivatives. Act with 
$J^r$ on \eqref{sch-e} and write the result in the form:
\be\label{u-r-eps}
\begin{aligned}
& \partial_t {J^r V} + {\cal P}\,\rho_\epsilon\left[V\cdot\nabla {J^r V}\right] = - {\cal P}\,\rho_\epsilon\left[J^r(V\cdot\nabla V) - V\cdot\nabla {J^r V}\right]\,.
\end{aligned}
\ee
It follows that 
\be\label{r-energy-eq}
\frac12\,\frac{d\hfill}{dt}\,\|J^r V(t)\|^2 \le  
\|J^r(V\cdot\nabla V) - V\cdot\nabla {J^r V}\|\;\|J^r V(t)\|\,.
\ee
By
the Kato-Ponce commutator 
estimate \cite{KP}, 
\be\label{kp}
\|J^r(V\cdot\nabla V) - V\cdot\nabla {J^r V}\| \lesssim \|\nabla V(t)\|_\infty\;\|J^r V(t)\|\,.
\ee
Then it follows from \eqref{r-energy-eq} that 
\be\label{pre-energy-eps}
\frac{d\hfill}{dt}\,\|V^\epsilon(t)\|_{H^r}^2 \lesssim \|\nabla V^\epsilon(t)\|_\infty\;\|V^\epsilon(t)\|_{H^r}^2\,.
\ee
This implies (with $t > 0$)
\be\label{r-energy-est}
\| V^\epsilon(t)\|^2_{H^r} \le \|V(0)\|^2_{H^r}\;
\exp\left( c(r,d)\int_0^t \|\nabla V^\epsilon(\tau)\|_\infty\,d\tau\right),
\ee
with the constant $c(r,d)$ independent of $\epsilon$.

\medskip

If we take $r = s > d/2 + 1$, then, by Sobolev embedding, $\|\nabla V^\epsilon(\tau)\|_\infty \lesssim \|V^\epsilon(\tau)\|_{H^s}$, and inequality \eqref{pre-energy-eps} implies 
\[
\frac{dy^\epsilon(t)}{dt}\le C(s, d)\, y^\epsilon(t)^{3/2}\,,
\]
for $y^\epsilon(t) = \|V^\epsilon(\tau)\|_{H^s}^2$. The $\epsilon$-independent estimate 
\[
y^\epsilon(t) \le y(0)\,\left(1 - \frac12\,C(s,d)\,{y(0)^{1/2}}\,t\right)^{-2}
\]
follows. Here we made use of the fact that with our choice of the mollifier,  we have $\|\rho_\epsilon[f]\;\big|\;H^s\|\le \|f\;\big|\;H^s\|$, and hence 
$y^\epsilon(0)\le y(0)$. 
If we take the initial conditions from the ball 
$B_R = \{\|V(0)\;\big|\;H^s\|^2 \le R^2\}$, then we can choose the common life-span 
\be
T_* = \frac{2}{C(s, d)\,R}\,.
\ee

\subsection{Passage to the limit}\label{subsec:highlimit} 

When  $\epsilon\to 0$, the approximate solutions $v^\epsilon$, $u^\epsilon$ converge to the true solutions $v$, $u$ of \eqref{u-11}, \eqref{u-21}. 
Pick a $T < T_*$. All approximate solutions $V^\epsilon(t)$ exist on the interval $[-T, T]$ and by \eqref{r-energy-est}  have uniformly bounded $H^s$ norms. Then there is a sequence $\epsilon_n\searrow 0$ such that 
$V^{\epsilon_n}$ converge weak-$*$ in $L^\infty([-T, T]\to H^s_\sigma)$ to some $V = (v, u)$ in that space. From the equations \eqref{u-12} and \eqref{u-22} we see that $(\partial_t v^\epsilon(t), \partial_t u^\epsilon(t))$ are uniformly bounded in $H^{s-1}_\sigma$. Therefore, we can arrange that 
$\partial_tV^{\epsilon_n}$ converge weak-$*$ in $L^\infty([-T, T]\to H^{s-1}_\sigma)$ to $\partial_tV = (\partial_tv, \partial_tu)$. In addition, taking into consideration Rellich's compact embedding theorem and the fact that $s$ is sufficiently large, 
we can assume that $V^{\epsilon_n}$ converges strongly to $V$ 
and each first partial derivative $\partial V^{\epsilon_n}$ converges strongly to $\partial V$ in every $L^2([-T, T]\times B_N(0))$ 
(where $B_N(0) = \{x\in\R^d: |x|\le N\}$) and $V^{\epsilon_n}$ and $\partial V^{\epsilon_n}$ converge almost everywhere in the slab $[-T, T]\times \R^d$. We also note that $V^{\epsilon}(t, x)$ and 
$\partial V^{\epsilon}(t, x)$ are uniformly bounded on $[-T, T]\times \R^d$.

The just mentioned facts are sufficient to conclude that $(v, u)$ solves  equations \eqref{u-11}, \eqref{u-21}
in the sense of distributions and assume the prescribed values at $t = 0$. To see this, let $\eta(t, x)$ be any smooth divergence free vector with compact support 
in $[-T, T]\times\R^d$. Multiply equations \eqref{u-12} and \eqref{u-22} by $\eta$ and integrate over $[0, T]\times \R^d$. After a few rearrangements, we get 
\be\label{uno-5}
\begin{aligned}
& \int_0^T\int - v^\epsilon\,\partial_t\eta\, + \left( v^{\epsilon j} \partial_j v^{\epsilon i} -  u^{\epsilon k}_a\; \partial_k u^{\epsilon i}_b  
- \partial_a u^{\epsilon i}_b \right)\,\rho_\epsilon\left[\eta^i\right]\,dt\,dx  = \int v^\epsilon(0)\,\eta(0)\,dx
\end{aligned}
\ee
and 
\be\label{dos-5}
\begin{aligned}
& \int_0^T\int - u^{\epsilon }_a\,\partial_t \eta + \left[v^{\epsilon k}\partial_k u^{\epsilon i}_a - u^{\epsilon k}_a\partial_k v^{\epsilon i} - \partial_a v^{\epsilon i}\right]\,\rho_\epsilon[\eta^i]\,dt\,dx  = \int u^\epsilon_a(0)\,\eta(0)\,dx
\end{aligned}
\ee
Since the terms such as $v^{\epsilon j} \partial_j v^{\epsilon i}$ and $u_a^{\epsilon k}\;\partial_k u^{\epsilon i}_a$ are uniformly in $\epsilon$ bounded in $L^2([-T, T]\times \R^d)$, expressions such as
\[
\int_0^T\int  u^{\epsilon k}_a\,\partial_k u^{\epsilon i}_b\;\left(\rho_\epsilon[\eta^i] - \eta^i\right)\,dt\,dx
\]
go to zero as $\epsilon = \epsilon_n\to 0$. At the same time, 
\[
\int_0^T\int  u^{\epsilon k}_a\,\partial_k u^{\epsilon i}_b\;\eta^i \,dt\,dx \to 
\int_0^T\int  u^{k}_a\,\partial_k u^{ i}_b\;\eta^i \,dt\,dx .
\]
Thus, the limit functions $v$ and $u$ satisfy 
the integral identities
\be\label{uno-6}
\begin{aligned}
& \int_0^T\int - v\,\partial_t\eta\, + \left( v^{ j} \partial_j v^{ i} -  u^{ k}_a\; \partial_k u^{ i}_b  
- \partial_a u^{ i}_b \right)\,\eta^i\,dt\,dx  = \int v(0)\,\eta(0)\,dx
\end{aligned}
\ee
and, for $a = 1, \dots, d$, 
\be\label{dos-6}
\begin{aligned}
& \int_0^T\int - u_a\,\partial_t \eta + \left[v^{ k}\partial_k u^{ i}_a - u^{ k}_a\partial_k v^{ i} - \partial_a v^{ i}\right]\,\eta^i\,dt\,dx  = \int u_a(0)\,\eta(0)\,dx
\end{aligned}
\ee
with any divergence free smooth $\eta$ with compact support in $(-T, T)\times \R^d$. This shows, in particular, that 
$(v, u)$ solve equations \eqref{u-11}, \eqref{u-21} in the sense of distributions.

So far we have a solution $V$ with $V\in L^\infty([-T, T]\to H^s_\sigma)$ and $\partial_tV\in L^\infty([-T, T]\to H^{s-1}_\sigma)$. This already implies that $V: [-T, T]\to H^s_\sigma$ is weakly continuous. Let us now prove the uniqueness of the obtained solution. Assume there is another solution, ${\underline V} = (\underline{v}, \underline{u})\in L^\infty([-T, T]\to H^s_\sigma)$ with the time derivative in 
$L^\infty([-T, T]\to H^{s-1}_\sigma)$, as seen from \eqref{u-11},\eqref{u-21} . The equation satisfied by the difference, ${\tilde V} = V - \underline{V}$, is  schematically 
\[
\partial_t {\tilde V} + {\cal P} V\cdot \nabla {\tilde V} = - {\cal P} {\tilde V}\cdot\nabla {\underline V}
\]
The same argument as for the energy estimate gives
\[
\frac12 \frac{d\hfill}{dt}\|{\tilde V}(t)\|^2  \lesssim \|\nabla {\underline V}(t)\|_\infty\,\|{\tilde V}(t)\|^2\,.
\]
As long as  $\int_0^t\|\nabla {\underline V}(\tau)\|_\infty\,d\tau$ is finite on $[0, T]$, and since ${\tilde V}(0) = 0$, 
it follows that ${\tilde V}(t) = 0$ on $[0, T]$.
\medskip

Now we shall prove that the solution $(v(t), u(t))$ is strongly continuous in $t$. 
Recall the inequality \eqref{r-energy-est}. We need it with $r=s$. When $s > d/2 + 1$, we have 
\[
\int_0^t \|\nabla V^\epsilon(t^\prime)\|_\infty \,dt^\prime \lesssim M\cdot t\,,
\]
where $M$ is a uniform in $\epsilon$ bound on the  $L^\infty([-T_*, T_*]\to H^s)$ norms of $V$. Thus, 
it follows from  \eqref{r-energy-est} that 
 for all $\epsilon = \epsilon_n$, and $t > 0$ 
\[
\| v^\epsilon(t)\|^2_{H^s} + \|u^\epsilon(t)\|^2_{H^s}\le \left(\|v(0)\|^2_{H^s} +  \|u(0)\|^2_{H^s}\right)\;
\exp\left( C\,M\cdot t\right)\,.
\]
Use the lower semi-continuity of the norms on the left to obtain
\[
\| v(t)\|^2_{H^s} + \|u(t)\|^2_{H^s}\le \left(\|v(0)\|^2_{H^s} +  \|u(0)\|^2_{H^s}\right)\;
\exp\left( C\,M\cdot t\right)\,.
\]
Since the solution $(v(t), u(t))$ is weakly continuous in $H^s$, it follows that it is strongly continuous from the right at $t = 0$. For any $t_0\in [-T, T]$, if we solve \eqref{u-11}, \eqref{u-21} 
with the initial condition $(v(t_0), u(t_0))$ at $t=t_0$, we obtain the same solution $(v(t), u(t))$ due to uniqueness. Consequently, $(v(t), u(t))$ is strongly continuous in $H^s$ from the right at every $t\in [-T, T]$. The equations are invariant under the time reversal transformation $(t, v(t), u(t)) \mapsto (-t, -v(-t), -u(-t)) $. Again, due to uniqueness, 
this implies strong continuity of $t\mapsto (v(t), u(t))\in H^s$ from the left.  This proves the strong continuity in time of $V(t)\in H^s_\sigma$, i.e.,  
$V\in C([0, T]\to H^s)$.  

\subsection{Revisiting energy estimates}

The limit case of the inequality \eqref{r-energy-est} is true, viz. for any solution satisfying the conditions of Theorem \ref{Thm-1}, 
\be\label{r-energy}
\|v(t)\|_{H^r}^2 + \|u(t)\|_{H^r}^2 \le \left( \|v(t_0)\|_{H^r}^2 + \|u(t_0)\|_{H^r}^2\right)\,
\exp \left(c(r, d)\,\int_{t_0}^t \|\nabla V(\tau)\|_\infty\,d\tau\right)
\ee 
for all $-T\le t_0 < t\le T$ and for any $0 < r \le s$. This is not proved by passing to the limit $\epsilon\to 0$ in \eqref{r-energy-est} but rather by appealing to the fact (cf.  \cite[p. 64]{BKM}) that 
$(J^r v, J^r u)$ is the unique $C([-T, T]\to L^2)$ solution of the linear non-homogeneous hyperbolic system
\be\label{u-rr}
\begin{aligned}
& \partial_t {{\mathfrak v}} + {\cal P}\,\left[v^j\, \partial_j {\mathfrak v} -  u^k_b\, \partial_k {\mathfrak u}_b  - \partial_b {\mathfrak u}_b\right] = F^r\,,\\ 
& \partial_t {\mathfrak u}_a + \left[v^k\,\partial_k {\mathfrak u}_a - u^k_a\,\partial_k {\mathfrak v} -  \partial_a {\mathfrak v}\right] = F^r_a\,, 
\end{aligned}
\ee
where
\be\label{F-r}
\begin{aligned}
F^r & = {\cal P}\,\left[\left(v^j\, \partial_j {J^r v} - J^r(v^j\,\partial_j v)\right) - 
\left(u^k_b\, \partial_k {J^r u}_b - J^r(u^k_b\;\partial_k u_b)\right)\right] \\ 
F^r_a & = {\cal P}\,\left[\left(v^j\, \partial_j {J^r u_a} - J^r(v^j\,\partial_j u_a)\right) - 
\left(u^k_a\, \partial_k {J^r v} - J^r(u^k_a\;\partial_k v)\right)\right]\,,
\end{aligned}
\ee
and
\[
{\mathfrak v}(0) = J^r v(0),\quad {\mathfrak u}_b(0) = J^r u_b(0)\,.
\]
The energy estimate for the solution can be written in the form (we take $t_0 = 0$ for simplicity)
\[
\left(\|{\mathfrak v}(t)\|^2 + \|{\mathfrak u}_a(t)\|^2\right)^{1/2} \le \left(\|{\mathfrak v}(0)\|^2 + \|{\mathfrak u}_a(0)\|^2\right)^{1/2} + \int_0^t \left(\|F^r(\tau)\|^2 + \|F^r_a(\tau)\|^2\right)^{1/2}\,d\tau\,.
\] 
Substitute the true values ${\mathfrak v} = J^r v$ and ${\mathfrak u}_a = J^r u_a$,  apply estimates \eqref{kp} to the norms in the integrand, and use Gronwall's inequality to arrive at \eqref{r-energy}.

\subsection{Continuous dependence}\label{sec:contdep}

The proof of continuous dependence of solutions on the initial conditions will follow the strategy 
of Bona and Smith \cite{BS} as has been done by Kato and Lai \cite{Kato-Lai} for the Euler equations.  
Take a sequence of initial conditions $V_n(0)\in H^s_\sigma$ converging (in $H^s$) to $V(0)$. 
Let $V_n(t)$ and $V(t)$ be the corresponding solutions
of system \eqref{u-11}, \eqref{u-21}. 
We can choose the same interval $[0, T]$ for all $V_n(0)$. The goal is to show that the solutions $V_n(t)$ converge to $V(t)$ in $C([0, T]\to H^{s}_\sigma$. 

Mollify the initial conditions, $V_n^\epsilon(0) = \rho_\epsilon[V_n(0)]$ and $V^\epsilon(0) = \rho_\epsilon[V(0)]$, using the same mollifier as in Lemma~\ref{mollifiers}. Let $V_n^\epsilon$ and $V^\epsilon$ be the corresponding solutions of  \eqref{u-11}, \eqref{u-21} lying in the space 
$C([0, T]\to H^{s+1}_\sigma)$. As the energy estimates show, we can adjust $T$ so that the solutions exist on  the same time interval, $[0, T]$ for all $\epsilon > 0$. We have 
\be\label{cd-scheme}
\begin{aligned}
\sup_{[0, T]}\|V_n(t) - V(t)\|_{H^s} \le{}& 
\sup_{[0, T]}\|V_n(t) - {V}^\epsilon_n(t)\|_{H^s} 
+ \sup_{[0, T]}\|V(t) - {V}^\epsilon(t)\|_{H^s} \\ 
&\quad + \sup_{[0, T]}\|{V}^\epsilon_n(t) - {V}^\epsilon(t)\|_{H^s}\,.
\end{aligned}
\ee
As will be shown, the first two norms on the right will be small by the choice of $\epsilon$, while the last norm, for a fixed $\epsilon$, will be small for large $n$.

We start with the first two norms on the right side of \eqref{cd-scheme} and prove a slightly more general result. Consider the system \eqref{u-11}, \eqref{u-21} with initial conditions 
$V(0) = (v(0), u(0))$ from a compact subset ${\mathfrak C}$ of $H^s_\sigma$, where $s > d/2+1$. 
Mollify the initial conditions, $V^\epsilon(0) = (\rho_\epsilon[v(0)], \rho_\epsilon[u(0)])$, and let $V^\epsilon(t)$ be 
the corresponding solutions on $[0, T]$. The interval can be chosen uniformly for $V(0)\in {\mathfrak C}$ and $\epsilon\in(0, 1)$ as follows from the existence part of Theorem  \ref{Thm-1}. 
\begin{proposition}\label{prop:A}   
For $V(0)\in {\mathfrak C}$ and $\epsilon\to 0$, 
the family $V^\epsilon$ is uniformly Cauchy in $C([0,T]\to H^s_\sigma)$.
\end{proposition}
\begin{proof} 
 Consider the difference $\tilde V = V^\delta - V^\epsilon$, where $0 < \delta < \epsilon$. At $t=0$, the $H^s$ norm of 
 $\tV(0)$ goes to $0$ as $\epsilon\searrow 0$. 
 The difference $\tV(t)$ satisfies, schematically, the equation
\be\label{toy}
\partial_t {\tilde V} + {\cal P}\,V^\delta\nabla {\tilde V} = - {\cal P}\,{\tilde V} \nabla V^\epsilon\,.
\ee
First get the $L^2$ estimate. We have
\[
\frac12\,\|\tV(t)\|^2 \le \|\nabla V^\epsilon(t)\|_\infty\,\|\tV(t)\|^2\,,
\]
and therefore,
\be\label{est:L2}
\|\tV(t)\|^2 \le \|\tV(0)\|^2\,\exp\int_0^t \|\nabla V^\epsilon(\tau)\|_\infty\,d\tau\,.
\ee
This implies
\be\label{est:L2-1}
\sup_{[0, T]} \|\tV(t)\| \lesssim \epsilon^s\,o(\epsilon)
\ee
since the integral of $\|\nabla V^\epsilon(\tau)\|_\infty$ is bounded and, since $V(0)\in H^s$, and therefore,  
\begin{align*} 
\|\tV(0)\|^2 ={}& \|V^\delta(0) - V^\epsilon(0)\|^2 \\
={}& \int_{1/\epsilon <|\kappa|<1/\delta}|\hat V(0,\kappa)|^2\,\dbar\kappa \\
={}& \int_{1/\epsilon <|\kappa|<1/\delta}\frac{|\widehat {J^sV(0)}(\kappa)|^2}{(1 + |\kappa|^2)^s}\,\dbar\kappa \\ 
\le{}& \left(\frac{\epsilon^2}{1 + \epsilon^2}\right)^s\,\int_{1/\epsilon <|\kappa|<1/\delta}|\widehat {J^sV(0)}(\kappa)|^2
\,\dbar\kappa \,.
\end{align*} 
Now, turn to the $H^s$ estimates. 
 Act with $J^s$ on equation \eqref{toy}:
\[
\begin{aligned}
\partial_t J^s{\tilde V} + {\cal P}\,V^\delta\nabla J^s{\tilde V} ={}&  {\cal P}\,\left(V^\delta\nabla J^s{\tilde V}
 - J^s(V^\delta\nabla {\tilde V})\right)  \\ 
& \quad - {\cal P}\,\left( J^s\left({\tilde V} \nabla V^\epsilon\right) - {\tilde V} \nabla J^s V^\epsilon \right) - {\cal P}\,{\tilde V} \nabla J^s V^\epsilon\,.
\end{aligned}
\]
As in the derivation of energy estimates, we obtain
\be\label{tilde-s}
\begin{aligned}
\frac{d \hfill}{dt}\; \|J^s{\tilde V}(t)\| \le{}&  \|V^\delta\nabla J^s{\tilde V} - J^s(V^\delta\nabla {\tilde V})\| 
\\ 
& \quad + \| J^s\left({\tilde V} \nabla V^\epsilon\right) - {\tilde V} \nabla J^s V^\epsilon\| + \|\tV\|_\infty\,\|\nabla J^s V^\epsilon\|
\end{aligned}
\ee
By the  Kato-Ponce inequality \eqref{kp}, 
\be\label{I-s-0}
\|V^\delta\nabla J^s{\tilde V} - J^s(V^\delta\nabla {\tilde V})\|\lesssim \|\nabla V^\delta(t)\|_\infty\,\|{\tilde V}(t)\|_{H^s} + \|\nabla {\tilde V}(t)\|_\infty\,\|V^\delta(t)\|_{H^s}\,,
\ee
and
\be\label{II-s-0}
\| J^s\left({\tilde V} \nabla V^\epsilon\right) - {\tilde V} \nabla J^s V^\epsilon\| \lesssim \|\nabla {\tilde V}(t)\|_\infty\,\|V^\epsilon(t)\|_{H^s} + \|\nabla V^\epsilon(t)\|_\infty\,\|{\tilde V}(t)\|_{H^s}\,.
\ee 
Thus, 
\be\label{tilde-s2}
\begin{split}
\frac{d\hfil}{dt} \|{\tV}(t)\|_{H^s}\lesssim{}&  (\|\nabla V^\delta(t)\|_\infty + \|\nabla V^\epsilon(t)\|_\infty)\,\|{\tV}(t)\|_{H^s} \\ 
& \quad + (\|V^\delta(t)\|_{H^s} + \|V^\epsilon(t)\|_{H^s})\,\|\nabla\tV(t)\|_\infty + 
\|\tV(t)\|_\infty\,\|V^\epsilon(t)\|_{H^{s+1}}\,.
\end{split}
\ee
The $L^\infty_tH^s$ norms of $V^\epsilon$ and $V^\delta$ are uniformly bounded. Since $s > d/2+1$, the functions 
$\|\nabla V^\epsilon(t)\|_\infty$ and $\|\nabla V^\delta(t)\|_\infty$ are (uniformly) bounded on $[0, T]$,  
and, also, $\|\nabla\tV(t)\|_\infty \lesssim \|\tV(t)\|_{H^s}$ . 
The energy estimate \eqref{r-energy} with $r = s+1$ tells us that 
\[
\|V^\epsilon(t)\|_{H^{s+1}} \le C\,\|\rho_\epsilon[V(0)]\|_{H^{s+1}} \le C\,\frac{1}{\epsilon}\,\|V(0)\|_{H^s}\,.
\]
By the Gagliardo-Nirenberg inequality,
\[
\|\tV(t)\|_\infty \lesssim \|\tV(t)\|^{1 - \frac{d}{2s}}\,\|\tV\|_{H^s}^{\frac{d}{2s}}\,.
\]
Taking into account \eqref{est:L2-1}, we obtain 
\[
\|\tV(t)\|_\infty \lesssim \epsilon^{s - \frac{d}{2}}\,{o(\epsilon)}\,\|\tV\|_{H^s}^{\frac{d}{2s}}\,.
\]
Hence, 
\[
\|\tV(t)\|_\infty\,\|V^\epsilon(t)\|_{H^{s+1}} \lesssim \epsilon^{s - 1 - \frac{d}{2}}\,{o(\epsilon)}\,\|\tV\|_{H^s}^{\frac{d}{2s}}\,\|V(0)\|_{H^s} \lesssim \epsilon^{s - 1 - \frac{d}{2}}\,.
\]
Combining all this, 
\[
\frac{d\hfil}{dt} \|{\tV}(t)\|_{H^s}\lesssim \|{\tV}(t)\|_{H^s} + \epsilon^{s - 1 - \frac{d}{2}}\,.
\]
Integrating this inequality and using the fact that $\|\tV(0)\|_{H^s}\to 0$ uniformly for $V(0)\in {\mathfrak C}$, we see that (again, uniformly for $V(0)\in {\mathfrak C}$) the supremum of $ \|{\tV}(t)\|_{H^s}$ goes to $0$ as $\epsilon\searrow 0$. 
\end{proof}
\begin{remark} \label{rem:KL} In the proof above, the term $\|\tV(t)\|_\infty\,\|V^\epsilon(t)\|_{H^{s+1}}$ occurring in equation \eqref{tilde-s2} is treated differently from \cite{Kato-Lai}. In the lower regularity case, where the argument of \cite{Kato-Lai} does not apply, this term will be treated by a generalization of the Gagliardo-Nirenberg inequality. 
\end{remark} 

Proposition \ref{prop:A} takes care of the first two terms of the right hand side of \eqref{cd-scheme}. 
To handle the last term, $\sup_t \|V_n^\epsilon(t) - V^\epsilon(t)\|_{H^s}$, for a fixed small $\epsilon$, we treat the difference 
$\tV = V_n^\epsilon - V^\epsilon$ as in the proof of the proposition and arrive at the inequality 
\be\label{tilde-s3}
\begin{split}
 \frac{d\hfil}{dt} \|{\tV}(t)\|_{H^s}\lesssim{}& (\|\nabla V_n^\epsilon(t)\|_\infty + \|\nabla V^\epsilon(t)\|_\infty)\,\|{\tV}(t)\|_{H^s} \\ 
& \quad + (\|V_n^\epsilon(t)\|_{H^s} + \|V^\epsilon(t)\|_{H^s})\,\|\nabla\tV(t)\|_\infty + 
\|\tV(t)\|_\infty\,\|V^\epsilon(t)\|_{H^{s+1}}\,.
\end{split}
\ee
Now $\|\nabla\tV(t)\|_\infty,\,\|\tV(t)\|_\infty \lesssim \|\tV(t)\|_{H^s}$ and the remaining factors are uniformly bounded in $t$. Hence, $\sup_t \|\tV(t)\|_{H^s} \lesssim \|\tV(0)\|_{H^s}\to 0$ as $n\to\infty$. 
\bigskip

This completes the proof of Theorem \ref{Thm-1}. 
\bigskip

\begin{remark}\label{pressure-E} It is standard that once the solution $(v, u)$ of the ``projected" system \eqref{u-11}, \eqref{u-21} is obtained, with 
$(v, u)\in C([0, T]\to H^s_\sigma)$ and $(\partial_t v, \partial_t u)\in C([0, T]\to H^{s-1}_\sigma)$, 
we can recover the pressure $p(t, x)$. From equations \eqref{u-1}, it follows that $\nabla p\in C([0, T]\to H^{s-1})$. Also, $- \Delta p = \partial_k v^i\cdot \partial_i v^k - \partial_k u_a^i\cdot \partial_i u_a^k$,  from \eqref{u-1}. Using the Riesz transforms ${\cal R}_j$, we can solve for 
$p = {\cal R}_i {\cal R}_k \left(v^i\,v^k - u^i_a\ u^k_a\right)$. 
As long as $s > d/2$,  we see that 
$p\in C([0, T]\to H^s)$. 

From the identities \eqref{uno-6}, \eqref{dos-6} we can now deduce the following identities
\be\label{uno-7}
\begin{aligned}
& \int_0^T\int - v^i\,\left(\partial_t\underline{\eta}^i + v^j \partial_j \underline{\eta}^i\right) +  v^{ i}_a\, v^{ k}_a  
\partial_k\underline{\eta}^i - p\,\partial_i \underline{\eta}^i\,dt\,dx  = \int v(0)\,\underline{\eta}(0)\,dx
\end{aligned}
\ee
and, for $a = 1, \dots, d$, 
\be\label{dos-7}
\begin{aligned}
& \int_0^T\int - v_a^i\,\left(\partial_t \underline{\eta}^i + v^{ k} \partial_k \underline{\eta}^i\right) + v^{ i}\,v^{ k}_a\partial_k \underline{\eta}^i\,dt\,dx  = \int v_a(0)\,\underline{\eta}(0)\,dx
\end{aligned}
\ee
valid for any smooth vector function $\underline{\eta}$ with compact support in $(-T, T)\times \R^d$. 
We have rearranged the terms (compared to \eqref{uno-6}, \eqref{dos-6}) for  future use.
\end{remark}

\section{Lagrangean picture} \label{sec:backtolagrange}

This section is located  between the sections on high regularity and lower regularity Cauchy problem 
in the Euler setting. We treat the Lagrange equations by switching to the Euler form, applying the results for the Euler setting, and switching back to Lagrange. Thus, in this section, our results will be 
conditioned on the available results for the Eulerian system. At the moment, we have results only in the case $s > d/2 + 1$, and we can apply them right away. After we show, in the next section, 
how to work with $s < d/2 + 1$, we shall immediately translate that into the Lagrangean setting 
by applying the results of the current section.

\subsection{A volume preserving diffeomorphism lemma} 

We start with a useful lemma. 
\begin{lemma}\label{difff} 
Let $A$ be a $d\times d$ matrix with real (constant) entries and $\det A = 1$.
Consider the map $\Phi:\,\xi\in\R^d_\xi \to x = A\xi + \varphi(\xi)\in \R^d_x$ such that 
\begin{enumerate} 
\item $\varphi \in H^{s+1}(\R^d_\xi)$ with $s > \frac{d}{2}$; 
\item $\det\,\left(A + \frac{\partial \varphi(\xi)}{\partial\xi}\right) = 1$\,.
\end{enumerate}
Then $\Phi$ is a $C^1$ diffeomorphism and its inverse, $\Phi^{-1}$, can be written as $\xi = A^{-1}x - g(x)$ where $g\in H^{s+1}(\R^d_x)$. 
\end{lemma}

\begin{remark}
Notice that the regularity restriction in this lemma corresponds to $V\in H^s$ with $s > d/2$.
\end{remark}
\begin{proof} 
By the Sobolev embedding,
the entries of the matrix $\partial \varphi(\xi)/\partial \xi$ are bounded, and so $\Phi\in C^1$. 
Since $\det\,\left(\partial\Phi /\partial\xi\right) = 1$, $\Phi$ is a local diffeomorphism. 
It is a global volume preserving $C^1$-diffeomorphism of $\R^d$ by  
Hadamard's global inverse function theorem  (see \cite{Gordon} for references and proofs). 
In particular,
\[
\int_{\R^d} f(\xi)\,d\xi = \int_{\R^d} f(\Phi^{-1}(x))\,dx
\]
for any reasonable $f$.  
Since $\Phi$ is volume preserving, 
\[
\left(A + \frac{\partial\varphi(\xi)}{\partial \xi}\right)^{-1} = A^{-1} + F_A(\frac{\partial\varphi(\xi)}{\partial \xi})\,,
\]
where the entries of $F_A$ are polynomials in $\partial\varphi/\partial\xi$ of degree $d-1$ without constant terms.
It follows (Moser's Lemma \ref{chain-2}) that $F_A(\partial\varphi/\partial\xi)$ is in $H^s(\R^d_\xi)$. Denote by 
$\psi(\xi)$ the matrix-function 
$F_A(\partial\varphi/\partial\xi)$ viewed as a function of $\xi$. We have  
\[
\frac{\partial\hfil}{\partial x} = \left(A + \frac{\partial\varphi(\xi)}{\partial \xi}\right)^{-1}\;\frac{\partial\hfil}{\partial \xi}\,=\, \left(A^{-1} + F_A(\frac{\partial\varphi(\xi)}{\partial \xi})\right)\;\frac{\partial\hfil}{\partial \xi}\,= \,(A^{-1} + \psi(\xi))\;\frac{\partial\hfil}{\partial \xi}. 
\]
Define $g(x) = A^{-1} x - \xi$. Then 
\[
\frac{\partial g(x)}{\partial x} = A^{-1} - \frac{\partial\xi}{\partial x} = - \psi(\xi)\,.
\]
Note that $\psi(\xi)$ is a bounded function. 
Calculating the higher $x$-derivatives of $g$ using that $\partial/\partial x = (A^{-1} + \psi(\xi))\,\partial/\partial \xi$, 
we see that, for any integer $r$, the derivatives $\partial^{r+1} g(x)/\partial x^{r+1}$ are sums 
of constant multiples of the terms of the form
\be\label{prod-psi}
(\psi)^{\alpha_0}\,(\frac{\partial \psi}{\partial\xi})^{\alpha_1}\,(\frac{\partial^2 \psi}{\partial\xi^2})^{\alpha_2}\cdots (\frac{\partial^m \psi}{\partial\xi^m})^{\alpha_m}\,,
\ee
where the $\alpha$'s are non-negative integers and 
\be
\alpha_1 + 2\,\alpha_2 + \cdots + m\,\alpha_m = r\,.
\ee
Let $r$ be the largest integer less than or equal to $s$. 
The $L^2(\R^d, d\xi)$ norm of each product \eqref{prod-psi} is bounded by the $H^r(\R^d_\xi)$ norm of $\psi$ 
as follows from Moser's argument. Hence, all $x$-derivatives of $g(x)$ of order $r+1$ are in $L^2(dx)$.
Assume now $s = r + \gamma$, where  $\gamma\in(0, 1)$. The $H^s$ norm of a function is equivalent to the $H^r$ norm plus 
the sum of the $H^\gamma$ norms of all partial derivatives of order $r$. This means we have to estimate the $H^\gamma$ norm of the product \eqref{prod-psi}. Note that, in general, if $f(x)$ and ${\tilde f}(\xi)$ are related by the equation ${\tilde f}(\xi) = f(x( \xi))$ and the transformation $x\to \xi$ is bi-Lipschitz (as in our case), the $H^\gamma(\R^d_\xi)$ norm of ${\tilde f}$ and 
the $H^\gamma(\R^d_x)$ norm of $f$ are equivalent if $0\le \gamma \le 1$. 
Thus, we will estimate the $H^\theta(\R^d_\xi)$ norms of the expressions \eqref{prod-psi}. But this is done in a slightly more general case in the proof of Lemma \ref{chain-2}. 
This completes the proof of the regularity of $g(x)$ claimed in the statement of the theorem.  

\end{proof}

\subsection{Cauchy problem in the Lagrange setting}

Recall that the basic equations in Lagrangean form are
\begin{align}
& \frac{\partial^2 x^i}{\partial t^2} -  \frac{\partial^2 x^j}{\partial\xi^a \partial \xi^a} + \frac{\partial p}{\partial x^i} = 0\,,\quad i = 1,\dots, d,\label{Lag-1} \\ 
& \det \frac{\partial x}{\partial\xi}\,=\,1 \label{Lag-2}
\end{align}
Given a constant matrix $A\in SL(d, \R)$, consider the Cauchy problem for this system with the initial conditions 
\be\label{Lag-ic}
x(0, \xi) = A\xi + \varphi^L(\xi)\,,\quad \frac{\partial x(0, \xi)}{\partial t} = v_0^L( \xi)\,.
\ee
We assume that $x(0, \xi)$ satisfies \eqref{Lag-2} and that $v_0^L$ is divergence free in the $x$ variables,
\be\label{div0}
\frac{\partial (v_0^L)^i(\xi)}{\partial x^i} = \frac{\partial (v^L)^i(\xi)}{\partial \xi^a}\;\frac{\partial \xi^a}{\partial x^i} = \hbox{Tr}\,\left(\frac{\partial v_0^L(\xi)}{\partial \xi}\,(A + \frac{\partial \varphi^L}{\partial\xi})^{-1}\right) = 0\,.
\ee 
Define the quantities 
\begin{equation}\label{eq:udef}
u^i_a(0, x) = \frac{\partial (\varphi^L)^i(\xi)}{\partial\xi^a}\,\big|_{\xi = \xi(0, x)}\,. 
\end{equation}
Then by Piola's identities, 
\[
\hbox{div}\,u_a(0, x) = \frac{\partial u^i_a(0, x)}{\partial x^i} = 0\quad \text{ for each $a=1,\dots, d$}\,. 
\]
Define $v(0, x) = v_0^L(\xi)$. 
\bigskip 

We assume that $\varphi^L\in H^{s+1}_\xi$ and $v_0^L \in H^s_\xi$ for some $s > d/2$. By Lemma \ref{difff},  
$u_a(0,\cdot)$ belongs to $H^s(\R^d_x)$.  Also,  $v(0, \cdot)\in H^s(\R^d_x)$.
\bigskip

\noindent{\bf Assumption I.} Assume that the system \eqref{A-u-1}, \eqref{A-u-2}, \eqref{A-u-3}, with the initial conditions $(v(0), u(0))$,  has a unique solution $(v, u)\in C([0, T]\to H^s_\sigma(\R^d_x))$ with $(\partial_t v, \partial_t u)\in C([0, T]\to H^{s-1}_\sigma(\R^d_x))$. Also, assume that for $v$ and $v_a = A_a + u_a$, the integral identities \eqref{uno-7} and \eqref{dos-7} are valid. In addition, assume that
\be\label{noblowup-2}
\int_0^T \|\nabla v(t)\|_\infty  \,dt < \infty\,.
\ee

\begin{theorem}\label{Lag-Cauchy} 
Under Assumption I, the problem \eqref{Lag-1}, \eqref{Lag-2}, \eqref{Lag-ic} has a unique solution 
$x(t, \xi), p^L(t, \xi)$ on the same time interval $[0, T]$, such that 
\[
x - A\xi \in C([0, T]\to H^{s+1}(\R^d_\xi)),\quad  \frac{\partial x}{\partial t}\in C([0, T]\to H^{s}(\R^d_\xi))\,,
\]
and 
\[
p^L\in C([0, T]\to H^{s}(\R^d_\xi))\,.
\]
The family of maps $\xi \mapsto x(t, \xi)$ is a family of volume preserving $H^{s+1}$-diffeomorphisms. The inverse maps, $x\mapsto \xi(t, x)$, have the following regularity:
\[
\xi - A^{-1}x\in C([0, T]\to H^{s+1}(\R^d_x)),\quad  \frac{\partial \xi}{\partial t}\in C([0, T]\to H^{s}(\R^d_x))\,.
\]
If the solutions $(v, u)$ of the Eulerian system \eqref{A-u-1}, \eqref{A-u-2}, \eqref{A-u-3} depend continuously in $C([0, T]\to H^{s}(\R^d_x))$ on the initial conditions, then the solutions 
$x(t,\xi), p^L(t, \xi)$ of \eqref{Lag-1}, \eqref{Lag-2}, \eqref{Lag-ic} depend continuously 
on the initial conditions $\varphi^L\in H^{s+1}$ and $v_0^L\in H^s$. 
\end{theorem}

\begin{proof}  Create the Eulerian initial data $(v(0), u(0))$ as described above and solve 
 the problem \eqref{u-11}, \eqref{u-21}. Thanks to  Assumption I there is a solution $(v, u)\in C([0, T]\to H^s_\sigma(\R^d_x))$ with $(\partial_t v, \partial_t u)\in C([0, T]\to H^{s-1}_\sigma(\R^d_x))$. 
Note that for $s > d/2$
\be\label{uv-in-C}
v, u \in C([0, T]\times \R^d)\,;
\ee
in particular, $v$ and $u$ are bounded in the slab $[0, T]\times \R^d_x$. 
Solve the ODE 
\be\label{ode-0}
\dot x = v(t, x)\,,\quad x(0) = A\xi + \varphi^L(\xi)\,.
\ee
Thanks to \eqref{noblowup-2}, there exists a unique solution $x(t, \xi)$ on $[0, T]$ such that, as  functions of $t$ with the values in $\R^d$,  
$x\in C([0, T]\to C^1(\R^d_\xi))$ and $\dot x\in C([0, T]\to C(\R^d_\xi))$.
Clearly, $\det \partial x/\partial \xi = 1$. Thus, $\xi\to x(t, \xi)$ is a local $C^1$-diffeomorphism. 
Since 
\be\label{x-v}
x(t, \xi) =  A\xi + \varphi^L(\xi) + \int_0^t v(s, x(s, \xi))\,ds\,
\ee
and $v$ is bounded, $|x(t, \xi)| \to \infty$ as $|\xi|\to \infty$. 
By Hadamard's global inverse function theorem, $x(t, \cdot): \R^d\to \R^d$ is a global $C^1$ diffeomorphism.  
The derivatives $\partial x/\partial \xi$ are solutions of the following system of ODEs  
\be\label{ode-1}
\frac{\partial\hfill}{\partial t} \frac{\partial x^i}{\partial \xi^a} =\partial_k v^i(t, x)\,\frac{\partial x^k}{\partial \xi^a}
\ee
with 
\[
\frac{\partial x^i(0, \xi)}{\partial \xi^a} = A^i_a + \frac{\partial(\varphi^L)^i(\xi)}{\partial\xi^a}\,.
\]
It follows that 
\be\label{u-infty}
|\frac{\partial x(t, \xi)}{\partial \xi}|\le C\,\exp \int_0^t \|\nabla v(\tau)\|_\infty\,d\tau
\ee
and
\be\label{x-x}
|x(t, \xi) - x(t, \eta)| \le C\,|\xi - \eta|\,\exp \int_0^t \|\nabla v(\tau)\|_\infty\,d\tau\,.
\ee
The inverse map, $\xi(t, x)$ satisfies the equation 
\be\label{ode-2}
\frac{\partial \xi(t, x)}{\partial t} = - \frac{\partial \xi(t, x)}{\partial x^i}\,v^i(t, x)\,.
\ee
Also 
\[
\frac{\partial\hfill}{\partial t}\frac{\partial \xi^a}{\partial x^i} = - \frac{\partial \xi^a}{\partial x^j}\;\frac{\partial^2 x^j}{\partial t\partial \xi^b}\;\frac{\partial \xi^b}{\partial x^i}
\]
or
\be\label{eq-xi}
\frac{\partial\hfill}{\partial t}\frac{\partial \xi^a}{\partial x^i} = - \frac{\partial \xi^a}{\partial x^j}\;\frac{\partial v^j}{\partial x^i}
\ee
In particular,  
\[
|\frac{\partial \xi^a(t, x)}{\partial x^k}| \le C\,\exp \int_0^t \|\nabla v(\tau)\|_\infty\,d\tau\,.
\]
Then
\be\label{xi-xi}
|\xi(t, x) - \xi(t, y)| \le C\,|x - y|\,\exp \int_0^t \|\nabla v(\tau)\|_\infty\,d\tau\,.
\ee
Thanks to the assumption \eqref{noblowup-2}, the transformation $\xi\to x(t, \xi)$ is bi-lipschitz uniformly on any fixed finite time interval. 
As a consequence, the Sobolev $H^1$ norms in Eulerian and Lagrangian coordinates are equivalent: 
\[
\|f(t,\cdot)\|_{H^1(\R^d_x)} \simeq \|\tilde f(t,\cdot)\|_{H^1(\R^d_\xi)}\,,
\]
where $\tilde f(t, \xi) = f(t, x(t, \xi))$. 

As seen from \eqref{u-21} and \eqref{dos-7}, the functions $v^i_a(t, x) = A^i_a + u^i_a(t, x)$ satisfy the equation
\[
\partial_t v_a^i + v^k\,\partial_k v_a^i = v^k_a\,\partial_k v^i\,
\]
with the initial condition $v^i_a(0, x) = A^i_a + u^i_a(0, x)$. 
In the Lagrangian coordinates,
\[
\frac{\partial v^i_a}{\partial t} = \partial_k v^i\,v^k_a\,.
\]
This means the matrix $v^i_a$ is the solution of the same system \eqref{ode-1} as $\partial x^i/\partial\xi^a$ and with the same initial condition. This proves that 
\be\label{dx/dxi}
\frac{\partial x^i(t, \xi)}{\partial \xi^a} = v^i_a(t, x(t, \xi))\,.
\ee
Now go to the integral identities \eqref{uno-7} and \eqref{dos-7}. 
Substitute $v^i_a = \partial x^i/\partial\xi^a$ and $v = \partial x/ \partial t$  and change variables to 
$(t, \xi)$. This results in 
\be\label{umo-7-L}
\int_0^T \int - \frac{\partial x^i}{\partial t}\,\frac{\partial \underline{\eta}^i}{\partial t} + 
\frac{\partial x^i}{\partial \xi^a}\,\frac{\partial \underline{\eta}^i}{\partial \xi^a} - p \,\frac{\partial \xi^a}{\partial x^i}\,\frac{\partial \underline{\eta}^i}{\partial \xi^a}\;d\xi dt = 
\int v^L_0\,\underline{\eta}(0)\,d\xi
\ee
and, for all $a = 1,\dots, d$,
\be\label{dos-7-L}
\int_0^T \int - \frac{\partial x^i}{\partial \xi^a}\,\frac{\partial \underline{\eta}^i}{\partial t} + 
\frac{\partial x^i}{\partial t}\,\frac{\partial \underline{\eta}^i}{\partial \xi^a} \;d\xi dt = 
\int \left(A^i_a + \frac{\partial \varphi^i}{\partial\xi^a}\right)\,\underline{\eta}(0)\,d\xi
\ee
for any smooth $\underline{\eta}$ with compact support in $(- T, T)\times \R^d$. In this sense 
$x$ and $p$ solve equations \eqref{Lag-1} and  \eqref{Lag-2} with the appropriate initial conditions. 

\bigskip

 This solution is unique in the following sense: any family of diffeomorphisms $x(t, \xi)$ satisfying \eqref{Lag-ic} and such that the quantities  
$\partial x^i/\partial t$ and $\partial x^i/\partial \xi^a - A^i_a$, when expressed as functions of $t$ and $x$, belong to the space 
$C([0, T]\to H^s_\sigma(\R^d_x))$, must coincide with the solution obtained by means of Theorem \ref{Thm-1}  (simply by the uniqueness of solutions in the Euler setting). 

It is easy to show that the functions $\xi(t, x) - A^{-1}x$ viewed as functions of $t$ and $x$ belong to the space 
$C([0, T]\to H^{s+1}(\R^d_x))$. Indeed, 
\[
\begin{aligned}
\frac{\partial \xi^a}{\partial x^i} = \left[\left(\frac{\partial x}{\partial\xi}\right)^{-1}\right]^a_i = \frac{1}{(d-1)!}\,\epsilon_{i i_2\dots i_d}\,\epsilon^{a a_2\dots a_d} v^{i_2}_{a_2} \dots v^{i_d}_{a_d}\,. 
\end{aligned}
\]
Hence,
\[
\frac{\partial \xi^a}{\partial x^i} - (A^{-1})^a_i = \frac{1}{(d-1)!}\,\epsilon_{i i_2\dots i_d}\,\epsilon^{a a_2\dots a_d} \sum_B B^{i_2}_{a_2} \dots B^{i_d}_{a_d}\,,
\]
where the summation is over all matrices $B$ such that each $B^{i_k}_{a_k}$ is either $A^{i_k}_{a_k}$ or $u^{i_k}_{a_k}$, but not all are $A^{i_k}_{a_k}$. 
Hence $\frac{\partial \xi^a}{\partial x^i} - (A^{-1})^a_i \in C([0, T]\to H^{s}(\R^d_x))$ ($H^s$ is an algebra). 
Also, (see \eqref{ode-2})
\[
\xi^a(t, x) - (A^{-1} x)^a = - (\varphi^L)^a(\xi(0,x)) + \int_0^t - v^i(\tau, x)\,\left[\frac{\partial \xi^a(\tau, x)}{\partial x^i} - (A^{-1})^a_i\right] - \left(A^{-1}\right)^a_i\,v^i(\tau, x)\,d\tau
\]
and it is clear that this function is in $C([0, T]\to L^2)$. It follows that $\xi(t, x) = A^{-1}x + h(t, x)$ for some 
$h(t,\cdot)\in H^{s+1}(\R^d_x)$ depending continuously on $t$. By Lemma~\ref{difff}, the inverse map, 
$x = A\xi - A h^L(t, \xi)$ is such that $h^L(t)\in H^{s+1}(\R^d_\xi)$. 

\end{proof} 

\section{Intermezzo} \label{sec:intermezzo}

In section \ref{sec:below}
we shall transition to lower regularity solutions of equations \eqref{u-11}, \eqref{u-21}. 
Given the initial conditions $V(0) = (v(0), u(0))\in H^s_\sigma$ with $s \le \frac{d}{2} + 1$, we 
will mollify them to get $V^\epsilon(0) = \rho_\epsilon [V(0)]$, and use Theorem~\ref{Thm-1} to 
obtain the corresponding solution $V^\epsilon(t)$ on some interval $[0, T_\epsilon]$. We would like 
to be able to pass to the limit as $\epsilon\to 0$ to obtain $V(t)$, and to be sure that $T_\epsilon$ 
can be bounded from below by some $T > 0$. We'll have to modify the corresponding arguments in 
section \ref{sec:Part1}.
When $s \le \frac{d}{2} + 1$, the norm $\|\nabla V(t)\|_\infty$ 
is no longer controlled by $\|V(t)\|_{H^s}$. However, we still need to show that the integral 
$\int_0^t\|\nabla V^\epsilon(t^\prime)\|_\infty\,dt^\prime$ is finite over some time interval (independent of $\epsilon$). In addition, we have to deal with $\|\nabla \tV(t)\|_\infty$ in the proof of 
continuous dependence, and we cannot use there the argument we had for $\|\tV(t)\|_\infty$. 
Thus, we need new tools. The first observation is that for any sufficiently smooth divergence-free 
vectorfield $v$ with $\curl v = \omega$, we have inequalities of the form
\be\label{gn-1}
\|\nabla v\|_\infty \lesssim \|v\|_2^{\gamma_1}\;\|\omega\,\big|\;{\dot B}^r_{p,p}\|^{1 - \gamma_1}
\ee
and
\be\label{gn-2}
\| v\|_\infty \lesssim \|v\|_2^{\gamma_2}\;\|\omega\,\big|\;{\dot B}^r_{p,p}\|^{1 - \gamma_2}\,,
\ee
each with its own allowed range for parameters $r>0$ and $p$, and each with the powers $\gamma_1$ and $\gamma_2$ expressed in terms of $r$, $p$, and $d$. The space ${\dot B}^r_{p,p}(\R^d)$ 
is the  homogeneous Besov space.  In dimension $d=2$ we use  $p=\infty$ in which case ${\dot B}^r_{\infty,\infty}(\R^2)$ is the (homogeneous) H\"older space ${\dot C}^r(\R^2)$. When $d > 2$, we must have  
$1\le p<\infty$, and then ${\dot B}^r_{p,p}(\R^d)$  is the (homogeneous) Sobolev-Slobodetsky space ${\dot W}^{r,p}$. The precise form of inequalities \eqref{gn-1} and \eqref{gn-2} is given in Lemma \ref{GN-E-2d} and Lemma \ref{GN-E-3} 
below, and we prove a more general result, Lemma~\ref{HA}, in Appendix \ref{sec:riesz-est}.  

Since the $L^2$ norm of velocities is bounded from the basic energy conservation, we can use \eqref{gn-1}  to obtain 
\[
\int_0^T\|\nabla v(t)\|_\infty\,dt\lesssim \int_0^T \|\omega(t)\,\big|\;{\dot B}^r_{p,p}\|^{1 - \gamma_1}\,dt\,
\] 
and show that the integral on the right is \textit{a priori} bounded. In fact, we will do better. We have room to  work with 
\be\label{v-infty-m}
\int_0^T\|\nabla v(t)\|_\infty^m\,dt\lesssim \int_0^T \|\omega(t)\,\big|\;{\dot B}^r_{p,p}\|^{m(1 - \gamma_1)}\,dt\,
\ee
for some $m > 1$. To prove that the integral on the right is bounded we should look 
 at the equations satisfied by vorticities.

\subsection{Derivation of the equations for vorticities}\label{sec:vort}

In the case of the general dimension $d\ge 2$, define the vorticities $\omega^{mn} = \partial_m v^n - \partial_n v^m$ and $\omega_a^{mn} = \partial_m v_a^n - \partial_n v_a^m = \partial_m u_a^n - \partial_n u_a^m$. Equations \eqref{u-1} and \eqref{u-2} imply the following equations for the vorticities: 
\be\label{om-1}
\partial_t \omega^{mn} + v^j\,\partial_j \omega^{mn} - v^j_a\,\partial_j \omega_a^{mn} = f^{mn}
\ee
and 
\be\label{om-2} 
\partial_t \omega^{mn}_b + v^j \partial_j \omega_b^{mn} - v_b^j\, \partial_j \omega^{mn} = f_b^{mn}\,,
\ee
where
\be\label{Fom-1}
f^{mn} = - \omega^{mj}\,\partial_j v^n + \omega^{nj}\,\partial_j v^m + \omega_a^{mj}\,\partial_j v_a^n - \omega_a^{nj}\,\partial_j v_a^m\,,
\ee
and
\be\label{Fom-2}
f_b^{mn} = - \left[\partial_m v^j \,\partial_j v^n_b - \partial_n v^j \,\partial_j v^m_b\right] + \left[\partial_m v_b^j \,\partial_j v^n - \partial_n v_b^j \,\partial_j v^m\right]\,.
\ee
We combine $f^{mn}$ and  $f^{mn}_b$ into an array, $F = (F^{mn}) = (f^{mn}, f_b^{mn})$. 

In terms of the Fourier transform, ${\hat \omega}^{mn} = i\,\left(\kappa^m {\hat v}^n - \kappa^n {\hat v}^m\right)$ and the zero divergence condition provides $\kappa^n {\hat v}^n = 0$. Thus,  
\[
{\hat v}^n = - i \;\frac{\kappa^m}{|\kappa|^2}\,{\hat \omega}^{mn}\,.
\]
In our notation,
\be\label{v-om-Riesz}
v^n = - i D^{-1}{\cal R}_m \,\omega^{mn}\quad\text{and}\quad \partial_k v^n = {\cal R}_k\,{\cal R}_m \,\omega^{mn}\,, 
\ee
where ${\cal R}_m = {\cal F}^{-1} \kappa^m/|\kappa| {\cal F}$ are the Riesz transforms. 
Thus, the right sides, $f^{mn}$ and $f_b^{mn}$, of equations \eqref{om-1} and \eqref{om-2}, are sums of products of Riesz transforms of vorticities. Here and in what follows $\Omega$ is a combined notation for $\omega$ and all $\omega_a$. We will omit the superscripts ${}^{mn}$ where possible. 

\subsection{Vorticities in the Lagrangean setting}\label{sec:vort-lag}

In Lagrangean coordinates equations \eqref{om-1} and \eqref{om-2} simplify as follows: \be\label{Lom-1-late}
\frac{\partial { \omega}^{L}}{\partial t} - \frac{\partial { \omega}_a^{L}}{\partial \xi^a} = { f}^{L}[\Omega]\,,
\ee
and 
\be\label{Lom-2-late}
\frac{\partial {\omega}_b^{L}}{\partial t} - \frac{\partial { \omega}^{L}}{\partial \xi^b} = { f}_b^{L}[\Omega]\,
\ee
If the right sides are known functions, this is a linear  system  with constant coefficients for the vector $[{ \omega}^{L}, { \omega}_1^{L},\dots, { \omega}_d^{L}]^T$. Apply the Fourier transform ${\cal F}_{\xi\to k}$ to \eqref{Lom-1-late} and \eqref{Lom-2-late}. We emphasize that here the Fourier transform is applied in Lagrangian coordinates.   The result is a system of ODEs:  
\be\label{FouLom}
\frac{d\hfill}{dt} {\hat\Omega}^{L} - i\,M\,{\hat \Omega}^{L} = {\hat F}^{L}\,,
\ee
where 
\[
{\hat\Omega}^{L} = [{\hat \omega}^{L}, {\hat \omega}_1^{L},\dots, {\hat \omega}_d^{L}]^T\;,
\]
with ${\hat \omega}^L(t, k) = {\cal F}_{\xi\to k} {\omega}^L(t, \xi)$, and $M = M(k)$ is a symmetric $(d+1)\times (d+1)$ matrix whose first row is 
\[
[0, k^1, \dots, k^d]\,,
\]
whose first column is the transpose of the first row, and all other entries are zeros. The eigenvalues of $M$ are $\pm |k|$ and the multiplicity $(d-1)$ eigenvalue $0$. The corresponding eigenvectors are 
${\bf e}_1 = [1, {\hat k}^1,\dots, {\hat k}^d]^T$, ${\bf e}_2 = [1, - {\hat k}^1,\dots, -{\hat k}^d]^T$, for $|k|$, $-|k|$, respectively, and for the zero eigenvalues, ${\bf e}_3 = [0, {\hat k}^2, - {\hat k}^1, 0,\dots, 0]^T$, ${\bf e}_4 = [0, {\hat k}^3, 0, - {\hat k}^1, 0,\dots, 0]^T$, \dots, 
${\bf e}_{d+1} = [0, 0, \dots, {\hat k}^d, -{\hat k}^{d-1}]^T$, where ${\hat k}^j = k^j/|k|$.
  Introduce the quantities 
\be\label{pis}
{\hat \pi}_+^L = {\hat \omega}^L + {\hat k}^a {\hat \omega}_a^L,\quad {\hat \pi}_-^L = {\hat \omega}^L - {\hat k}^a {\hat \omega}_a^L,\quad
{\hat \pi}_{ab}^L = {\hat k}^a {\hat \omega}_b^L - {\hat k}^b {\hat \omega}_a^L\,.
\ee
Then
\be\label{om-pi}
\begin{aligned}
\omega^L & = {\cal F}^{-1}\frac12 ({\hat \pi}_+^L + {\hat \pi}_-^L)\,,\quad  
 \omega_a^L  = {\cal F}^{-1}\left(\frac12\,({\hat\pi}_+^L - {\hat\pi}_-^L)\,{\hat k}^a - {\hat k}^b\,{\hat\pi}_{ab}^L\right)\,. 
\end{aligned}
\ee
When written in terms of ${\hat \pi}^L$, equation \eqref{FouLom} split into three groups:
\be\label{FouPi}
\begin{aligned}
& \frac{d{\hat \pi}_+^L}{dt} = i |k|\,\pi_+^L + {\hat f}^L + {\hat k}^a {\hat f}_a^L \\ 
& \frac{d{\hat \pi}_-^L}{dt} = - i |k|\,\pi_-^L + {\hat f}^L - {\hat k}^a {\hat f}_a^L \\ 
& \frac{d{\hat \pi}_{ab}^L}{dt} = {\hat k}^a {\hat f}_b^L - {\hat k}^b {\hat f}_a^L
\end{aligned}
\ee 
Denote
\be\label{Fs}
{\hat F}_{\pm}^L = {\hat f}^L \pm {\hat k}^a {\hat f}_a^L\,,\quad {\hat F}_{ab}^L = {\hat k}^a {\hat f}_b^L - {\hat k}^b {\hat f}_a^L\,.
\ee
With this notation equations \eqref{FouPi} read
\be\label{FouPi-1}
\begin{aligned}
& \frac{d{\hat \pi}_+^L}{dt} = i |k|\,{\hat \pi}_+^L + {\hat F}_+^L \\ 
& \frac{d{\hat \pi}_-^L}{dt} = - i |k|\,{\hat \pi}_-^L + {\hat F}_-^L \\ 
& \frac{d{\hat \pi}_{ab}^L}{dt} =  {\hat F}_{ab}^L 
\end{aligned}
\ee
The solution to \eqref{FouPi-1} can be written as follows
\be\label{SolPi-1}
\begin{aligned}
& {\hat \pi}_{\pm}^L(t) = e^{i t |k|} \,{\hat \pi}_{\pm}^L(0) + \int_0^t e^{i(t - t^\prime) |k|}\,{\hat F}_{\pm}^L(t^\prime)\,dt^\prime \\ 
& {\hat \pi}_{ab}^L(t) = {\hat \pi}_{ab}^L(0) + \int_0^t {\hat F}_{ab}^L(t^\prime)\,dt^\prime\,.
\end{aligned}
\ee 
\bigskip

We will use the notation $\pi_{\pm}^L$ and $\pi_{ab}^L$ for the inverse Fourier transforms of the hatted quantities. 
As seen from \eqref{FouPi-1}, the quantities $\pi_{\pm}^L$ satisfy the first order hyperbolic equations of the form  
\[
\frac{\partial \pi_{\pm}}{\partial t} = \pm\,i\,\sqrt{-\Delta}\,\pi_{\pm} + F_{\pm}\,.
\] 
There is a useful version of the Strichartz inequality from \cite[Theorem 2]{Kap-1}: if $w$ is a solution of $\frac{\partial w}{\partial t} + i \sqrt{-\Delta} w = f$, then  
\be\label{Str-0}
\left(\int_0^T \|w(t)\;\big|\;{\dot B}^r_{p, p}\|^{q}\,dt\right)^{\frac{1}{q}} \le 
C\,\left(\|w(0)\;\big|\; {\dot H}^\theta\| + \int_0^T \|f(t)\;\big|\; {\dot H}^\theta\|\,dt\right)
\ee
with the parameters $r, p, q$, and $\theta$ in certain admissible ranges depending on the dimension $d \ge 2$. The constant $C$ depends on $d, r, p$, and $q$, but not on $T$ and not on $w(0)$ and $f$ (see Lemma~\ref{Strichartz} in Appendix~\ref{sec:stricharts-app} for a precise statement).  
It is important that we use the homogeneous spaces both in the inequalities \eqref{gn-1}, \eqref{gn-2} 
and in \eqref{Str-0}: the ubiquitous Riesz transforms (see e.g. \eqref{v-om-Riesz}, \eqref{pis}) are bounded in those homogeneous spaces. 
\bigskip

\noindent Now, the estimate \eqref{Str-0}, when applied to the quantities $\pi$,  will lead to the estimates 
for the quantities  
\[
\left(\int_0^T \|\Omega^L(t)\;\big|\;{\dot B}^r_{p, p}\|^{q}\,dt\right)^{\frac{1}{q}}
\]
for vorticities $\Omega^L$ in the Lagrangean coordinates. We would like to translate those as bounds 
for the right hand side of \eqref{v-infty-m}. But the norm of vorticites in \eqref{v-infty-m} is in 
the Eulerian coordinates. Thus, we need to establish correspondences (inequalities) between the  
like norms in the Lagrange and the Euler pictures. This we do in Appendix~\ref{sec:Eul-Lag-App}.
Not surprisingly, the $L^\infty$ norms of the Jacobian matrices $\partial x/\partial \xi$ and $\partial\xi/\partial x$ 
show up as factors. They are expressed in terms of $\|u(t)\|_\infty$, which in turn is bounded via \eqref{gn-2}. The $f(t)$ in \eqref{Str-0} is $F^L_+$ or $F^L_-$ from equations \eqref{FouPi-1}. After discarding the Riesz transforms and transitioning to the Euler coordinates (more of the factors $\|u(t)\|_\infty$), we end up working with the quadratic expressions \eqref{Fom-1} and \eqref{Fom-2}.  
Their ${\dot H}^\theta$ norms are essentially $\|\nabla V\|_\infty\cdot \|\Omega^E(t)\,\big|\;{\dot H}^\theta\|$, and we use \eqref{gn-1} again with appropriate values of the parameters.

\section{%
Below the critical regularity: $s \le \frac{d}{2} + 1$}\label{sec:below}

We shall use the following notation:  
\begin{subequations}\label{eq:notationOmF}
\begin{align}
u ={}& (u^1_1,\dots, u^d_1,\dots,u^1_d,\dots, u^d_d), \\ 
V ={}& (v^1,\dots,v^d, u^1_1,\dots, u^d_1,\dots,u^1_d,\dots, u^d_d), \\ 
\Omega ={}& ( \omega^{mn}, \omega^{mn}_a ), \\
F ={}& (f^{mn}, f^{mn}_a), 
\end{align}
\end{subequations} 
and use superscripts ${}^L$ or ${}^E$ to indicate the corresponding quantities are viewed in the Lagrangean, $(t, \xi)$,  or the Eulerian, $(t, x)$,  coordinates. Note that the $L^p(\R^d, dx) $ and 
$L^p(\R^d, d\xi) $, $1\le p \le\infty$, norms are the same for ${}^L$ and ${}^E$. When the setting is clear 
we omit ${}^L$ and ${}^E$.
As always, $\|\cdot \|$ is the notation for the $L^2$ norm and $(\cdot, \cdot )$ is the corresponding inner product.  

The meaning of $\lesssim$ is standard by now. 
We write $A\simeq B$ when $A\lesssim B$ and $B \lesssim A$. We shall use the symbol 
$\lesssim_{e_0}$ to indicate that the hidden constants in the inequality depend on the initial energy $e_0$.  
When integrating differential inequalities 
of the form $dy(t)/dt \lesssim a(t)\,y(t)$, we write the solution of the latter inequality as 
$
y(t) \le y(0)\ \exp\{ c\int_0^t a(t^\prime)\,dt^\prime  \} .
$
If the constant depends on the initial energy, we write $c(e_0)$.

\subsection{ A priori $L^2$ estimates}\label{sec:omega-L2} 

The basic conservation of energy follows from equations \eqref{u-11}, \eqref{u-21}:
\be\label{Energy}
\int \frac12\,|v(t)|^2 + \frac12\,|u(t)|^2\,dx = \int \frac12\,|v(0)|^2 + \frac12\,|u(0)|^2\,dx\,.
\ee
In what follows, $e_0$ is the square root of the initial energy:
\be\label{e_0}
e_0^2 =  \int \frac12\,|v(0)|^2 + \frac12\,|u(0)|^2\,dx\,.
\ee
Note that since the velocities $V^E$ are divergence free, 
\[
\|\nabla V^E\| \simeq \|\Omega^E\| .
\]
From  
equations for vorticities, we derive an $L^2$ estimate in the Euler coordinates. 
Multiplying \eqref{om-1} by $\omega^{mn}$ and \eqref{om-2} by $\omega^{mn}_b$, summing, and integrating over $\R^d$ yields, using the notation \eqref{eq:notationOmF} 
$$
\frac{1}{2} \frac{d}{dt} \| \Omega(t)\|^2 = (F,\Omega), 
$$
and hence \\\be\label{DE-om-L2}
\frac{d}{dt} \| \Omega(t)\| \lesssim \|F(t)\|, 
\ee
As formulas \eqref{Fom-1} and \eqref{Fom-2} show, we have 
\[
\begin{aligned}
& \|F\| \lesssim \|\nabla_x V^E\|_\infty\;\|\nabla_x V^E\| \lesssim \|\nabla_x V^E\|_\infty\;\|\Omega^E\|
\end{aligned}
\]
Taking this into account, integrate 
 \eqref{DE-om-L2} to obtain 
\be\label{om-L2-10}
\|\Omega(t)^E\| \le \|\Omega(0)^E\|\;\,e^{c\int_0^t \|\nabla V^E(t^\prime)\|_\infty\,dt^\prime}\,.
\ee

\subsection{ A priori ${\dot H}^\theta$ estimates for vorticities}\label{sec:omega-H}
\bigskip

We shall need the following Kato-Ponce commutator estimate: 
\begin{lemma}\label{full-Ka-Po}
Let $v = (v^1,\dots ,v^d)$ be a sufficiently smooth vectorfield in $\R^d$, and let $g$ be a sufficiently smooth scalar function. Then, for any $s > 0$, 
\be\label{KP}
\|J^s(v^k\,\partial_k g) - v^k\;\partial_k J^s g\|_2 \lesssim \|\nabla v\|_\infty\,\|J^s g\|_2 + 
\|J^{s+1} v\|_2\,\|g\|_\infty
\ee
and, for any $\theta > 0$,  
\be\label{hom-KP}
\|D^\theta(v^k\,\partial_k g) - v^k\;\partial_k D^\theta g\|_2 \lesssim \|\nabla v\|_\infty\,\|D^\theta g\|_2 + 
\|D^{\theta+1} v\|_2\,\|g\|_\infty
\ee
\end{lemma}
The proof of \eqref{KP} is in the original paper 
\cite{KP}. 
The homogeneous version of the Kato-Ponce inequality, \eqref{hom-KP}, can be obtained from \eqref{KP} via scaling.

\begin{lemma}\label{Dth-om}
\be\label{Hth-om}
\|\Omega^E(t)\,\big|\; {\dot H}^\theta\| \le \|\Omega^E(0)\,\big|\; {\dot H}^\theta\|\;\exp \left \{ c \int_0^t\left(\|\nabla V(t^\prime)\|_\infty + \|\Omega(t^\prime)\|_\infty\right) \,\,dt^\prime \right \} 
\ee
\end{lemma}
\begin{proof}
Act with $D^\theta$, $\theta > 0$,  on equations \eqref{om-1} and \eqref{om-2} and write the result in the form
\be\label{D-th-om}
\begin{aligned}
& \partial_t D^\theta\omega^{mn} + v^j\,\partial_j D^\theta\omega^{mn} - v^j_a\,\partial_j D^\theta\omega_a^{mn} = D^\theta f^{mn}[\Omega] + \Theta_1 \\ 
& \partial_t D^\theta\omega^{mn}_b + v^j \partial_j D^\theta\omega_b^{mn} - v_b^j\, \partial_j D^\theta\omega^{mn} = D^\theta f_b^{mn}[\Omega]  + \Theta_2\,,
\end{aligned}
\ee
where $\Theta_{1,2}$ are the commutator error terms:
\be
\begin{aligned}
& \Theta_1 = - \left[D^\theta\left(v^j\,\partial_j \omega^{mn}\right)  - v^j\,\partial_j D^\theta\omega^{mn}\right] + \left[D^\theta\left(v^j_a\,\partial_j \omega_a^{mn}\right) - v^j_a\,\partial_j D^\theta\omega_a^{mn}\right] \\ 
& \Theta_2 = - \left[D^\theta\left(v^j\,\partial_j \omega_b^{mn}\right)  - v^j\,\partial_j D^\theta\omega_b^{mn}\right] + \left[D^\theta\left(v^j_b\,\partial_j \omega^{mn}\right) - v^j_b\,\partial_j D^\theta\omega^{mn}\right]
\end{aligned}
\ee
By the homogeneous Kato-Ponce commutator estimate \eqref{hom-KP},  
\[
\|\Theta_1\|, \,\|\Theta_2\| \lesssim \|\nabla V\|_\infty\,\|\Omega^E\|_{{\dot H}^\theta} + \|V^E\,\big|\;{\dot H}^{\theta + 1}\|\cdot \|\Omega^E\|_\infty
\]
Note that 
\[
\|V^E\,\big|\;{\dot H}^{\theta + 1}\| \simeq \|\Omega^E\|_{{\dot H}^\theta}\,.
\]
Then 
\be
\|\Theta_1\|, \,\|\Theta_2\| \lesssim \left(\|\nabla V\|_\infty + \|\Omega\|_\infty\right)\,\|\Omega^E\|_{{\dot H}^\theta}
\ee
Applying the fractional product rule, cf. Lemma \ref{frac-leibniz}
to the terms $D^\theta f, D^\theta f_a$ and, taking into account \eqref{v-om-Riesz}, we obtain 
\be\label{F-TH}
\|D^\theta F\| \lesssim \|\nabla V\|_\infty\,\|\Omega^E\|_{{\dot H}^\theta}\,.
\ee
Thus, equations \eqref{D-th-om} lead to this estimate
\be\label{OM-TH}
\frac{d\hfill}{dt} \|\Omega^E(t)\|_{{\dot H}^\theta}^2 \lesssim \left(\|\nabla V\|_\infty + \|\Omega\|_\infty\right)\,\,\|\Omega^E(t)\|_{{\dot H}^\theta}^2\,,
\ee
and \eqref{Hth-om} follows.
\end{proof}

\subsection{Tropical ``norms" and homogeneous norms}

It shall be convenient to use the following notation:
\[
\langle g\rangle_p = \max ( 1, \|g\|_p)\,.
\]
These are not norms, but they are ``tropical" norms in the following sense.  
For a real number $a$, denote $\langle a \rangle = \max (1, |a|)$. 
Let $a$, $b$, and $c$ be real numbers. Then 
\begin{enumerate}
\item If $0 \le a \le b$, then $\langle a\rangle \le \langle b \rangle$ .
\item If $|a|\le 1 + |b|$, then $\langle a \rangle \le 2\,\langle b \rangle$. Also, $1 + |b|\le 2\,\langle b\rangle$.
\item If $\lambda \ge 0$, then $\langle \lambda\,a \rangle \le \langle\lambda\rangle\,\langle a \rangle$. 
\item If $\lambda \ge 0$, then $\langle |a|^\lambda \rangle = \langle a \rangle^\lambda$ .
\item $\langle\langle a\rangle\rangle = \langle a\rangle$\ .
\item If $\langle a \rangle \le \,\langle b \rangle$ and $\langle b \rangle \le \,\langle c \rangle$, then 
$\langle a \rangle \le \,\langle c \rangle$.
\item The triangle inequality: $\langle a + b \rangle \le \,\langle a \rangle + \langle b \rangle$. 
\item $\int_0^t \langle f(t^\prime) \rangle\,dt^\prime \le t + \int_0^t |f(t^\prime)|\,dt^\prime$ .
\end{enumerate}
The tropical norms absorb additive constants, which simplifies formulas. 

Recall that 
\[
\frac{\partial x^i}{\partial \xi^a} = v^i_a = A^i_a + u^i_a
\]
and, therefore,  
\be\label{vbul-infty}
\|v_a\|_\infty \lesssim \langle u_a\rangle_\infty\,,\quad \|u_a\|_\infty \lesssim \langle v_a\rangle_\infty\,,\quad \langle v_a\rangle_\infty \simeq \langle u_a\rangle_\infty\,.
\ee
The hidden constants depend on the matrix $A$.
\medskip

Throughout, we shall be using the following abbreviation for the homogeneous Sobolev, Besov, and H\"older norms:
\[
\begin{aligned}
& \{g\}_r = \|g\,\big|\; {\dot B}^r_{\infty, \infty}(\R^d)\|\,, \quad\{g\}_{r, p} = \|g\,\big|\; {\dot B}^r_{p,p}(\R^d)\|\,,\quad\quad  
[g]_\theta = \|g\,\big|\; {\dot H}^\theta(\R^d)\|\,.
\end{aligned}
\]

\subsection{Technical inequalities when $d = 2$}

Throughout Section~\ref{sec:below} the regularity parameter $s$ is less than $\frac{d}{2} + 1$ and we are trying to make it as small as possible. When $d = 2$, we use auxiliary parameters $\theta $ and $r$ such that 
\be\label{theta-r-2d}
s = 1 + \theta\,,\quad \theta = r + \frac34\,,
\ee 
and
\be\label{theta-r-2d><}
0 \le \theta \le 1,\quad 0 < r < 1\,,
\ee
and the goal is to show that $r$ can be arbitrarily small.
\medskip

There are two technical lemmas, Lemma~\ref{GN-E-2d} and Lemma~\ref{Strest-2d}, that we rely upon. The first one offers the Gagliardo-Nirenberg type 
multiplicative inequalities \textit{\`a la} Y.~Meyer, cf. e.g. \cite{MR1905314}.

\begin{lemma}
\label{GN-E-2d}
Let $v$ be any vectorfield in $\R^2$ with $\hbox{\rm div}\;v = 0$, and let $\omega = \curl v$.
Assuming the parameters $r$ and $\theta$ satisfy \eqref{theta-r-2d><}, we have 
\be\label{x-1}
\|\nabla v\|_\infty\,,\;\|\omega\|_\infty \lesssim \|v\|_2^{r/(r+2)}\;\|\omega\,\big| {\dot C}^r\|^{2/(r+2)}\,,
\ee
and
\be\label{x-2}
\|v\|_\infty \lesssim \|v\|_2^{(r+1)/(r+2)}\;\|\omega\,\big| {\dot C}^r\|^{1/(r+2)}\,,
\ee
and
\be\label{x-3}
\|v\|_\infty \lesssim \|v\|_2^{\theta/(\theta+1)}\;\|\omega\,\big| {\dot H}^\theta\|^{1/(\theta+1)}\,,
\ee
when $v\in L^2(\R^2)$ and $\omega\in {\dot C}^r(\R^2)$ or $\omega\in {\dot H}^\theta(\R^2)$, respectively.
\end{lemma}
\begin{proof}  In the Euler coordinates, as follows from  
formulas \eqref{v-om-Riesz}, $v$ and $\omega$ are related, schematically, as 
$v = D^{-1} {\cal R}{\cal R}\,\omega$. It follows immediately that  
\[
\|v\| \simeq \|\omega\,\big|\;{\dot H}^{-1}\|\,.
\]
Keeping this in mind, we turn to Lemma~\ref{HA} and apply it to $f = \omega$. To prove inequality \eqref{x-1} for $\|\omega\|_\infty$, take the symbol $a(\kappa) \equiv 1$  in the second inequality of Lemma~\ref{HA}. For estimates \eqref{x-2} and \eqref{x-3}, use Lemma~\ref{HA} with the operator ${\mathfrak a}$ being the product of two Riesz transforms. Applying Lemma~\ref{HA} with the operator ${\mathfrak a} = {\cal R} {\cal R}$,   obtain the inequality for $\nabla v$. 
In particular,  \eqref{eq:D.3}  with $p = \infty$ implies \eqref{x-1} and \eqref{gamma1} with $p = \infty$ 
implies \eqref{x-2}, while, \eqref{gamma1} with $p=2$ and $r=\theta$ implies \eqref{x-3}. 
\end{proof}

The following corollary is a combination of Lemma~\ref{GN-E-2d} and Lemma~\ref{E-L-E}.
\begin{corollary}\label{GN-2d}
Assume $x: \R^2_\xi \to \R^2_x$ is a volume preserving diffeomorphism. Let  $u_a^i = \partial x^i/\partial\xi^a - A^i_a$, $i, a=1, 2$, be the components of the deformation gradient tensor and 
let $\omega_a$ be the corresponding vorticities. 
If $0 < r < 1$, and if $0\le\theta\le 1$, then (with the tropical norms) 
\be\label{x-4}
\langle u_a\rangle_\infty \lesssim \langle u_a\rangle_2^{(r+1)/(r+2)}\;\langle\{\omega_a^E\}_r\rangle^{1/(r+2)}
\ee
and 
\be\label{x-5}
\langle u_a\rangle_\infty \lesssim \langle u_a\rangle_2^{(r+1)/2}\;\langle\{\omega_a^L\}_r\rangle^{1/2}
\ee
and
\be\label{x-6}
\langle u_a\rangle_\infty \lesssim \langle u_a\rangle_2^{\theta/(\theta+1)}\;\langle [\omega_a^E]_\theta\rangle^{1/(\theta+1)}
\ee
and
\be\label{x-7}
\langle u_a\rangle_\infty \lesssim \langle u_a\rangle_2^{\theta}\;\langle[\omega_a^L]_\theta\rangle\,.
\ee
\end{corollary}

\begin{lemma}\label{Om-E-L}
When $d=2$ we have
\be\label{om1}
\langle\{\Omega^E\}_r\rangle^{2/(r+2)} \lesssim_{e_0} \langle\{\Omega^L\}_r\rangle
\ee
and 
\be\label{om2}
\langle[\Omega^E]_\theta\rangle\lesssim_{e_0} \langle[\Omega^L]_\theta\rangle^{1 + \theta}\,.
\ee
\end{lemma}
\begin{proof} This follows 
 from \eqref{vbul-infty}, Lemma \ref{GN-E-2d}, and Corollary~\ref{GN-2d}. 
 \end{proof}

\bigskip

The second technical lemma presents the Strichartz inequalities. 
\begin{lemma}\label{Strest-2d}
Let $w$ be a solution of the Cauchy problem 
\[
w_t + i\sqrt{-\Delta}\;w = f\,,\quad w(0) = w_0\,,
\]
on the time interval $[0, T]$. Let 
$w_0 \in {\dot H}^\theta(\R^2)$ and $f\in L^1([0, T]\to {\dot H}^\theta(\R^2))$.
If the parameters $r$ and $\theta$ satisfy $\theta = r + 3/4$, then   
\be\label{tr-str-2d}
\left(\int_0^T \{w(t)\}_r^4\,dt\right)^{1/4} \lesssim 
[w_0]_\theta + \int_0^T [f(t)]_\theta\,dt\,.
\ee
{\sl [The constant in $\lesssim$ does not depend on $T$.]}  Written with tropical norms,  inequalities \eqref{tr-str-2d} lead to  
\be\label{trop-str-2d}
\left(\int_0^T \langle\{w(t)\}_r\rangle^4\,dt\right)^{1/4} \lesssim T^{1/4} + 
[w_0]_\theta + \int_0^T [f(t)]_\theta\,dt\,.
\ee
\end{lemma}
\begin{proof}
The proof of a more general statement than \eqref{tr-str-2d} is given in Appendix \ref{sec:stricharts-app}.  To obtain the second inequality from the first one, use
\[
\int_0^T \langle f(t) \rangle\,dt = |\,[0, T]\cap \{t : |f(t)|\le 1\}\,| + \int_0^T \chi_{\{t : |f(t)| > 1\}}\;|f(t)|\, dt\le T + \int_0^T |f(t)|\,dt\,.
\]

\end{proof}

\subsection{ A priori estimates, $d = 2$}\label{sec:d2}

The conservation of energy \eqref{Energy} tells us that the $L^2$ norm of $V(t)$ is bounded, $\| V(t) \| \leq \| V(0)\| =e_0$.
It follows from \eqref{om1} that
\be\label{yy-1}
\int_0^t \{\Omega^E(t^\prime)\}_r^{8/(r+2)}\;dt^\prime\; \lesssim_{e_0} 
\;\int_0^t \langle\{\Omega^L(t^\prime)\}_r\rangle^4\;dt\,,
\ee 
and then from  \eqref{x-1} we obtain 
\be\label{yy-3-first}
\begin{aligned}
& \int_0^t \|\nabla V^E(t^\prime)\|_\infty + 
\|\Omega^E(t^\prime)\|_\infty\;dt^\prime \lesssim_{e_0}\int_0^t \{\Omega^E(t^\prime)\}_r^{2/(r+2)}\;dt^\prime \\ 
& \lesssim_{e_0} 
\,t^{3/4}\;\left(\int_0^t \{\Omega^E(t^\prime)\}_r^{8/(r+2)}\;dt^\prime\right)^{1/4} \lesssim_{e_0} 
\;t^{3/4}\;\left(\int_0^t \langle\{\Omega^L(t^\prime)\}_r\rangle^4\;dt^\prime\right)^{1/4}
\end{aligned}
\ee 
Now, 
estimate \eqref{Hth-om} proved in Lemma~\ref{Dth-om} can be extended as follows:
\be\label{yy-2}
\begin{aligned}
& [\Omega^E(t)]_\theta \leq [\Omega^E(0)]_\theta\;\exp \left\{c(e_0)\,
t^{3/4}\;\left(\int_0^t \{\Omega^E(t^\prime)\}_r^{8/(r+2)}\;dt^\prime\right)^{1/4}
\right\}
\end{aligned}
\ee

\begin{proposition}\label{apri-2d}
Let $V(t)$ be a solution of \eqref{u-11}, \eqref{u-21} in $C([0, T_1]\to H^{s_2}_\sigma(\R^2))$ in the sense of Theorem~\ref{Thm-1}, where $s_2 > \frac{d}{2} + 1 = 2$. Assume that for some $r > 0$ the following norms of the vorticities $\Omega(t)$ at $t=0$ are bounded: 
\be
[\Omega^E(0)]_\theta \le C_0\,,\quad \{\omega^E_a(0)\}_r \le C_1\,,
\ee
where $\theta = r + \frac34$. Then there exists time $T_0 > 0$, depending only on $\|V(0)\|$ (basic energy)  and the values of $C_0$ and $C_1$, such that 
\be\label{dot-Cr-2d}
\left(\int_0^t \{\Omega^E(t^\prime)\}_r^{8/(r+2)}\;dt^\prime\right)^{1/4} \le C_2 
\ee
and 
\be\label{L4-2d}
\int_0^t \|\nabla V(t^\prime)\|_\infty^4 + \|\Omega(t^\prime)\|_\infty^4\;dt^\prime \le C_3
\ee
for all $t$ in the interval $[0, T_0]$, where the bounds, $C_2$ and $C_3$, depend only on $e_0$ and the values of $C_0$ and $C_1$. 
\end{proposition}
\begin{proof} It is convenient to use 
\be\label{yy}
y(t) = \left(\int_0^t \langle\{\Omega^E(t^\prime)\}_r\rangle^{8/(r+2)}\;dt^\prime\right)^{1/4}
\ee
as the  quantity for which we obtain the \textit{ a priori} estimates. 
In view of \eqref{yy-1}, 
\[
y(t) \lesssim_{e_0} \left(\int_0^t \langle\{\Omega^L(t^\prime)\}_r\rangle^4\;dt^\prime\right)^{1/4}\,.
\]
Due to the representations \eqref{om-pi}, and due to the boundedness of the Riesz transforms 
in ${\dot C}^r$,  we have 
\be\label{yy-3}
\begin{aligned}
 \left(\int_0^t \langle\{\Omega^L(t^\prime)\}_r\rangle^4\;dt^\prime\right)^{1/4} \lesssim_{e_0}  & \sum_{\pm}\;\left(\int_0^t \langle\{\pi_\pm^L(t^\prime)\}_r\rangle^4\;dt^\prime\right)^{1/4} + 
 \left(\int_0^t \langle\{\pi_{12}^L(t^\prime)\}_r\rangle^4\;dt^\prime\right)^{1/4}
\end{aligned}
\ee
The integrals with $\pi_\pm^L$ are  controlled by the Strichartz  estimates of Lemma \ref{Strest-2d}:
\[
\begin{aligned}
& \left(\int_0^t \langle\{\pi_\pm^L(t^\prime)\}_r\rangle^4\;dt^\prime\right)^{1/4} \lesssim\,t^{1/4} +  
[\pi_\pm^L(0)]_\theta + \int_0^t [F_\pm^L(t^\prime)]_\theta\;dt^\prime
\end{aligned}
\]
Note that  
\be\label{L-E-inity}
[\pi_{\pm}^L(0)]_\theta \lesssim [\Omega^L(0)]_\theta \underset{\eqref{ineq2}}{\lesssim}  \langle u_a\rangle_\infty^\theta\;[\Omega^E(0)]_\theta \underset{\eqref{x-6}}{\lesssim_{e_0}} 
\,\langle[\Omega^E(0)]_\theta\rangle^{(\theta+2)/(\theta+1)}
\ee
Also, by \eqref{Fs}, 
\[
[F_\pm^L]_\theta \lesssim [f^L]_\theta + [f^L_a]_\theta\,.
\]
In the two dimensional case $f^L = 0$. 
By Lemma~\ref{E-L-E}, $[f^L_a]_\theta \lesssim \langle u_a\rangle_\infty^\theta\;[f_a^E]_\theta$.  By \eqref{Fom-2}, each $f_a^E$ is a sum of products of the form $\partial v^E\;\partial v_a^E$. Each factor, $\partial v$ and $\partial v_a$, is a linear combination of the Riesz transforms of the vorticities. Applying the fractional product rule, we obtain 
\[
[\partial v\;\partial v_a]_\theta \lesssim [\partial v]_\theta \;\|\partial v_a\|_\infty + \|\partial v\|_\infty\;[\partial v_a]_\theta\,.
\]
Since $[\partial V^E]_\theta \lesssim [\Omega^E]_\theta$, we have 
\[
[F_\pm^L]_\theta \lesssim \langle u_a\rangle_\infty^\theta\;\langle u_a\rangle_\infty\,[\Omega^E]_\theta = \langle u_a\rangle_\infty^{1+\theta}\;[\Omega^E]_\theta\,.
\]
Invoking \eqref{x-6}, we get
\be\label{yy-4}
[F_\pm^L]_\theta \lesssim_{e_0} \,\langle [\Omega^E]_\theta \rangle^2\,.
\ee
Thus, 
\be\label{yy-5}
\left(\int_0^t \langle \{\pi_\pm^L(t^\prime)\}_r \rangle^4\;dt^\prime\right)^{1/4} \lesssim_{e_0} t^{1/4} + 
\langle [\Omega^E(0)]_\theta\rangle^{(\theta+2)/(\theta+1)} + \int_0^t \langle [\Omega^E(t^\prime)]_\theta \rangle^2\;dt^\prime
\ee
Use \eqref{yy-2} for $[\Omega^E]_\theta$ followed by \eqref{yy-3}:
\be\label{yy-51}
\begin{aligned}
\left(\int_0^t  \langle \{\pi_\pm^L(t^\prime)\}_r \rangle^4\;dt^\prime\right)^{1/4} & \lesssim_{e_0} t^{1/4} + 
\langle [\Omega^E(0)]_\theta \rangle^{(\theta+2)/(\theta+1)} \\  
& + t\;\langle [\Omega^E(0)]_\theta \rangle^2\;\exp\left\{ c(e_0)
t^{3/4}\;y(t)
\right\} 
\end{aligned}
\ee
To estimate the $L^4_t {\dot C}^r_x$ norm of $\pi^L_{12}$ we go directly to the equation, see \eqref{SolPi-1}, 
\[
\pi_{12}^L(t) = \pi_{12}^L(0) + \int_0^t F^L_{12}(t^\prime)\,dt^\prime\,.
\]
Then
\[
\langle \{\pi^L_{12}(t)\}_r \rangle  \le \langle \{\pi^L_{12}(0)\}_r \rangle + \langle\,\int_0^t \{F^L_{12}(\tau)\}_r\,d\tau \rangle\,.
\]
Observe that
\[
\begin{aligned}
& \int_0^t \langle\,\int_0^{t^\prime} \{F^L_{12}(\tau)\}_r\,d\tau \rangle^4\,dt^\prime \le t + \int_0^t\left(\int_0^{t^\prime} \{F^L_{12}(\tau)\}_r\,d\tau\right)^4\,dt^\prime \le 
 t + t\,\left(\int_0^{t} \{F^L_{12}(\tau)\}_r\,d\tau\right)^4 \,.
\end{aligned}
\]
Therefore, 
\[
\begin{aligned}
& \left(\int_0^t \langle\,\int_0^{t^\prime} \{F^L_{12}(\tau)\}_r\,d\tau \rangle^4\,dt^\prime \right)^{1/4} \le t^{1/4} + t^{1/4}\,\int_0^{t} \{F^L_{12}(\tau)\}_r\,d\tau\,dt\,, 
\end{aligned}
\]
and hence,
\be\label{pi-12y}
\left(\int_0^t \langle\{\pi_{12}^L(\tau)\}_r\rangle^4\,d\tau\right)^{1/4} \lesssim t^{1/4}\;\left(\langle\{\pi^L_{12}(0)\}_r\rangle + \int_0^t \{F^L_{12}(\tau)\}_r\,d\tau\,\right).
\ee
In view of \eqref{Fs},  and the continuity of the Riesz transform on Holder spaces,
\[
\{F^L_{12}\}_r \lesssim \{f^L_1\}_r + \{f^L_2\}_r\,.
\]
By Lemma~\ref{E-L-E}, 
\be\label{yy-6}
\{f^L_a\}_r\lesssim \langle u_a\rangle_\infty^r\;\{f^E_a\}_r\,. 
\ee
Again, since each $f_a^E$ is a sum of products of the form $\partial v^E\;\partial v_a^E$, we have
\[
\{f^E_a\}_r \lesssim \|\nabla V\|_\infty\;\{\Omega^E\}_r \underset{\eqref{x-1}}{\lesssim_{e_0}} \langle\{\Omega^E\}_r \rangle^{(r+4)/(r+2)}\,.
\]
Combine this with \eqref{x-4} and derive from \eqref{yy-6}
\be\label{yy-61}
\{f^L_a\}_r\lesssim_{e_0} \langle \{\Omega^E\}_r \rangle^2\,. 
\ee
Also,   
\be\label{yy-8}
\{\pi^L_{12}(0)\}_r \lesssim \{\omega^L_a(0)\}_r\underset{\eqref{ineq1}}{\lesssim} 
\langle u_a(0)\rangle_\infty^r\,\{\omega^E_a(0)\}_r \underset{\eqref{x-4}}{\lesssim_{e_0}}
 \langle \{\omega^E_a(0)\}_r \rangle^{(r+3)/(r+2)} 
\ee
Use this and \eqref{yy-61} in \eqref{pi-12y} to obtain
\be\label{yy-7}
\left(\int_0^t \langle \{\pi_{12}^L(\tau)\}_r\rangle^4\,d\tau\right)^{1/4} \lesssim t^{1/4}\;
\left(\langle \{\omega^E_a(0)\}_r \rangle^{(r+3)/(r+2)} + t^{4/(2-r)}\,y(t)^{r+2}\right)
\ee
It follows from \eqref{yy-3}, \eqref{yy-51}, and \eqref{yy-7} that 
\be\label{yy-9}
\begin{aligned}
y(t) \lesssim_{e_0} & t^{1/4} + \langle[\Omega^E(0)]_\theta \rangle^{(\theta+2)/(\theta+1)} + t\;[\Omega^E(0)]_\theta^2\;\exp\left\{ c(e_0)
t^{3/4}\;y(t)\right\} \\ 
& + t^{1/4}\;
\left(\langle \{\omega^E_a(0)\}_r \rangle^{(r+3)/(r+2)} + t^{4/(2-r)}\,y(t)^{r+2}\right)
\end{aligned} 
\ee
The continuity argument now shows that, for a continuous nonnegative $y(t)$, if $y(0) \lesssim  [\Omega^E(0)]_\theta^{(\theta+2)/(\theta+1)}$, then there exists a $T_0 > 0$ ($T_0$ depends only on the initial data: the 
energy $\|V(0)\|^2$, and the norms $[\Omega^E(0)]_\theta$ and $\{\omega^E_a(0)\}_r$) such that 
$y(t)$ stays bounded for $t\in [0, T_0]$. 

It is easy to verify (using \eqref{x-1} and \eqref{om1}) that
\[
\int_0^t \|\nabla V(t^\prime)\|_\infty^4 + \|\Omega^E(t^\prime)\|_\infty^4\;dt^\prime \lesssim_{e_0} 
\int_0^t \langle \{\Omega^E(t^\prime)\}_r \rangle^{8/(r+2)}\,dt^\prime
\]
which stays finite for $t\in[0,T_0]$. 
\end{proof}

\begin{corollary}\label{apri-col} 
In the setting of Proposition~\ref{apri-2d}, 
\be
\int_0^{T_0} \|\Omega^E(t)\,\big|\;C^r(\R^2)\|^{8/(r+2)}\;dt < \infty\,. 
\ee
\end{corollary}
\begin{proof} This follows from \eqref{dot-Cr-2d} and \eqref{L4-2d} due to the fact that 
the norm in $C^r$ is equivalent to the sum of the $L^\infty$ and the ${\dot C}^r$ norms.
\end{proof}

\subsection{Low regularity well-posedness in the case $d=2$}\label{sec:low2}

In this section we obtain the existence and uniqueness results for $H^s$-solutions with $s\le 2$. 
To achieve this, we impose an additional regularity restriction on the initial deformation. This restriction is satisfied automatically if $u_a(0) = 0$. After existence and uniqueness are settled, we 
show the continuous dependence of solutions in $C([0, T_0]\to H^s_\sigma)$ on the initial conditions. See Theorem \ref{contdep-2d}.

\begin{theorem}\label{exist-2d}  
In $\R^2$ consider the system \eqref{u-11}, \eqref{u-21}. Let $s$ be any number greater than $7/4$. 
If the initial velocities $v(0)$ and $u_a(0)$ belong to $H^s_\sigma$, and if the vorticities 
$\omega_a(0)$ belong to the homogeneous H\"older space ${\dot C}^r$ with some $0 < r \le s - \frac74$, 
then there is $T_0 > 0$ depending only on the 
$H^s_\sigma$ norms of the initial data and on $\|\omega_a(0)\,\big|\;{\dot C}^r\|$, and a unique solution $(v, u_a)$ such that
\be
v, u_a\in C([0, T_0]\to H^s_\sigma)\,, 
\ee
\be
\int_0^{T_0} \|\nabla V(t^\prime)\|_\infty^4 + \|\Omega^E(t^\prime)\|_\infty^4\;dt^\prime < \infty\,,
\ee
and
\be
\int_0^{T_0} \|\Omega^E(t)\,\big|\;C^r(\R^2)\|^{8/(r+2)}\,dt < \infty\,.
\ee
\end{theorem}
\begin{proof} Recall that $s=2$ is the critical regularity in $\R^2$ and we have already settled the case $s > 2$. Assume 
now that $7/4 < s\le 2$ and $v(0), u_a(0) \in H^s_\sigma$ and that $\omega_a(0)\in {\dot C}^r$ with $0 < r \le s - \frac74$. Mollify the initial conditions, $V^\epsilon(0) = \rho_\epsilon*V(0)$, and obtain the 
solution $V^\epsilon(t)$ by Theorem~\ref{Thm-1}. Let $\Omega^\epsilon(t)$ be the corresponding vorticities. Set $\theta = r + 3/4$ (note that then $\theta + 1\le s$). Since, as $\epsilon \to 0$,  $V^\epsilon(0)\to V(0)$ in $H^s$ and $\Omega^\epsilon(0)\to \Omega(0)$ in ${\dot H}^\theta$, by Proposition~\ref{apri-2d},  there exist $T_0 > 0$  and $C > 0$ such that for  all sufficiently small $\epsilon$ 
\be\label{unif-infty-2d}
\int_0^{T_0} \|\nabla V^\epsilon(t)\|_\infty^4 + \|\Omega^\epsilon(t)\|_\infty^4\;dt \le C\,,
\ee
and, in addition, by Corollary~\ref{apri-col}, 
\be\label{unif-om-2d}
\int_0^{T_0} \|\Omega^\epsilon(t)\,\big|\; C^r(\R^2)\|^{8/(r+2)}\,dt \le C\,.
\ee
As in the high regularity case, cf. section \ref{sec:Part1}, the estimate 
\be\label{s-energy}
\|V^\epsilon(t)\|_{H^s}^2 \leq \|V^\epsilon(0)\|_{H^s}^2\cdot \exp\left\{ c\int_0^{T_0} \|\nabla V^\epsilon(t^\prime)\|_\infty\;dt^\prime\right\}
\ee 
is true for all $t\in[0, T_0]$. Taking all this into account, we can pass to the limit along some sequence $\epsilon_n\to 0$ as in Section~\ref{subsec:highlimit}, and obtain a solution $V(t)$ of equations 
\eqref{u-11}, \eqref{u-21}. Of course, $V^\epsilon \to V$ in the sense of distributions on $(-T_0, T_0)\times\R^2$.  Due to the fact  that the space $L^4([0, T_0]\to L^\infty(\R^2))$ is the dual of $L^{4/3}([0, T_0]\to L^1(\R^2))$ and the space $L^{8/(r+2)}([0, T_0]\to B^r_{\infty, \infty}(\R^2))$ is the dual of 
$L^{8/(6-r)}([0, T_0]\to B^{-r}_{1, 1}(\R^2))$,
the functionals in \eqref{unif-infty-2d} and \eqref{unif-om-2d} are weak-$*$ lower semicontinuous, and we obtain
\be\label{infty-2d}
\int_0^{T_0} \|\nabla V(t)\|_\infty^4 + \|\Omega(t)\|_\infty^4\;dt \le C\,,
\ee
and
\be\label{om-2d}
\int_0^{T_0} \|\Omega(t)\,\big|\; C^r(\R^2)\|^{8/(r+2)}\,dt \le C\,.
\ee
This allows one to prove uniqueness of such solutions, and to show that $V\in C([0, T_0]\to H^s(\R^2))$, as in Section~\ref{subsec:highlimit}. 
\end{proof}

\bigskip

We remark that in Theorem~\ref{exist-2d}, the H\"older parameter $r \in (0,s-7/4]$ can be arbitrarily small, and we still get 
a unique solution in $H^s$ as long as $s > 7/4$. We shall refer to  the solutions described in Theorem~\ref{exist-2d} as $(s, r)$-solutions. 
Once again, for the existence and uniqueness of local in time solutions in $H^s$ with $\frac74 < s \le 2$, it is sufficient that $\{\omega_a^E(0)\}_r < \infty$ for some
 $r>0$. 
\begin{theorem}\label{contdep-2d}
Consider the set of $(s, r)$-solutions corresponding to the appropriate initial conditions. 
Assume that 
\be\label{2d-s-interval}
\frac{\sqrt{65} + 7}{8} \le s \le 2\,,
\ee
and that $r > 0$ and 
\be\label{sr-2d}
\frac{2 - s}{s - 1} \le r \le s - \frac74\,. 
\ee
Then the $(s, r)$-solutions depend continuously on the initial conditions: If $V_n(0)\to V(0)$ in $H^s_\sigma$ and 
$(\omega_n)_a(0)\to \omega_a(0)$ in ${\dot C}^r$, then $V_n$ converges to $V$ in 
$C([0, T_1]\to H^s_\sigma(\R^2))$, for some $T_1\in (0, T_0]$.

If $r \ge \frac{\sqrt{65}-7}{8}$, then continuous dependence takes place for any $s$ in the interval 
\eqref{2d-s-interval}.
\end{theorem}

\begin{proof}
The proof of continuous dependence follows the same 
steps as the proof of Proposition~\ref{prop:A} in Section~\ref{sec:contdep} until the inequality 
\be\label{tilde-V-low}
\begin{aligned}
 \frac{d\hfill}{dt}\,\|{\tilde V}(t)\|_{H^s} \lesssim & \left(\|\nabla V^\delta(t)\|_\infty 
+ \|\nabla V^\epsilon(t)\|_\infty \right)\,\|{\tilde V}(t)\|_{H^s}  \\ 
& + \left(\|V^\delta(t)\|_{H^s} + \|V^\epsilon(t)\|_{H^s}\right)\,\|\nabla {\tilde V}(t)\|_\infty \\ 
& + \|{\tilde V}(t)\|_\infty\,\|V^\epsilon(t)\|_{H^{s+1}}\,. 
\end{aligned} 
\ee
Recall that here $\tilde V$ is composed of the differences ${v}^\delta - {v}^\epsilon$ 
and ${u}_a^\delta - {u}_a^\epsilon$ with $0<\delta <\epsilon<1$, where $V^\epsilon$ is a solution corresponding to the mollified initial condition $\rho_\epsilon*V(0)$ (and similar for $\delta$). 
The time-span $[0, T_1]$ is determined according to Proposition~\ref{apri-2d} and accommodates the solutions $V_n$ for all sufficiently large $n$, and $V$, and all $V_n^\epsilon$,  and $V^\epsilon$ for all 
sufficiently small $\epsilon$. Uniformly in (sufficiently small) $\epsilon$, we have estimates 
\eqref{unif-infty-2d} and \eqref{unif-om-2d} with $T_1$ instead of $T_0$. Also, uniformly in $\epsilon$, 
\[
\sup_{0\le t\le T_1} \|V^\epsilon(t)\|_{H^s} \le M
\]
for some constant $M > 0$. 

Let us deal with the third term on the right side of \eqref{tilde-V-low}. 
As  in Section~\ref{sec:contdep}, 
\[
\|V^\epsilon(t)\|_{H^{s+1}} \lesssim \frac{C}{\epsilon}\,.
\]
At the same time, by Lemma~\ref{GN-E-2d}, inequality \eqref{x-2},
\[
\|{\tilde V}(t)\|_\infty \lesssim \|V^\delta(t) - V^\epsilon(t)\|^{\gamma_2}\,\|\Omega^\delta(t) - \Omega^\epsilon(t)\,\big|\; {\dot B}^r_{\infty, \infty}\|^{1-\gamma_2}
\]
where
\[
\gamma_2 = \frac{r+1}{r+2}\,.
\]
For the $L^2$ norms we have  \eqref{est:L2-1}, i.e.,
\[
\sup_{[0, T_1]}\,\|{\tilde V}(t)\|\lesssim \epsilon^s\,o(\epsilon)
\]
Thus, 
\be\label{restrict-1}
\begin{aligned}
& \|{\tilde V}(t)\|_\infty\,\|V^\epsilon(t)\|_{H^{s+1}} \lesssim \epsilon^{s(r+1)/(r+2) - 1}\,o(\epsilon)\;
\|\Omega^\delta(t) - \Omega^\epsilon(t)\,\big|\; {\dot B}^r_{\infty, \infty}\|^{1/(r+2)}
\end{aligned}
\ee 
For $(s, r)$-solutions,  $s \ge \frac74 + r$. This is the first restriction. 
In order that $s(r+1)/(r+2) - 1 \ge 0$,  we must have $s \ge 1 + \frac{1}{r + 1}$. This is the second restriction. Both restrictions must be satisfied, which leads to \eqref{sr-2d}. Note that for the inequality  
\[
\frac{2 - s}{s - 1} \le  s - \frac74\,
\]
to be valid, we must have $s \ge \frac78 + \frac{\sqrt{65}}{8}$.

Now consider the second term on the right side of \eqref{tilde-V-low}. It is less than $2 M\,\|\nabla{\tilde V}(t)\|_\infty$. By inequality \eqref{x-1} of Lemma~\ref{GN-E-2d}, 
\[
\|\nabla{\tilde V}(t)\|_\infty \lesssim \|{\tilde V}(t)\|^{r/(r+2)}\,\|\Omega^\delta(t) - \Omega^\epsilon(t)\,\big|\; {\dot B}^r_{\infty, \infty}\|^{2/(r+2)}
\]
and hence,
\be\label{term-2}
\|\nabla{\tilde V}(t)\|_\infty \lesssim \epsilon^{sr/(r+2)}\,o(\epsilon)\;\|\Omega^\delta(t) - \Omega^\epsilon(t)\,\big|\; {\dot B}^r_{\infty, \infty}\|^{2/(r+2)}
\ee
Now we see that \eqref{tilde-V-low} implies the inequality 
\be\label{tilde-V-8}
\begin{aligned}
 \frac{d\hfill}{dt}\,\|{\tilde V}(t)\|_{H^s} \lesssim & \left(\|\nabla V^\delta(t)\|_\infty 
+ \|\nabla V^\epsilon(t)\|_\infty \right)\,\|{\tilde V}(t)\|_{H^s}  \\ 
& + 2 M\,\epsilon^{sr/(r+2)}\,o(\epsilon)\;\|\Omega^\delta(t) - \Omega^\epsilon(t)\,\big|\; {\dot B}^r_{\infty, \infty}\|^{2/(r+2)} \\ 
& +  \epsilon^{s(r+1)/(r+2) - 1}\,o(\epsilon)\;
\|\Omega^\delta(t) - \Omega^\epsilon(t)\,\big|\; {\dot B}^r_{\infty, \infty}\|^{1/(r+2)}\,. 
\end{aligned} 
\ee
which can be integrated:
\[
\begin{aligned}
& \|{\tilde V}(t)\|_{H^s} \lesssim \|{\tilde V}(0)\|_{H^s} \;e^{ \int_0^t\left(\|\nabla V^\delta(t^\prime)\|_\infty 
+ \|\nabla V^\epsilon(t^\prime)\|_\infty \right)\,dt^\prime} \\ 
& + 2 M\,\epsilon^{sr/(r+2)}\,o(\epsilon)\;\int_0^t e^{\int_\tau^t\left(\|\nabla V^\delta(t^\prime)\|_\infty 
+ \|\nabla V^\epsilon(t^\prime)\|_\infty \right)\,dt^\prime}\;\|\Omega^\delta(\tau) - \Omega^\epsilon(\tau)\,\big|\; {\dot B}^r_{\infty, \infty}\|^{2/(r+2)}\,d\tau \\ 
& + \epsilon^{s(r+1)/(r+2) - 1}\,o(\epsilon)\;\int_0^t e^{\int_\tau^t\left(\|\nabla V^\delta(t^\prime)\|_\infty 
+ \|\nabla V^\epsilon(t^\prime)\|_\infty \right)\,dt^\prime}\;\|\Omega^\delta(\tau) - \Omega^\epsilon(\tau)\,\big|\; {\dot B}^r_{\infty, \infty}\|^{1/(r+2)}\,d\tau
\end{aligned}
\]
Using estimates \eqref{unif-infty-2d} and \eqref{unif-om-2d},  uniformly on $[0, T_1]$ we have 
\[
\|{\tilde V}(t)\|_{H^s} \lesssim \|{\tilde V}(0)\|_{H^s} + \epsilon^{sr/(r+2)}\,o(\epsilon) + \epsilon^{s(r+1)/(r+2) - 1}\,o(\epsilon) \to 0
\]
as $\epsilon \to 0$ (as $\|{\tilde V}(0)\|_{H^s} = \|V^\delta(0) - V^\epsilon(0)\|_{H^s} \to 0$). 
This proves continuous dependence as explained in Section~\ref{sec:contdep}. 
\medskip

In the case $r \ge (\sqrt{65} - 7)/8$, the space  $L^2 \cap {\dot B}^r_{\infty,\infty}$ is a subspace of any space $L^2 \cap {\dot B}^{r^\prime}_{\infty,\infty}$ with smaller $r^\prime$ (see Lemma~\ref{besov-nested}), in particular, with $0 < r^\prime < (\sqrt{65} - 7)/8$. Such $r^\prime$ can be chosen to satisfy inequality \eqref{sr-2d} provided $s$ satisfies \eqref{2d-s-interval}. This justifies the last claim of the proposition. 
\end{proof}

\subsection{Technical inequalities when $d=3$}\label{d=3} 

In the three dimensional case, the general scheme of proving low regularity well-posedness is the same as in the two dimensional case. The differences are in details. Lemma~\ref{Strichartz} suggests that in $3D$ instead of the H\"older 
space ${\dot B}^r_{\infty, \infty}$ we should use the homogeneous Slobodetsky spaces ${\dot B}^r_{p, p}$ 
with $p < \infty$. In addition, the $\theta$ in the homogeneous Sobolev spaces ${\dot H}^\theta$ will have to 
be greater than $1$. This complicates the transition from Lagrangian to Eulerian settings.

In what follows, $2\le p < \infty$ and $\theta = r + (p-2)/p$, as needed for the Strichartz inequalities. 
Lemma~\ref{GN-E-3} will require $ r > 3/p$. Denote 
\[
h = r - \frac{3}{p}\,.
\]
Note that this is the three dimensional example of the quantity $r - \frac{d}{p}$, which is the scaling (regularity) parameter for the space ${\dot B}^r_{p, p}(\R^d)$. 
Using $h$, we can write 
\[
\theta = 1 + \frac{1}{p} + h\,.
\]
We would like to have $\theta$ as small as possible (because then we would get $s = \theta + 1$ small), and we are going to show that $h$ and $1/p$ can be chosen arbitrarily small positive. 
\bigskip

We shall use the notation 
\[
[g]_\theta = \|g\,\big|\; {\dot H}^\theta(\R^3)\|\,,\quad 
\{g\}_{r,p} = \|g\,\big|\; {\dot B}^r_{p,p}(\R^3)\|\,
\]
for the relevant homogeneous spaces.

As in the two dimensional case, we have the following inequalities that follow from Lemma~\ref{HA}.
\begin{lemma}\label{GN-E-3}
Let $v$ be a vectorfield in $\R^3$ such that $\hbox{\rm div}\,v = 0$ and let $\omega = \curl v$.
Assuming the parameters $r$, $h$, 
and $\theta$ satisfy the conditions 
\be\label{param-3d}
0 < r < 1,\quad h > 0,\quad  \theta > \frac12\,,
\ee 
we have
\be\label{z-1}
\|\nabla v\|_\infty, \|\omega\|_\infty \lesssim  \|v\|_2^{h/(h+\frac52)}\;\{\omega\}_{r,p}^{\frac52/(h + \frac52)}
\ee
and
\be\label{z-2}
\|v\|_\infty \lesssim \|v\|_2^{(h+1)/(h+\frac52)}\;\{\omega\}_{r,p}^{\frac32/(h + \frac52)}
\ee
and 
\be\label{z-3}
\|v\|_\infty \lesssim \|v\|_2^{(\theta - \frac12)/(\theta+1)}\;[\omega]_\theta^{\frac32/(\theta+1)}
\ee
\end{lemma}

We need to know how these norms change under transition from the Lagrange to Euler coordinates.
\begin{lemma}\label{trop-3d} 
Assume $d=3$, $0 < r < 1$, $0 \le \theta\le 1$, and $2\le p < \infty$. Then, for sufficiently smooth $g$,  
\be\label{trans-3d<}
\begin{aligned}
& \{g^L\}_{r,p} \lesssim \langle u_a\rangle_\infty^r\;\{g^E\}_{r,p}\,,\quad \{g^E\}_{r,p} \lesssim \langle u_a\rangle_\infty^{2r}\;\{g^L\}_{r,p}\,,\quad  
 [g^L]_\theta \lesssim \langle u_a\rangle_\infty^\theta\;[g^E]_\theta\,.
\end{aligned}
\ee
If $1<\theta < 2$, then 
\be\label{trans-3d>}
[g^L]_\theta \lesssim  \langle u_a\rangle_\infty^{\theta-1}\;\left(\langle u_a\rangle_\infty + \langle u_a^E\rangle_2^{(\theta - \frac12)/(\theta + 1)}\cdot [\omega_a^E]_\theta^{\frac32/(\theta+1)}\right)\;[g^E]_\theta\,.
\ee
\end{lemma}
\begin{proof} The first group of inequalities comes from Lemma~\ref{E-L-E}. In the case $1 < \theta < 2$, we start with inequality \eqref{ineq3}, i.e.,
\be\label{ineq-3}
[g^L]_\theta \lesssim  \| v_a\|_\infty^{\theta-1}\;\left(\| v_a\|_\infty + \|u_a\|_\infty^{2 - \theta}\;\|u_a^E\,\big|\;{\dot F}^1_{d, 2}(\R^d)\|^{\theta - 1}\right)\;[g^E]_\theta\,.
\ee
Now,
\[
\|u_a\,\big|\;{\dot F}^1_{d, 2}(\R^d)\| \simeq \|D^{-1}{\cal R} {\cal R} \omega_a\,\big|\;{\dot F}^1_{d, 2}(\R^d)\| \lesssim \|\omega_a\,\big|\;{\dot F}^0_{d, 2}(\R^d)\|
\]
When $d = 3$, ${\dot F}^{1/2}_{2, 2}(\R^3)\subset {\dot F}^0_{3, 2}(\R^3)$. Hence
\[
\|\omega_a\,\big|\;{\dot F}^0_{3, 2}(\R^3)\| \lesssim [\omega_a]_{\frac12}\,.
\]
Now, $-1 < 1/2 < \theta$,  and therefore,
\[
\|\omega_a^E\;\big|\; {\dot H}^{1/2}\| \le \|\omega_a^E\;\big|\; {\dot H}^{-1}\|^{(\theta - \frac12)/(\theta + 1)}\cdot \|\omega_a^E\,\big|\;{\dot H}^\theta\|^{\frac32/(\theta+1)} \,.
\]
Recall that (regardless of the dimension) $\|\omega_a^E\;\big|\; {\dot H}^{-1}\|\lesssim \|u_a^E\|$. 
Thus, when $d = 3$,
\[
\|u_a\,\big|\;{\dot F}^1_{d, 2}(\R^d)\|\lesssim \|u_a\|^{(\theta - \frac12)/(\theta + 1)}\cdot \|\omega_a\,\big|\;{\dot H}^\theta\|^{\frac32/(\theta+1)}\,.
\]
On the other hand, $\|u_a\|_\infty \lesssim \|u_a\|_2^{(\theta - \frac12)/(\theta+1)}\;[\omega_a]_\theta^{\frac32/(\theta+1)}$ by \eqref{z-3}. Computing the exponents, we arrive at 
\eqref{trans-3d>}.

\end{proof}

\begin{corollary}\label{GN-3d}
Assume $x: \R^3_\xi \to \R^3_x$ is a volume preserving diffeomorphism. Let  $u_a^i = \partial x^i/\partial\xi^a - A^i_a$, $i, a=1, 2, 3$,  be the components of the deformation gradient tensor and 
let $\omega_a$ be the corresponding vorticities.
\be\label{z-4}
\langle u_a\rangle_\infty \lesssim \langle u_a\rangle_2^{(h+1)/(h+\frac52)}\;\langle\{\omega_a^E\}_{r,p}\rangle^{\frac32/(h+\frac52)}
\ee
and
\be\label{z-6}
\langle u_a\rangle_\infty \lesssim \langle u_a\rangle_2^{(\theta-1/2)/(\theta+1)}\;\langle[\omega_a^E]_\theta\rangle^{\frac32/(\theta+1)}
\ee
\end{corollary}
\begin{proof} These inequalities follow from \eqref{z-2} and \eqref{z-3}. 
\end{proof}

A three dimensional analogue of Lemma~\ref{Om-E-L} is as follows.
\begin{lemma}\label{Om-E-L-3d} 
When $d = 3$, 
\be\label{z-7}
\langle \{\Omega^E\}_{r, p} \rangle \lesssim  \langle u_a\rangle_2^{2 r(\theta-1/2)/(\theta+1)}\;\langle[\Omega^E]_\theta\rangle^{3 r/(\theta+1)}\;\langle \{\Omega^L\}_{r, p} \rangle\,.
\ee
\end{lemma}
\begin{proof} From \eqref{trans-3d<}, when $d = 3$ we have
\[
\langle\{\Omega^E\}_{r,p}\rangle \lesssim \langle u_a\rangle_\infty^{2r}\;\langle\{\Omega^L\}_{r,p} \rangle\,.
\]
Now use   \eqref{z-6} for the $L^\infty$ norm. This leads to \eqref{z-7}.
\end{proof}

\noindent In the three dimensional case, the Strichartz inequality we use is given by 
\begin{lemma}\label{Str-3d}
Let $w$ be a solution of the Cauchy problem 
\[
w_t + i\sqrt{-\Delta}\;w = f\,,\quad w(0) = w_0\,,
\]
on the time interval $[0, T]$. Let 
$w_0 \in {\dot H}^\theta(\R^3)$ and $f\in L^1([0, T]\to {\dot H}^\theta(\R^3))$.
If $2\le p < \infty$ and $\theta$ and $r$ are related as follows, 
\be\label{theta-3d}
\theta = r + \frac{p-2}{p} = 1 + h + \frac1p\,,
\ee
then    
\be\label{trop-str-3d}
\left(\int_0^T \{w(t)\}_{r,p}^{\frac{2p}{p-2}}\,dt\right)^{\frac{p-2}{2p}} \lesssim 
[w_0]_\theta + \int_0^T [f(t)]_\theta\,dt\,.
\ee
\end{lemma}
The proof of a more general statement is in Appendix \ref{sec:stricharts-app}.
\bigskip

\subsection{A priori estimates, $d=3$}

As always $\|V(t)\| = \|V(0)\| = e_0$. 
As in the two dimensional case, we would like to have
\[
\sup_{0\le t\le T_0} \int_0^t \|\nabla V^E(t^\prime)\|_\infty^q + \|\Omega^E(t^\prime)\|_\infty^q\,dt^\prime  < \infty
\]
for some $q > 1$. 
Thanks to \eqref{z-1} we have
\[
\int_0^t \|\nabla V^E(t^\prime)\|_\infty^q + \|\Omega^E(t^\prime)\|_\infty^q\,dt^\prime \lesssim_{e_0} 
\int_0^t \langle \{\Omega^E(t^\prime)\}_{r,p} \rangle^{\frac52 q/(h+\frac52)}\;dt^\prime
\]
We will use the Strichartz estimates in the Lagrange setting. So, we will transform $\Omega^E$ in the integral on the right to $\Omega^L$ using Lemma~\ref{Om-E-L-3d}:
\[
\begin{aligned}
& \langle \{\Omega^E(t^\prime)\}_{r,p} \rangle^{\frac52 q/(h+\frac52)}  \lesssim 
\langle u_a\rangle_2^{\frac{\theta - \frac12}{\theta+1}\cdot \frac{5rq}{h+\frac52}}\;
 \langle[\omega_a^E]_\theta \rangle^{\frac{\frac32}{\theta+1}\cdot \frac{5rq}{h+\frac52}}\;
 \langle \{\Omega^L(t^\prime)\}_{r,p} \rangle^{\frac52 q/(h+\frac52)}
\end{aligned} 
\]
Thus, absorbing the $L^2$ norm of $V$, we have
\be\label{O-1}
\begin{aligned}
& \int_0^t \langle \{\Omega^E(t^\prime)\}_{r,p} \rangle^{\frac52 q/(h+\frac52)}\;dt^\prime \lesssim_{e_0} 
\sup_{0\le t^\prime\le t}\,\langle [\omega_a^E]_\theta\rangle^{\frac{\frac32}{\theta+1}\cdot \frac{5rq}{h+\frac52}}\; \int_0^t \langle \{\Omega^L(t^\prime)\}_{r,p} \rangle^{\frac52 q/(h+\frac52)}\;dt^\prime
\end{aligned} 
\ee
Since 
\[
\frac12 - \frac1p < 1 + \frac{2h}{5}
\]
when $h > 0$ and $p\ge 2$, we define $q > 1$ by the equality 
\be\label{qhp}
\frac{\frac52 q}{\frac52 + h} = \frac{2p}{p-2}\,.
\ee
Thus, with  $q$ as in \eqref{qhp}, we have   
\be\label{gradV}
\begin{aligned}
& \int_0^t\|\nabla V^E(t^\prime)\|_\infty^{q}+ \|\Omega^E(t^\prime)\|_\infty^q\,dt^\prime \lesssim
\sup_{0\le\tau\le t}[\Omega^E(\tau)]_\theta^{\frac{2p}{p-2} \,\frac{3r}{\theta+1}}\; \int_0^t\langle \{\Omega^L(t^\prime)\}_{r,p} \rangle^{\frac{2p}{p-2}}\,dt^\prime \,.
\end{aligned}
\ee

\begin{proposition}\label{apri-3d}
Let $V(t)$ be a solution of \eqref{u-11}, \eqref{u-21} in $C([0, T_1]\to H^{s_2}_\sigma(\R^3))$ in the 
sense of Theorem~\ref{Thm-1}, where $s_2 > \frac{d}{2} + 1 = \frac52$. Assume that for some $r>0$ and $p\in [2, +\infty)$ such that $r - 3/p > 0$, the following norms of the vorticities $\Omega^E(t)$ at $t=0$ are bounded:
\[
[\Omega^E(0)]_\theta \le C_0,\quad \{\omega_a^E(0)\}_{r,p} \le C_1\,,
\]
where $\theta = 1 + r - 2/p$. Then there exists a time $T_0 > 0$, depending only on $\|V(0)\|$, $C_0$, and $C_1$, such that
\be\label{dot-Brp}
\left(\int_0^t \{\Omega^E(t^\prime)\}_{r,p}^{\frac{2p}{p-2}}\,dt^\prime\right)^{\frac{p-2}{2p}} \le C_2
\ee
and
\be\label{Lq-3d}
\int_0^t\|\nabla V^E(t^\prime)\|_\infty^{q}+ \|\Omega^E(t^\prime)\|_\infty^q\,dt^\prime \le C_3
\ee
where
\be\label{q}
q = \frac{2p}{p-2}\cdot \frac{\frac52 + r - \frac{3}{p}}{\frac52}\,,
\ee
and
\be\label{Hs-3d}
\|V(t)\|_{H^{\theta + 1}} \le C_4\,,
\ee
for all $t$ in the interval $[0, T_0]$, where $C_2$, $C_3$,  and $C_4$ depend only on $e_0$, $C_0$, and $C_1$.
\end{proposition}

\begin{proof} Keep in mind \eqref{qhp}. Denote
\be\label{Y-1}
y(t) = \left(\int_0^t \langle \{\Omega^E(\tau)\}_{r,p} \rangle^{\frac{2p}{p-2}}\,d\tau\right)^{\frac{p-2}{2p}}\,\quad\text{and}\quad z(t) = \sup_{0\le t^\prime \le t} \langle [\Omega^E(t^\prime)]_\theta \rangle\,.
\ee
It follows from \eqref{z-1} that
\[
\begin{aligned}
\int_0^t\|\nabla V^E(t^\prime)\|_\infty + \|\Omega^E(t^\prime)\|_\infty \,dt^\prime \lesssim_{e_0} \,\int_0^t \langle \{\Omega^E(t^\prime)\}_{r,p} \rangle^{\frac52/(h+\frac52)}\,dt^\prime
\end{aligned}
\]
and, therefore,
\be\label{int-Linfty-3d}
\int_0^t\|\nabla V^E(t^\prime)\|_\infty + \|\Omega^E(t^\prime)\|_\infty \,dt^\prime \lesssim_{e_0} \;t^{1/q^\prime}\,y(t)^{\frac52/(h+\frac52)}\,.
\ee 
Incorporate this into estimate \eqref{Hth-om} and obtain 
\be\label{Z-1}
z(t) \le \langle [\Omega^E(0)]_\theta \rangle\,\exp\{ c(e_0) \, t^{1/q^\prime}\,y(t)^{\frac52/(h+\frac52)} \}\,.
\ee
Now proceed to estimate $y(t)$. Applying estimate \eqref{z-7}, we obtain  
\be\label{Y-2}
y(t) \lesssim_{e_0} z(t)^{\frac{3r}{\theta+1}}\;\left(\int_0^t \langle \{\Omega^L(\tau)\}_{r,p} \rangle^{\frac{2p}{p-2}}\,d\tau\right)^{\frac{p-2}{2p}}\,.
\ee
As formulas \eqref{om-pi} show, the vorticities $\Omega^L$ are linear combinations of the Riesz transforms of the quantities $\pi^L_\pm$ and $\pi^L_{ab}$. Then,
\[
\left(\int_0^t \langle \{\Omega^L(\tau)\}_{r,p} \rangle^{\frac{2p}{p-2}}\,d\tau\right)^{\frac{p-2}{2p}} \lesssim 
\left(\int_0^t \langle \{\pi^L_\pm(\tau)\}_{r,p} \rangle^{\frac{2p}{p-2}}\,d\tau\right)^{\frac{p-2}{2p}} +  \left(\int_0^t \langle \{\pi^L_{ab}(\tau)\}_{r,p} \rangle^{\frac{2p}{p-2}}\,d\tau\right)^{\frac{p-2}{2p}}
\]
(summation over $\pm$ and $ab = 12, 23, 31$). 
Use  Strichartz inequalities \eqref{trop-str-3d} to estimate the norms $\|\pi^L_\pm\,\big|\; L^{2p/(p-2)}([0, t]\to {\dot B}^r_{p,p})\|$  
whereas $\|\pi^L_{ab}\,\big|\; L^{2p/(p-2)}([0, t]\to {\dot B}^r_{p,p})\|$ can be estimated directly 
from the representation of $\pi^L_{ab}(t)$ in \eqref{SolPi-1}. Thus we have
\be\label{pm-1}
\left(\int_0^t\langle \{\pi_\pm^L(\tau)\}_{r,p} \rangle^{2p/(p-2)}\,d\tau\right)^{\frac{p-2}{2p}}\lesssim 
t^{\frac{p-2}{2p}} + [\pi^L_\pm(0)]_\theta + \int_0^t [F^L_\pm(\tau)]_\theta\,d\tau
\ee
and
\be\label{ab-1}
\left(\int_0^t \langle \{\pi_{ab}^L(\tau)\}_{r,p} \rangle^{\frac{2p}{p-2}}\,d\tau\right)^{\frac{p-2}{2p}}\lesssim 
t^{\frac{p-2}{2p}}\;\left[ \langle \{\pi_{ab}^L(0)\}_{r,p}\rangle +  \int_0^t\{ F^L_{ab}(t^\prime)\}_{r, p}\,dt^\prime\right]\,.
\ee
We see from \eqref{pis} that 
\[
[\pi^L_\pm(0)]_\theta \lesssim [\Omega^L(0)]_\theta\,.
\]
Now use \eqref{trans-3d>} to transition to the Euler setting:
\begin{align*}
[\Omega^L(0)]_\theta \lesssim{}&  \langle u_a(0)\rangle_\infty^{\theta-1}\;\left(\langle u_a(0)\rangle_\infty + \langle u_a^E(0)\rangle_2^{(\theta - \frac12)/(\theta + 1)}\cdot [\omega_a^E(0)]_\theta^{\frac32/(\theta+1)}\right)\;[\Omega^E(0)]_\theta \\
\intertext{use \eqref{z-6}} 
\lesssim_{e_0}{}& [\omega_a^E(0)]_\theta^{\frac{\frac32 \theta}{\theta+1}}\;[\Omega^E(0)]_\theta\,.
\end{align*} 
Thus,
\be\label{omth-0}
[\pi^L_\pm(0)]_\theta \lesssim_{e_0} [\Omega^E(0)]_\theta^{\frac{(\frac52 \theta + 1)}{\theta+1}}\,.
\ee
Also,
\begin{align*} 
\{\pi_{ab}^L(0)\}_{r,p} \underset{\eqref{pis}}{\lesssim }{}& \{\omega_a^L(0)\}_{r,p}  \underset{\eqref{trans-3d<}}{\lesssim} \langle u_a(0)\rangle_\infty^r\;\{\omega_a^E(0)\}_{r,p} \\
\underset{\eqref{z-4}}{\lesssim}{}& \langle u_a^E(0)\rangle_2^{r(h+1)/(h+\frac52)}\;\{\omega_a^E(0)\}_{r,p}^{1 + \frac32 r/(h+\frac52)}
\end{align*} 
and therefore,
\be\label{omrp-0}
\{\pi_{ab}^L(0)\}_{r,p} \lesssim_{e_0} \{\omega_a^E(0)\}_{r,p}^{1 + \frac32 r/(h+\frac52)}\,.
\ee
Consider now the terms $[F^L_\pm]_\theta$. Of course,
\[
[F^L_\pm]_\theta \lesssim [f^L]_\theta + [f^L_a]_\theta\,.
\]
Then, by \eqref{trans-3d>} with \eqref{z-6},
\be\label{FLs}
\begin{aligned}
& \langle [f^L(t)]_\theta \rangle \rangle \lesssim \langle[\omega_a^E(t)]_\theta  \rangle^{\frac32 \theta/(\theta+1)}\;\langle [f^E(t)]_\theta \rangle \,, \\  
& \langle [f^L_a(t)]_\theta \rangle \lesssim [\omega_a^E(t)]_\theta  \rangle^{\frac32 \theta/(\theta+1)}\;\langle [f^E_a(t)]_\theta \rangle
\end{aligned}
\ee
Each of $f^E$ and $f_a^E$ is a linear combination of quadratic terms of the form $\partial V^E\,\partial V^E$. 
Thus, consider the ${\dot H}^\theta$  norm of $\partial V^E\,\partial V^E$. Using the fractional product rule (cf. Lemma \ref{frac-leibniz}), 
\[
[\partial V^E\,\partial V^E]_\theta \simeq \|D^\theta(\partial V^E\,\partial V^E)\|\lesssim 
 \|D^\theta \partial V^E\|\,\|\partial V^E\|_\infty 
\]
Thus, 
\be\label{temp-12}
[\partial V^E(t)\,\partial V^E(t)]_\theta \lesssim_{e_0} [\Omega^E(t)]_\theta\;\langle \{\Omega^E(t)\}_{r,p}\rangle^{\frac52/(h+\frac52)} 
\ee
Combine this with \eqref{FLs}:
\be\label{FLs-2}
[f^L(t)]_\theta\,,\;[f^L_a(t)]_\theta \lesssim_{e_0} \langle [\Omega^E(t)]_\theta \rangle^{(\frac52 \theta + 1)/(\theta+1)}\;
\langle \{\Omega^E(t)\}_{r,p} \rangle^{\frac52/(h+\frac52)} 
\ee
Then
\be\label{int-FLs-th}
\int_0^t [F^L(t^\prime)]_\theta\,dt^\prime \lesssim z(t)^{(\frac52 \theta + 1)/(\theta+1)}\;t^{1/q^\prime}\;y(t)^{\frac52/(h+\frac52)}\,.
\ee
Consider next the norms $\{ F^L_{ab}(t^\prime)\}_{r, p}$ that appear in \eqref{ab-1}. Undoing the Riesz transforms, 
\[
\{ F^L_{ab}(t)\}_{r, p} \lesssim \{ f^L_{a}(t)\}_{r, p}\,.
\]
By \eqref{trans-3d<} followed by \eqref{z-6} we have
\be\label{Frp-1}
\{ f^L_{a}(t)\}_{r, p} \lesssim \langle u_a(t)\rangle_2^{r(\theta-1/2)/(\theta+1)}\;[\omega_a(t)]_\theta^{\frac32 r/(\theta+1)}\;\{f^E_a(t)\}_{r,p}\,.
\ee
Since each $f^E_a$ is a sum of the products $\partial V^E\;\partial V^E$, we look at the ${\dot B}^r_{p,p}$ norm of the product:
\[
\{\partial V^E\;\partial V^E\}_{r,p} \lesssim \langle\partial V^E\rangle_\infty\;\{\partial V^E\}_{r,p} 
\lesssim \langle\nabla V^E\rangle_\infty\,\langle \{\Omega^E\}_{r,p}\rangle \,.
\]
Use \eqref{z-1} to continue:
\[
\{f^E_a(t)\}_{r,p} \lesssim \langle \{\Omega^E(t)\}_{r,p} \rangle^{1 + \frac52/(h+\frac52)}\,.
\]
Thus,
\be\label{Frp-2}
\{ f^L_{a}(t)\}_{r, p} \lesssim \langle [\omega_a(t)]_\theta \rangle^{\frac32 r/(\theta+1)}\;\langle \{\Omega^E(t)\}_{r,p} \rangle^{1 + \frac52/(h+\frac52)}\,.
\ee
It turns out that 
\be
1 + \frac{\frac52}{\frac52 + h} < \frac{2p}{p-2} = \frac{\frac52 q}{\frac52 + h}\,.
\ee
Indeed, 
\[
\frac{\frac52 + h}{5+h} > \frac12 - \frac{1}{p}\;\Leftrightarrow \;
\frac52 + h > \frac52 + \frac{h}{2} - \frac{5+h}{p}\;\Leftrightarrow \;
\frac{h}{2} >  - \frac{5+h}{p}\,.
\]
With $\ell > 1$ defined by the formula 
\be\label{ell}
\ell = \frac{2p}{p-2}\,\left(1 + \frac{\frac52}{\frac52 + h}\right)^{-1}\,,
\ee
we obtain
\be\label{int-FLs-rp}
\int_0^t\{ F^L_{ab}(t^\prime)\}_{r, p}\,dt^\prime \lesssim z(t)^{\frac32 r/(\theta+1)}\;t^{1/\ell^\prime}\;y(t)^{1 + \frac52/(h+\frac52)}
\ee
Tying the ends:
\[
\begin{aligned}
& \left(\int_0^t  \langle \{\Omega^L(t^\prime)\}_{r,p} \rangle^{\frac{2p}{p-2}}\,dt^\prime\right)^{\frac{p-2}{2p}} \lesssim_{e_0} 
\langle [\omega_a^E(0)]_\theta \rangle^{\frac{\frac32 \theta}{\theta+1}}\;\langle[\Omega^E(0)]_\theta  \rangle + t^{\frac{p-2}{2p}}\,\langle \{\omega_a^E(0)\}_{r,p} \rangle^{1 + \frac32 r/(h+\frac52)}  \\ 
& \ \ \\ 
& + t^{\frac{p - 2}{2p}} + z(t)^{(\frac52 \theta + 1)/(\theta+1)}\;t^{1/q^\prime}\;y(t)^{\frac52/(h+\frac52)} 
 + t^{\frac{p-2}{2p}}\;z(t)^{\frac32 r/(\theta+1)}\;t^{1/\ell^\prime}\;y(t)^{1 + \frac52/(h+\frac52)}
\end{aligned}
\]
and 
\be\label{eq:y-3d}
\begin{aligned}
& y(t) \lesssim_{e_0} z(t)^{\frac{3r}{\theta+1}}\,\left[\langle [\omega_a^E(0)]_\theta \rangle^{\frac{\frac32 \theta}{\theta+1}}\;\langle[\Omega^E(0)]_\theta  \rangle + t^{\frac{p-2}{2p}}\,\langle \{\omega_a^E(0)\}_{r,p} \rangle^{1 + \frac32 r/(h+\frac52)}\right. \\ 
& \left. + t^{\frac{p - 2}{2p}} + t^{1/q^\prime}\;z(t)^{(\frac52 \theta + 1)/(\theta+1)}\;y(t)^{\frac52/(h+\frac52)}  + t^{\frac{p-2}{2p}+\frac{1}{\ell^\prime}}\;z(t)^{\frac32 r/(\theta+1)}\;y(t)^{1 + \frac52/(h+\frac52)}\right] 
\end{aligned}
\ee
After we substitute $z(t)$ with the right hand side of \eqref{Z-1}, we obtain from \eqref{eq:y-3d} an inequality of the form 
\[
y(t) \le {\mathfrak F}(t, y(t))\,,
\]
where the function ${\mathfrak F}(t, y)$ is non-negative, continuous,  and has the property that there is a constant $C > 0$
such that, for every fixed $y > 0$, ${\mathfrak F}(t, y)\to C$ as $t\searrow 0$. The continuity argument then works 
and yields $T_0 > 0$ such that $y(t)$ remains bounded on $[0, T_0]$. 
\end{proof}

\subsection{Low regularity well-posedness in the case $d=3$}\label{sec:low3}

\bigskip
\begin{theorem}\label{exist-3d} 
In $\R^3$ consider the system \eqref{u-11}, \eqref{u-21}. Let $s$ be any number greater than $2$. 
If the initial velocities $v(0)$ and $u_a(0)$ belong to $H^s_\sigma$, and if the voriticities 
$\omega_a(0)$ belong to the homogeneous Slobodetski space ${\dot B}^r_{p,p}$ with $r>\frac3p$ and $p\in[2, +\infty)$ such that $s \ge 2 + r - \frac2p$, 
then there is $T_0 > 0$ depending only on the 
$H^s_\sigma$ norms of the initial data and on $\|\omega_a(0)\,\big|\;{\dot B}^r_{p,p}\|$,   and a unique solution $(v, u_a)$ such that
\be
v, u_a\in C([0, T_0]\to H^s_\sigma)\,, 
\ee
and 
\be
\int_0^{T_0} \|\nabla V(t^\prime)\|_\infty + \|\Omega^E(t^\prime)\|_\infty\;dt^\prime < \infty\,.
\ee
The solution will satisfy 
\be
\int_0^{T_0} \|\nabla V(t^\prime)\|_\infty^q + \|\Omega^E(t^\prime)\|_\infty^q\;dt^\prime < \infty\,,
\ee
where 
\[
q = \frac{2p}{p-2}\cdot \frac{\frac52 + r - \frac3p}{\frac52}\;,
\]
and 
\be
\int_0^{T_0} \|\Omega^E(t)\,\big|\;{\dot B}^r_{p,p}(\R^3)\|^{\frac{2p}{p-2}}\,dt < \infty\,.
\ee
\end{theorem}

\begin{proof}The proof is analogous to the proof of Theorem~\ref{exist-2d}, with necessary modifications (e.g. $L^4_t$ norms of $\|\nabla V\|_\infty$ and $\|\Omega\|_\infty$ are replaced by their $L^q_t$ norms) .
\end{proof}
\bigskip

Solutions whose existence and uniqueness are proved in Theorem~\ref{exist-3d} 
will be called the $(s, r, p)$-solutions.

\begin{theorem}\label{wp-3d-1} 
Let $V$ be an $(s, r, p)$-solution, where $1 + \sqrt{\frac32} < s \le 5/2$, $h = r - \frac3p > 0$, and 
\be\label{cond5}
\begin{aligned}
&  \frac{\frac52 - s}{s - 1}\le  h \le s - 2 - \frac1p \,.
\end{aligned}
\ee
Then the $(s, r, p)$-solution $V$ depends continuously on the initial conditions: If $V_n(0)\to V(0)$ in $H^s_\sigma$ and 
$(\omega_n)_a(0)\to \omega_a(0)$ in ${\dot B}^r_{p,p}$, then $V_n$ converges to $V$ in 
$C([0, T_1]\to H^s_\sigma(\R^3))$, for some $T_1\in (0, T_0]$.
\end{theorem}

\begin{proof}
The restrictions $r > 0$, $2\le p < +\infty$, and  
\be\label{cond1}
h = r - \frac3p > 0
\ee
come from Lemma~\ref{GN-E-3}. That 
\[
\theta = 1 + h + \frac1p
\]
is the requirement of the Strichartz inequality for the vorticities. This shows that 
we ought to have
\be\label{cond2}
s \ge \theta + 1 = 2 + h + \frac1p
\ee
for $V(t)\in H^s_\sigma$. We start with $s$ satisfying \eqref{cond2}.

 The general scheme of the proof is the same as in the two dimensional case. 
The main step is the analysis of the inequality 
\be\label{tilde-V-low-3}
\begin{aligned}
 \frac{d\hfill}{dt}\,\|{\tilde V}(t)\|_{H^s} \lesssim & \left(\|\nabla V^\delta(t)\|_\infty 
+ \|\nabla V^\epsilon(t)\|_\infty \right)\,\|{\tilde V}(t)\|_{H^s}  \\ 
& + \left(\|V^\delta(t)\|_{H^s} + \|V^\epsilon(t)\|_{H^s}\right)\,\|\nabla {\tilde V}(t)\|_\infty \\ 
& + \|{\tilde V}(t)\|_\infty\,\|V^\epsilon(t)\|_{H^{s+1}}\,. 
\end{aligned} 
\ee
We have the bounds (with some constants $M$ and $C$)
\be\label{b-1}
\sup_{0\le t\le T_1} \|V^\epsilon(t)\|_{H^s}\le M\,,
\ee
and 
\be\label{b-2}
\sup_{0\le t\le T_1} \|{\tilde V}(t)\|\lesssim \epsilon^s\,o(\epsilon)\,,
\ee
and 
\be\label{b-3}
\sup_{0\le t\le T_1} \|V^\epsilon(t)\|_{H^{s+1}}\lesssim \frac{C}{\epsilon}\,,
\ee
and
\be\label{b-4}
\int_0^{T_1} \|\nabla V^\epsilon(t)\|_\infty + \|\nabla V^\delta(t)\|_\infty\;dt \lesssim M\,.
\ee
Also, in $\R^3$ we have, from \eqref{z-2}, 
\be\label{b-5}
\|{\tilde V}(t)\|_\infty \lesssim  \|{\tilde V}(t)\|^{\gamma_2}\;\|\Omega^\delta(t) - \Omega^\epsilon(t)\,\big|\; {\dot B}^r_{p,p}\|^{1-\gamma_2}
\ee
where
\[
\gamma_2 = \frac{1+h}{\frac52 + h}\,.
\]
Using these bounds we get
\[
\|{\tilde V}(t)\|_\infty\,\|V^\epsilon(t)\|_{H^{s+1}} \lesssim  \epsilon^{s\gamma_2 - 1}\,o(\epsilon)\,\|\Omega^\delta(t) - \Omega^\epsilon(t)\,\big|\; {\dot B}^r_{p,p}\|^{1-\gamma_2}
\]
We require that $s\gamma_2 - 1 \ge 0$. This is equivalent to the inequality
\be\label{cond3}
s \ge 2 + \frac{\frac12 - h}{1 + h}\,.
\ee
The second term on the right in inequality \eqref{tilde-V-low-3} is bounded as follows (using \eqref{b-1}, \eqref{z-1}, and \eqref{b-2}):
\begin{align*} 
\left(\|V^\delta(t)\|_{H^s} + \|V^\epsilon(t)\|_{H^s}\right)\,\|\nabla {\tilde V}(t)\|_\infty \lesssim{}& 
2 M\, \|{\tilde V}(t)\|^{h/(h+\frac52)}\;\|\Omega^\delta(t) - \Omega^\epsilon(t)\,\big|\; {\dot B}^r_{p,p}\|^{\frac52/(h+\frac52)} \\ 
\lesssim{}& 2 M\,\epsilon^{sh/(h+\frac52)}\;o(\epsilon)\;\|\Omega^\delta(t) - \Omega^\epsilon(t)\,\big|\; {\dot B}^r_{p,p}\|^{\frac52/(h+\frac52)}\,.
\end{align*} 
Thus, integration of \eqref{tilde-V-low-3} leads to
\[
\begin{aligned}
& \sup_{0\le t\le T_1} \|{\tilde V}(t)\|_{H^s} \lesssim \|{\tilde V}(0)\|_{H^s}\,e^M + 
2 M e^M\;\epsilon^{sh/(h+\frac52)}\;o(\epsilon)\;\int_0^{T_1}\|\Omega^\delta(t) - \Omega^\epsilon(t)\,\big|\; {\dot B}^r_{p,p}\|^{\frac52/(h+\frac52)}\,dt  \\ 
& +  o(\epsilon)\;\int_0^{T_1}\|\Omega^\delta(t) - \Omega^\epsilon(t)\,\big|\; {\dot B}^r_{p,p}\|^{\frac32/(h+\frac52)}\,dt\,.
\end{aligned}
\]
The integrals of the vorticity norms are uniformly in $\epsilon$ bounded thanks to Proposition~\ref{apri-3d}, so the right hand side goes to $0$ as $\epsilon \searrow 0$. 

It remains to observe that \eqref{cond5} is equivalent to \eqref{cond2} and \eqref{cond3} combined.
By the way,  condition  $s > 1 + \sqrt{\frac32}$ comes from the inequality
\be\label{cond6}
 \frac{\frac52 - s}{s - 1}  <  s - 2\,
\ee
that must be satisfied for \eqref{cond5} to be possible.

\end{proof}

\begin{corollary}
In dimension $d=3$, if $\frac52 \ge s > 1 + \sqrt{\frac32}$, then for every initial condition $V(0) = (v(0), u_a(0))\in H^s_\sigma$ such that $\curl u_a(0) \in  B^{\varkappa}_{\infty, \infty}(\R^3)$, where
\be\label{varkappa-3}
\varkappa = \sqrt{\frac32} - 1\,,
\ee
there exists a unique local in time solution $v\in C([0, T_0]\to H^s_\sigma)$ which depends continuously on $V(0)$.  
\end{corollary}
\begin{proof} If $s = \frac52$, then one can find $r$ and $p$ so that $h = r - \frac3p $ satisfy (see \eqref{cond5})
\[
0 < h < \frac12 - \frac1p
\]
and, in addition, 
\be\label{less}
h < \varkappa\,.
\ee
Now assume 
$\frac52 > s > 1 + \sqrt{\frac32}$ .  On this interval,
\[
\frac{\frac52 - s}{s - 1} < \varkappa \,.
\]
Hence, one can find $r$ and $p$ so that 
\[
 \frac{\frac52 - s}{s - 1} \le h < \varkappa - \frac1p \le s - 2 - \frac1p \,.
\]
Again, $h < \varkappa$. Thus, in every case, there are $r$ and $p$ such that conditions \eqref{cond5} and \eqref{less} are satisfied. Since $L^2(\R^3)\cap {\dot B}^\varkappa_{\infty,\infty}(\R^3) \subset L^2(\R^3)\cap {\dot B}^r_{p,p}(\R^3)$ by Lemma~\ref{besov-nested}, we can apply Theorem~\ref{wp-3d-1} to 
conclude that the solutions depend continuously on the initial data.

\end{proof}

\subsection*{Acknowledgements}
L.K. would like to thank the Max-Planck Institute for Gravitational Physics (Albert Einstein Institute) for hospitality during his visits, where a substantial part of the work was done.

\appendix

\section{Function spaces} \label{sec:spaces} 

For the convenience of the reader, we include some background on the function spaces used in this paper. 
All spaces we use belong to the scales of Besov spaces $B^s_{p, q}(\R^d)$ and Lizorkin-Triebel spaces $F^s_{p, q}(\R^d)$.  Also, we use the homogenous versions of these spaces. For their definition and basic properties we rely on \cite{Triebel}.  

The $L^p(\R^d)$ norm of a function, $f$, is denoted $\|f\|_p$ or, if it is convenient, in one of the following forms: 
\[
 \|f\;\big|\;L^p\| = \|f(x)\;\big|\;L^p(dx)\|\,.
\]
Similar forms are used for the norms in other function spaces. The integral $\int$ is the Lebesgue integral over $\R^d$. 
${\cal S} = {\cal S}(\R^d)$ is the Schwartz space of rapidly decreasing test functions, and ${\cal S}^\prime = {\cal S}^\prime(\R^d)$ 
is its dual, the space of tempered distributions. The pairing between ${\cal S}^\prime$ and ${\cal S}$ is denoted $\langle f, g\rangle$ 
and, for a regular distribution $f\in L^1\subset {\cal S}^\prime$, $\langle f, g\rangle = \int f(x)\,g(x)\,dx$. 

\begin{itemize} 
\item The  Fourier transform ${\cal F}$ is defined as 
\[
{\cal F}_{x\to\kappa} f  = {\hat f}(\kappa) = \int e^{-i x \kappa} f(x)\,dx\,,
\]
with the inverse 
\[
f(x) = {\cal F}^{-1}_{\kappa\to x} {\hat f} = \int e^{i x \kappa} {\hat f}(\kappa)\,\dbar\kappa\,,
\]
where $\dbar\kappa = (2\pi)^{-d}\,d\kappa$. 
\item Notation for the Riesz and Bessel potentials: $D^s = \left(-\Delta\right)^{s/2} = {\cal F}^{-1}_{\kappa\to \cdot} |\kappa|^s {\cal F}_{x\to\kappa}$ and 
$J^s = \left(1 - \Delta\right)^{s/2} = {\cal F}^{-1}_{\kappa\to \cdot}\left(1 + |\kappa|^2\right)^{s/2} {\cal F}_{x\to\kappa}$, with $s\in\R$.
\item In this paper, all function spaces are subspaces of ${\cal S}^\prime$. If $A_1$ and $A_2$ are such Banach spaces, their intersection, $A_1 \cap A_2$, is viewed as a Banach space with the norm $\|f\,\big|\,A_1 \cap A_2\| = \max \left(\|f\,\big|\,A_1\|,\,\|f\,\big|\,A_2\|\right)$. 
\end{itemize} 

\subsection{Littlewood-Paley decomposition}

Pick a smooth function $\psi_0: [0, +\infty)\to [0, 1]$ such that  $\psi_0(s) = 1$ if $s \le 1$, $\psi_0(s) = 0$ if $s \ge 2$, and 
$0 < \psi_0(s) < 1$ if $1 < s < 2$. Set $\vp_0(s) = \psi_0(s) - \psi_0(2 s)$. Then $\hbox{supp}\,\vp_0 = [ 2^{-1}, 2]$. 
For $n\in\Z$, define $\psi_n(t) = \psi_0(2^{-n}t)$ and $\vp_n(t) = \vp_0(2^{-n}t)$. Then,
for all integer $n$,   
\[
\vp_{n-1}(t) + \vp_m(t) + \vp_{n+1}(t) = 1\quad \text{when}\quad 
t\in\hbox{supp}\; \vp_n = \{2^{n-1} \le s \le 2^{n+1}\}\,
\]
and
\[
\sum_{n = -\infty}^{+\infty} \vp_n(t) = 1\,,\quad \forall t > 0\,.
\]
Also, for all $N\in\Z$ and for all $t\ge 0$, we have 
\be\label{psiN}
\begin{aligned}
& \psi_N(t) + \sum_{n = N+1}^\infty \vp_n(t) = 1\,, \\ 
& \psi_0(t) + \sum_{n=1}^N \vp_n(t) = \psi_N(t)\,.
\end{aligned}
\ee
Abusing the notation we write $\vp_n(\kappa)$ instead of $\vp_n(|\kappa|)$, where $\kappa\in\R^d$, and similarly understood are $\psi_n(\kappa)$. 
Also, we  write $\vp_n(D)$, $\psi_n(D)$, etc. for the corresponding 
pseudodifferential operators, i.e., 
\[
\vp_n(D)\,f = \int e^{i x \kappa}\,\vp_n(\kappa)\, {\hat f}(\kappa)\,\dbar\kappa\, = {\cal F}^{-1}\vp_n {\cal F} f\,.
\]
Thanks to the first identity in \eqref{psiN}, any tempered distribution $f\in {\cal S}^\prime$ can be expanded as (the Littlewood-Paley decomposition)  
\be\label{LP}
 f = {\cal F}^{-1}\psi_N {\cal F} f + \sum_{n=N+1}^\infty {\cal F}^{-1}\vp_n {\cal F} f\;= \psi_N(D) f +\sum_{n=N+1}^\infty \vp_n(D) f\, , 
\ee
where the series converges in ${\cal S}^\prime$. 
We will abbreviate sometimes $f_n = \vp_n (D) f$.

\subsection{Homogeneous Besov and Lizorkin-Triebel spaces}

Following Triebel \cite{Triebel}, define ${\cal Z} = {\cal Z}(\R^d)$ as the subspace of $\cal S$ consisting of those test functions $\eta$ 
which satisfy the condition $\int x^\alpha\,\eta(x)\,dx = 0$ for all multiindices $\alpha = (\alpha_1,\dots, \alpha_d)\in\Z^d$ with all $\alpha_j\ge 0$. Equivalently, 
$\eta\in {\cal Z}$ iff ($\eta\in{\cal S}$ and) $\partial^\alpha {\hat \eta}(0) = 0$ for all nonnegative  multiindices $\alpha$. With the topology inherited from $\cal S$, $\cal Z$ is a complete locally convex space. 
Polynomials when viewed as elements of ${\cal S}^\prime$, annihilate  $\cal Z$: If $P$ is a polynomial, 
$P(x) = \sum c_\alpha x^\alpha$, and if $\eta\in{\cal Z}$, then
\[
\langle P, \eta\rangle = \langle {\hat P}, {\hat \eta}\rangle = \sum_\alpha c_\alpha\,\langle (i \partial)^\alpha\,\delta, {\hat \eta}\rangle = \sum_\alpha c_\alpha\,(-i)^\alpha\,\langle \delta, \partial^\alpha{\hat \eta}\rangle = 0\,.
\]
Conversely, any tempered distribution $f$ that annihilates ${\cal Z}$ is a polynomial. Indeed, if $\langle f, \eta\rangle = 0$ for all $\eta\in{\cal Z}$, then,  in particular, $\langle {\hat f}, {\hat\eta}\rangle = 0$ for every $\eta$ with $0\notin \hbox{supp}\,{\hat\eta}$. Hence, 
$\hbox{supp}\,{\hat f} = \{0\}$. Therefore, $f$ is a polynomial. 
Denote by ${\cal Z}^\prime$ the topological dual of $\cal Z$. If $\ell$ is a linear continuous functional on ${\cal Z}$, then there exist 
constants $C\ge 0$ and $K\in\Z$, $K\ge 0$, such that 
\[
|\ell(\eta)|\le\,C\,\sum_{|\alpha|\le K, |\beta|\le K}\sup_x |x^\alpha\,\partial^\beta \eta(x)|\,. 
\]
By the Hahn-Banach theorem, there exists a linear extension of $\ell$ from ${\cal Z}$ to $\cal S$ with the same inequality valid 
for all $\eta$ in $\cal S$. As elements of ${\cal S}^\prime$, any two such extensions must differ by a polynomial. 
This leads to identification of ${\cal Z}^\prime$ with the quotient space 
of ${\cal S}^\prime$ by the subspace $P\subset {\cal S}^\prime$ of all polynomials: 
${\cal Z}^\prime \simeq {\cal S}^\prime/P$. The following Littlewood-Paley decomposition applies to distributions in ${\cal Z}^\prime$: 
\be\label{homLP}
f = \sum_{n=-\infty}^\infty \vp_n(D)\,f\,,
\ee 
which really means that for every $f\in{\cal S}^\prime$ there exist an integer $K\ge 0$,  a sequence of polynomials $p_N(x)$ of degree 
not greater than $K$, and a polynomial $p_\infty(x)$ such that
\be
\sum_{n = - N}^\infty \vp_n(D)\,f + p_N \underset{N\to +\infty}{\longrightarrow} f + p_\infty\quad\text{in}\;\;{\cal S}^\prime\, 
\ee
(see \cite{MR0461123}).

\medskip

\begin{itemize} 
\item The homogeneous Besov space ${\dot B}^s_{p, q} = {\dot B}^s_{p, q}(\R^d)$ with the parameters $s\in\R$, $1\le p\le \infty$, and $1\le q < \infty$, is the subspace of ${\cal Z}^\prime$ composed of those $f\in {\cal S}^\prime$ for which the norm
\be
\|f\;\big|\; {\dot B}^s_{p, q}\| =  
\left(\sum_{n=-\infty}^\infty 2^{s n q} \|\vp_n (D) f\|_p^q\right)^{1/q}
\ee
is finite. If $q = \infty$, then
\be\label{Besov}
\|f\;\big|\; {\dot B}^s_{p, \infty}\| =  \sup_{n\in\Z} 2^{s n} \|\vp_n(D) f\|_p
\ee
\item For $s\in\R$, $1\le p < \infty$, and $1\le q < \infty$, the homogeneous Lizorkin-Triebel space ${\dot F}^s_{p, q} = {\dot F}^s_{p, q}(\R^d)$  is the subspace of ${\cal Z}^\prime$ composed of those $f\in {\cal S}^\prime$ for which the norm 
\[
\|f\;\big|\; {\dot F}^s_{p, q}\| =  \|\left(\sum_{n=0}^\infty 2^{s n q}\,|\vp_n(D) f|^q \right)^{1/q}\;\big|\;L^p\|
\]
is finite. A modification as above is needed in the case $q = \infty$. (The case $p=\infty$ requires a special treatment, see \cite{Triebel}.)
\item The homogeneous  Sobolev space ${\dot H}^s_p = {\dot H}^s_p(\R^d)$ is the space of all $f\in{\cal Z}^\prime$ such that 
the norm
\be
\|f\;\big|\; {\dot H}^s_{p}\| = \|\sum_{n=-\infty}^\infty D^s \vp_n(D)f\;\big|\;L^p\|
\ee
is finite (the range of parameters is $-\infty < s < +\infty$, $1\le p\le \infty$). 
When $p = 2$, we write ${\dot H}^s$ instead of ${\dot H}^s_2$. 
\item {\bf Basic embeddings.} 
\be
\begin{aligned}
& {\dot B}^s_{p, q_1}\subset {\dot B}^s_{p, q_2}\,,\quad {\dot F}^s_{p, q_1}\subset {\dot F}^s_{p, q_2}\,,\;\text{if}\; 1\le q_1\le q_2\le\infty\\ 
& {\dot B}^s_{p, \min(p,q)}\subset {\dot F}^s_{p, q}\subset {\dot B}^s_{p, \max(p,q)}\,,\\ 
& {\dot B}^{s_1}_{p_1, q_1} \subset {\dot B}^{s_2}_{p_2, q_2}\;\text{if}\;1\le p_1\le p_2\le\infty, \;  1\le q_1\le q_2\le\infty, \;s_2 - \frac{d}{p_2} = s_1 - \frac{d}{p_1} \\  
& {\dot F}^{s_1}_{p_1, q_1} \subset {\dot F}^{s_2}_{p_2, q_2}\;\text{if}\;1\le p_1 < p_2 < \infty, \;  1\le q_1,\, q_2\le\infty, \;s_2 - \frac{d}{p_2} = s_1 - \frac{d}{p_1}
\end{aligned}
\ee
\item For all $s, r\in\R$, the operator $D^r = (-\Delta)^{r/2} = {\cal F}^{-1}_{\kappa\to\cdot}|\kappa|^r{\cal F}_{x\to\kappa}$ is an isomorphism between 
${\dot B}^{s+r}_{p, q}$ and ${\dot B}^{s}_{p, q}$ if $p, q\in[1,\infty]$, and between ${\dot F}^{s+r}_{p, q}$ and ${\dot F}^{s}_{p, q}$, when 
$1\le p < \infty$, $1\le q\le\infty$, see Theorem 5.2.3.1 in \cite{Triebel}. 
\item The topological dual of ${\dot B}^s_{p,q}$ is ${\dot B}^{-s}_{p^\prime,q^\prime}$ and the topological dual of ${\dot F}^s_{p,q}$ is ${\dot F}^{-s}_{p^\prime,q^\prime}$, where $s\in\R$, $1\le q <\infty$, and $1\le p <\infty$. (As usual, $1/p+1/p^\prime = 1/q+1/q^\prime = 1$.)
\item {\bf Interpolation inequalities.} If $0 < \theta < 1$, $s = (1-\theta) s_0 + \theta s_1$, and if 
\[
\frac{1}{p} = \frac{1-\theta}{p_0} + \frac{\theta}{p_1}\,,\quad \frac{1}{q} = \frac{1-\theta}{q_0} + \frac{\theta}{q_1}\,,
\]
where $1\le p_0, p_1\le\infty$ and $1\le q_0, q_1\le\infty$, then 
\[
\|f\;\big|\; {\dot B}^s_{p, q}\|\lesssim \|f\;\big|\; {\dot B}^{s_0}_{p_0, q_0}\|^{1-\theta}\;\|f\;\big|\; {\dot B}^{s_1}_{p_1, q_1}\|^{\theta}\,.
\]
This is due to the fact that ${\dot B}^s_{p, q} = \left[{\dot B}^{s_0}_{p_0, q_0}, {\dot B}^{s_1}_{p_1, q_1}\right]_\theta$, the complex interpolation. The analogous result is true for the Lizorkin-Triebel spaces:
\[
\|f\;\big|\; {\dot F}^s_{p, q}\|\lesssim \|f\;\big|\; {\dot F}^{s_0}_{p_0, q_0}\|^{1-\theta}\;\|f\;\big|\; {\dot F}^{s_1}_{p_1, q_1}\|^{\theta}\,.
\]
\item Isomorphisms between spaces. 
\be
\begin{aligned}
& {\dot F}^0_{p,2} \simeq L^p\,,\quad {\dot F}^s_{p,2}\simeq {\dot H}^s_p \,,\; 1 < p < \infty, s\in\R \\ 
& {\dot F}^s_{p,p} \simeq {\dot B}^s_{p,p}  
\end{aligned}
\ee
It is known (\cite{Triebel}) that, for $r\in (0, 1)$, the ${\dot B}^r_{\infty, \infty}$-seminorm is equivalent to  the homogeneous H\"older 
${\dot C}^r$ seminorm:
\[
\sup_{n\in \Z} 2^{rn}\;\|{\cal F}^{-1}\vp_n {\cal F} f\|_\infty\quad \simeq\quad \{f\}_r = \sup_{x \neq y} \frac{|f(x) - f(y)|}{|x - y|^r}\,,
\]
and, for $1\le p < \infty$, the ${\dot B}^r_{p, p}$-seminorm is equivalent to the Gagliardo seminorm
\[
\left[f\;\big|\;{\dot B}^r_{p, p}\right]_* = \left(\int \int \frac{|f(x) - f(y)|^p}{|x - y|^{r p + d}} \right)^{1/p}
\]
An equivalent seminorm in  ${\dot H}^s$ is
\[
[f]_s = \left(\int |\kappa|^{2\,s} \,|{\hat f}(\kappa)|^2\,\dbar \kappa\right)^{1/2}
\]
\item The nonhomogeneous Besov and Lizorkin-Triebel spaces are made of tempered distributions with the norms
\[
\|f\;\big|\; B^s_{p, q}\| = \|\psi_0(D)\,f\|_p + \left(\sum_{n=0}^\infty 2^{s n q} \|\vp_n (D) f\|_p^q\right)^{1/q}
\]
and 
\[
\|f\;\big|\; F^s_{p, q}\| = \|\psi_0(D)\,f\|_p + \|\left(\sum_{n=0}^\infty 2^{s n q}\,|\vp_n(D) f|^q \right)^{1/q}\;\big|\;L^p\|\,,
\]
respectively. Equivalent norms are obtained when the part $\|\psi_0(D)\,f\|_p$ is replaced with $\|f\|_p$. For all $s, r\in\R$, the operator $J^r = (1-\Delta)^{r/2} = {\cal F}^{-1}_{\kappa\to\cdot}(1+|\kappa|^2)^{r/2}{\cal F}_{x\to\kappa}$ is an isomorphism between the nonhomogeneous spaces 
${ B}^{s+r}_{p, q}$ and ${ B}^{s}_{p, q}$ if $p, q\in[1,\infty]$, and between ${ F}^{s+r}_{p, q}$ and ${ F}^{s}_{p, q}$, when 
$1\le p < \infty$, $1\le q\le\infty$. The nonhomogeneous spaces are monotone with respect to the parameter $S$:
\[
B^{s_1}_{p, q} \subset B^{s_2}_{p, q}\quad\text{and}\quad F^{s_1}_{p, q} \subset F^{s_2}_{p, q}
\]
when $s_1\ge s_2$. The corresponding homogeneous spaces are not monotone with respect to $s$. However, the following result is easy to prove.
\end{itemize}

\begin{lemma}\label{besov-nested}
Let $s > 0$, $1\le m\le p\le\infty$, and $1\le q\le\infty$. Then $L^m\cap {\dot B}^s_{p, q} = L^m\cap { B}^s_{p, q}$. If $s_1 > s_2 > 0$, then 
$L^m\cap {\dot B}^{s_1}_{p, q}\subset L^m\cap {\dot B}^{s_2}_{p, q}$. Moreover, if $s_1 \ge s_2 > 0$, $1\le m\le p_2\le p_1\le\infty$, and 
\[
s_1 - \frac{d}{p_1} \ge s_2 - \frac{d}{p_2}\,,
\]
then  $L^m\cap {\dot B}^{s_1}_{p_1, q} \subset L^m\cap {\dot B}^{s_2}_{p_2, q}$ (for any $1\le q\le\infty$).
\end{lemma}
\begin{proof}
Assume $f\in L^m\cap {\dot B}^s_{p, q}$. Then, for any $g\in L^{p^\prime}$, 
\[
\langle g, \psi_0(D) f\rangle = \int h(x - y)\,f(y)\,g(x)\,dx\,dy\,,
\]
where 
\[
h(x) = \int e^{i x \kappa}\,\psi_0(\kappa)\,\dbar\kappa\,.
\]
By Young's convolution inequality, 
\[
|\langle g, \psi_0(D) f\rangle | \le \|h\|_\ell\;\|f\|_m\,\|g\|_{p^\prime}\,,
\]
if
\[
\frac1\ell = 1 + \frac1p - \frac1m\,.
\]
Since $m \le p$, this equality defines $\ell$ so that $1\le \ell < \infty$. Clearly, $h\in L^\ell$, since $\psi_0$ is smooth and has compact support. Thus, 
\[
\|\psi_0(D) f\|_p \lesssim \|f\|_m\,.
\]
This proves that $f\in L^m\cap { B}^s_{p, q}$ and $L^m\cap {\dot B}^s_{p, q} \subset L^m\cap { B}^s_{p, q}$. 
If $s > 0$, then ${ B}^s_{p, q}\subset {\dot B}^s_{p, q}$. 
Together, these observations prove the isomorphism 
of spaces: $L^m\cap {\dot B}^s_{p, q} = L^m\cap { B}^s_{p, q}$. The remaining statements follow from the corresponding statements for non-homogeneous Besov spaces.

\end{proof}

\subsection{Gagliardo-Nirenberg inequality and Runst's lemma}

Recall the classical Gagliardo-Nirenberg inequality:
\be\label{classic-GN}
\| D^j g\|_p \lesssim \|D^m g\|_{p_1}^\alpha\,\|g\|_{p_2}^{1 - \alpha}\,,
\ee
where $j$ and $m$ are  integers such that $0 < j < m$, $1 \le p_1 < \infty$, $1 \le \,p_2 \le\infty$, and 
\[
\alpha = \frac{j}{m}\,,\quad \frac1p = \frac{\alpha}{p_1} + \frac{1 - \alpha}{p_2}\,.
\]
For more general forms/versions of this inequality see  \cite{MR3818110,MR3813967}.

Runst's inequality is a type of a Gagliardo-Nirenberg inequality stated in terms of the Lizorkin-Triebel spaces.   Its interesting feature is that there is no restriction on the parameters $q_1$ and $q$ within the range $(0, +\infty]$ (though the constants depend on their choice). 
\begin{lemma}\label{Runst} 
 Let $\alpha\in (0, 1)$ and $0 < p < \infty$, and  $0 < q_1, q \le \infty$, $r > 0$. Then, for any 
 $g\in  L^\infty\cap F^s_{p, q_1}$, 
\be\label{runst}
\|g\;\big|\; F^{\alpha r}_{p/\theta, q}\| \lesssim \|g\;\big|\; F^r_{p, q_1}\|^\alpha\,\|g\|_\infty^{1 - \alpha}\,.
\ee
Also, for any 
 $g\in  L^\infty\cap {\dot F}^r_{p, q_1}$, 
\be\label{hrunst}
\|g\;\big|\; {\dot F}^{\alpha r}_{p/\theta, q}\| \lesssim \|g\;\big|\; {\dot F}^r_{p, q_1}\|^\alpha\,\|g\|_\infty^{1 - \alpha}\,.
\ee

\end{lemma}
\begin{proof}
The proof in the nonhomogeneous case, \eqref{runst},  is given by Runst, see 
Lemma 1, Section 5.2 of \cite{Runst}, and also Lemma 1, Section 5.3.7 \cite{Runst-Sickel-book}). It relies 
on Oru's lemma, \cite[Lemma 3.7]{MR3813967}. For the homogeneous spaces one needs a slight generalization of Oru's lemma, namely,
\begin{lemma}\label{oru} 
If $-\infty < s_1 < s_2 < +\infty$, $0 < q < \infty$, $0 < \alpha < 1$, and $s = \alpha s_1 + (1 - \alpha) s_2$, then
\be\label{foru}
\| 2^{s j} a_j\|_{\ell^q} \lesssim \|2^{s_1j} a_j\|_{\ell^\infty}^{\alpha}\;\|2^{s_2j} a_j\|_{\ell^\infty}^{1 - \alpha}
\ee
for any sequence $\{a_j\}_{j=-\infty}^\infty$. 
\end{lemma} 
We leave its proof to the reader and continue with the proof of \eqref{hrunst}. 

Denote $g_n = {\cal F}^{-1} \vp_n {\cal F}g$. With $s = \alpha r$, $s_1 = r$, and $s_2 = 0$, it follows from Lemma \ref{oru} that
\[
\left(\sum_n 2^{s n q} |g_n|^q\right)^{1/q} \lesssim \left(\sup_n 2^{s_1 n} |g_n|\right)^{\alpha}\;\left(\sup_n 2^{s_2 n} |g_n|\right)^{1 - \alpha} = \left(\sup_n 2^{r n} |g_n|\right)^{\alpha}\;\left(\sup_n  |g_n|\right)^{1 - \alpha}
\]
and, consequently,
\[
\left(\sum_n 2^{\alpha r n q} |g_n|^q\right)^{1/q} \lesssim \left(\sup_n 2^{r n} |g_n|\right)^{\alpha}\;\|\sup_n  |g_n|\|_\infty^{1 - \alpha}\,.
\]
Since $\sup_n 2^{r n} |g_n| \le \left(\sum_{n\in\Z} 2^{r n q_1} |g_n|^{q_1} \right)^{1/q_1}$ for any $q_1 > 0$, and since 
$\|\sup_n  |g_n|\|_\infty \lesssim \|g\|_\infty$, we have
\[
\left(\sum_n 2^{\alpha r n q} |g_n|^q\right)^{1/q} \lesssim \left(\sum_{n\in\Z} 2^{r n q_1} |g_n|^q \right)^{\alpha/q_1}\;\|g\|_\infty^{1 - \alpha}\,.
\]
Take the $L^{p/\alpha}$ norm of both sides to obtain \eqref{hrunst}.

\end{proof}

\section{Norms of products and compositions of functions}  

\subsection{Norms of products} \label{sec:products}

There are several results (sometimes called the fractional Leibniz rule) on the Sobolev/Lizorkin-Triebel norms of products of functions. 

\begin{lemma}\label{frac-leibniz}
If $s > 0$ and $1 < p < \infty$, then
\be \label{product-H}
\begin{aligned} 
\|D^s (f_1\cdot f_2)\|_p \lesssim 
 \|D^s f_1\|_{q_1}\,\|f_2\|_{q_2} + \|f_1\|_{q_3}\,\|D^s f_2\|_{q_4}
\end{aligned}
\ee
provided
\[
1 < q_1, q_2, q_3,  q_4\le \infty, \quad \frac{1}{q_1} + \frac{1}{q_2} = \frac{1}{q_3} + \frac{1}{q_4} = \frac{1}{p}\,.
\]
The same inequality is true with $D^s$ replaced by the operator $J^s = (1 - \Delta)^{s/2}$, with the same restrictions on the parameters:
\be\label{product-F}
\begin{aligned}
\|f_1\cdot f_2\;\big|\; F^s_{p, 2}\| \lesssim 
 \|f_1\;\big|\; F^s_{p\,q_1, 2}\|\,\|f_2\|_{p q_2} + \|f_1\|_{p q_3}\,\|f_2\;\big|\; F^s_{p\,q_4, 2}\|\,.
\end{aligned}
\ee
As a corollary, 
if $s > 0$ and $1 < p < \infty$,  then  
\be\label{product-many}
\|f_1\cdot \cdots \cdot f_N\;\big|\;F^s_{p, 2}\|\le C\,\sum_{j = 1}^N \|f_j\;\big|\; F^s_{p\,p_j, 2}\|\,\cdot\,\prod_{i\neq j} \|f_i\|_{p\,p_i}
\ee
provided
\[
\sum_1^N \frac{1}{p_j} = 1\,.
\] 

\end{lemma}

For the proofs of \eqref{product-H} and \eqref{product-F} see \cite[Theorem 1]{Grafakos} and \cite[Theorem1.1]{BL-1}. Some more general inequalities  are established in \cite{Runst-Sickel-book}. For earlier results see the Christ and Weinstein paper \cite[Prop. 3.3]{CW} and \cite[Prop. 2.1.1]{Taylor}.

\subsection{Norms of compositions} 

Lemma~\ref{frac-leibniz} is used, in particular, to 
 estimate the $H^s$ norms of compositions $f(u)$, where $f$ is a sufficiently smooth functions and $u\in H^s\cap L^\infty$. The simplest result deals with the case $0 < s \le 1$. 
\begin{lemma}\label{chain-1}
Suppose $f: \R\to \R$ is a locally Lipschitz  function such that $f(0) = 0$. 
Then, if $0 < s \le 1$ and $1 < p < \infty$, 
\[
\|f(u(\cdot))\,\big|\;F^s_{p, 2}(\R^d)\| \le C_f(\|u\|_\infty)\;\|u\,\big|\;F^s_{p, 2}(\R^d)\|
\]
for any function $u\in F^s_{p, 2}(\R^d)$ (recall that $F^s_{p, 2}(\R^d) = H^{s, p}(\R^d)$). Here 
\[
C_f(\|u\|_\infty) = \inf \{C\,:\;|f(z_1) - f(z_2)|\le C \,|z_1 - z_2|\,,\;\forall z_1, z_2 :\;|z_{1,2}|\le \|u\|_\infty\}
\]
\end{lemma}
The proof of Lemma \ref{chain-1} relies on the fact that $|f(u(x)) - f(u(y))|\le C_f(\|u\|_\infty)\,|u(x) - u(y)|$, see \cite[Prop. 2.4.1]{Taylor}. Next, consider larger $s$. In the main body of the paper we need only the case of the  Sobolev scale $H^s$.

\begin{lemma}\label{chain-2}
Let $s > 0$.  Let $f: \R\to \R$ be $r =\lfloor s \rfloor$ times  continuously differentiable function such that  its $r$\,th derivative is locally Lipschitz, and $f(0) = 0$. 
Then there is a continuous, nondecreasing function 
$C_{f, s, d}: \R_+\to \R_+$ such that
\be\label{fuHs}
\|f(u(\cdot))\,\big|\;{H^s(\R^d)} \| \le C_{f, s, d}(\|u\|_\infty)\;\|u\|_{H^s(\R^d)}
\ee
for any function $u\in H^s(\R^d)\cap L^\infty(\R^d)$.
\end{lemma}
\begin{proof} For infinitely differential functions $f$, there is an elegant and short proof of \eqref{fuHs} in H\"ormander's book 
\cite[Theorem 8.5.1]{MR1466700}.
However, having in mind finitely differentiable $f$, we present a different argument.  It can be used for other purposes as well (we use it to prove 
regularity of the inverse map in Lemma~\ref{difff}).
\medskip

For $s$ an integer, to prove \eqref{fuHs} we follow Moser's argument  on p. 273 of  paper \cite{Moser}. Let $s = r$ be an integer greater than $1$. 
Observe that $|f(u)|\le C_f(|u|)\,|u|$, and hence $\|f(u)\| \le C_f(\|u\|_\infty)\,\|u\|$. Now it suffices to consider the derivatives of the order $r$ of $f(u(x))$. 
Schematically, 
\be\label{partialf}
\partial^r f(u) = \sum_{k \le r} f^{(k)}(u) \sum_{|\alpha| = k} C_{k \alpha} (\partial u)^{\alpha_1}(\partial^2 u)^{\alpha_2}\dots (\partial^r u)^{\alpha_r} 
\ee
where $\alpha_1 + 2\alpha_2 + \dots + r\alpha_r = r$ and $C_{k \alpha}$ are non-negative constants. We have 
\[
\|\partial^r f(u)\| \le \sum_{k \le r} \|f^{(k)}(u)\|_\infty \sum_{|\alpha| = k} C_{k \alpha}\,
\prod_{m=1}^r\|(\partial^m u)^{\alpha_m}\|_{2 p_m}\,,
\]
where $1\le p_m\le\infty$ and $\sum_m 1/p_m = 1$. The right choice of $p_m$ is   
\[
 p_m = \frac{r}{m\,\alpha_m}\,,
\]
because then
\[
\|(\partial^m u)^{\alpha_m}\|_{2 p_m} =  \|\partial^m u\|_{2 \alpha_m p_m}^{\alpha_m} 
\]
and
\[
\|\partial^m u\|_{2 \alpha_m p_m} = \|\partial^m u\|_{2 r/m} \lesssim \|u\|_{H^r}^{m/r}\,\|u\|_\infty^{1 - m/r}\,, 
\]
by the Gagliardo-Nirenberg inequality. Collecting all the terms we obtain
\[
\|\partial^r f(u)\|_2 \lesssim \|u\|_{H^r}\;\sum_{k \le r} C_{f^{(k)}}(\|u\|_\infty)\,\|u\|_\infty^{k-1} \sum_{|\alpha| = k}C_{k\alpha}\,.
\]
Now consider the case of fractional $s$. Assume $s = r + \gamma$, $r \ge 1 $ is an integer and $0 < \gamma < 1$. 
Since we have the $H^r$ norm of $f(u)$ already bounded, it remains to show that the $H^\gamma$ norm of each term of the form
\be\label{partialg}
g(u)\,(\partial u)^{\alpha_1}(\partial^2 u)^{\alpha_2}\dots (\partial^r u)^{\alpha_r}
\ee
is bounded (see \eqref{partialf}). 
First, consider the case $g(u)$ is not a constant. Use the product estimate \eqref{product-F} 
to obtain
\be\label{prod-theta-1}
\begin{aligned}
\|g(u)\,\prod_{k = 1}^r (\partial^k u)^{\alpha_k}\|_{H^\gamma} \lesssim 
\|g(u)\|_\infty\,\|\prod_{k = 1}^r (\partial^k u)^{\alpha_k}\|_{H^\gamma} + 
      \|\prod_{k = 1}^r (\partial^k u)^{\alpha_k}\|_{2 q}\,\|g(u)\,\big|\; F^\gamma_{2q^\prime, 2}\|
\end{aligned}
\ee
where 
\be
q = \frac{s}{r}\,,\quad q^\prime = \frac{s}{s - r}\;.
\ee
Since $\|g(u)\|_\infty \lesssim C_g(\|u\|_\infty)\,\|u\|_\infty$, the first term on the right will be treated later (with the case $g(\cdot) = const$). Thus, look at the second term.
The choice of $q$ is dictated by the following computation. First, use H\"older's inequality,
\[
\|\prod_{m = 1}^r (\partial^m u)^{\alpha_m}\|_{2 q} \le \prod_{m = 1}^r\,\|\partial^m u\|_{2 q \alpha_m p_m}^{\alpha_m}\,. 
\]
Then, 
\[
\|\partial^m u\|_{2 q \alpha_m p_m} \le \|u\;\big|\; F^m_{2 q \alpha_m p_m, 2}\|
\]
The norm $\|u\;\big|\; F^m_{2 q \alpha_m p_m, 2}\|$ will be bounded using \eqref{runst} as follows:
\[
\|u\;\big|\; F^m_{2 q \alpha_m p_m, 2}\| \le \|u\;\big|\;F^s_{2, 2}\|^{\theta_m}\,\|u\|_\infty^{1 - \theta_m}\;.
\]
This means the indices should satisfy
\[
m = s\,\theta_m,\quad 2 q\,\alpha_m p_m = \frac{2}{\theta_m}
\]
and $1\le p_m \le\infty$, 
\[
\sum_m \frac{1}{p_m} = 1\,.
\]
Consequently, we must have 
\[
\theta_m = \frac{m}{s}
\]
Then 
\[
2 q \alpha_m p_m = \frac{2s}{m} \quad\Leftrightarrow\quad m\,\alpha_m = \frac{2s}{2 q}\,\frac{1}{p_m} 
\]
Since $\sum m \alpha_m = r$ and we want $\sum 1/p_m = 1$, we must have $q = s/r$.

\medskip

To estimate the norm $\|g(u)\,\big|\; F^\gamma_{2q^\prime, 2}\|$ apply Lemma \ref{chain-1}: 
\[
\|g(u)\,\big|\; F^\gamma_{2q^\prime, 2}\|\lesssim C(\|u\|_\infty)\,\|u\,\big|\; F^\gamma_{2q^\prime, 2}\|
\]
By Runst's Lemma~\ref{Runst}, 
\[
\|u\,\big|\; F^\gamma_{2q^\prime, 2}\| \lesssim 
\|u\;\big|\; F^s_{2, 2}\|^{1 - \frac{r}{s}}\,\|u\|_\infty^{\frac{r}{s}}\,.
\]
Collecting the estimates, we get
\[
\begin{aligned}
& \|\prod_{m = 1}^r (\partial^m u)^{\alpha_m}\|_{2 q}\,\|g(u)\,\big|\; F^\gamma_{2q^\prime, 2}\|\lesssim \\ 
& C(\|u\|_\infty)\,\|u\,\big|\;H^s\|^{1 - \frac{r}{s}}\,\|u\|_\infty^{\frac{r}{s}}\;\prod_{m=1}^r \|u\,\big|\;H^s\|^{\alpha_m \theta_m}\,\|u\|_\infty^{\alpha_m (1 - \theta_m)} = \\ 
& C(\|u\|_\infty)\,\|u\|_\infty^{\sum \alpha_m }\;\|u\|_{H^s}\,,
\end{aligned}
\]
where we have used that, by construction, 
\[
\sum_{m=1}^r \alpha_m\,\theta_m = \frac{r}{s}\,.
\]
Finally, consider the case $g(u) \equiv const$. To estimate  $\|\prod_m (\partial^m u)^{\alpha_m}\;\big|\;{H^\gamma}\| $
use \eqref{product-many}:
\[
\|\prod_m (\partial^m u)^{\alpha_m}\|_{H^\gamma} \lesssim \sum_m \|(\partial^m u)^{\alpha_m}\;\big|\; F^\gamma_{2 q_m, 2}\|\,\prod_{j\neq m} 
\|(\partial^j u)^{\alpha_j}\;\big|\; L^{2q_j}\|\,,
\]
where $q_1, \dots, q_r$ will be chosen later (and will satisfy $\sum_j 1/q_j = 1$). 
The factors in the product  are bounded first as follows:
\[
\|(\partial^j u)^{\alpha_j}\;\big|\; L^{2q_j}\| \lesssim \|u\;\big|\;F^j_{2\alpha_j q_j, 2}\|^{\alpha_j}\,,
\]
and then Runst's lemma is applied:
\[
\|u\;\big|\;F^j_{2\alpha_j q_j, 2}\| \lesssim \|u\;\big|\;F^s_{2,2}\|^{\beta_j}\,\|u\|_\infty^{1 - \beta_j}\,.
\]
Lemma \ref{Runst} imposes the following restrictions
\[
j = \beta_j\,s\,,\quad 2 \beta_j\,\alpha_j\,q_j = 2 \,,
\]
i.e.,
\[
\beta_j = \frac{j}{s}\,,\quad \frac{1}{q_j} = \alpha_j\,\frac{j}{s}\,.
\]
Thus,
\[
\|(\partial^j u)^{\alpha_j}\;\big|\; L^{2q_j}\| \lesssim \|u\;\big|\;H^s\|^{j\,\alpha_j/s}\,\|u\|_\infty^{\alpha_j(1 - j/s)}\,,
\]
for $j\neq m$. 

Next, consider $ \|(\partial^m u)^{\alpha_m}\;\big|\; F^\gamma_{2 q_m, 2}\|$.  
Assuming $\alpha_m > 0$, first use \eqref{product-many} with equal exponents:
\begin{align*} 
\|(\partial^m u)^{\alpha_m}\;\big|\; F^\gamma_{2 q_m, 2}\| \lesssim{}& 
\sum_{k = 1}^{\alpha_m} \|(\partial^m u)\;\big|\; F^\gamma_{2\alpha_m q_m, 2}\|\,
\prod_{j = 1}^{\alpha_m - 1}\|(\partial^m u)\;\big|\;L^{2\alpha_m q_m}\| \\ 
={}& \alpha_m \,\|(\partial^m u)\;\big|\; F^\gamma_{2\alpha_m q_m, 2}\|\,\|(\partial^m u)\;\big|\;L^{2\alpha_m q_m}\|^{\alpha_m -1} 
\end{align*} 
and continue as follows:
\[
\lesssim \|u\;\big|\; F^{m +\gamma}_{2\alpha_m q_m, 2}\|\,\|u \;\big|\;F^m_{2\alpha_m q_m}\|^{\alpha_m -1} 
\]
Now each norm is bounded using \eqref{runst}. We have
\[
\|u\;\big|\; F^{m +\gamma}_{2\alpha_m q_m, 2}\| \lesssim \|u\,\big|\; H^s\|^\lambda\,\|u\|_\infty^{1 - \lambda}\,,
\]
where
\[
\lambda = \frac{m + \gamma}{s}
\]
and $q_m$ must be chosen so that
\[
\frac{1}{q_m} = \alpha_m\,\frac{m + \gamma}{s}\,.
\]
Similarly, 
\[
\|u \;\big|\;F^m_{2\alpha_m q_m}\| \lesssim \|u\,\big|\; H^s\|^{\lambda_m}\,\|u\|_\infty^{1 - \lambda_m}\,,
\]
where $q_m$ is as above and 
\[
\lambda_m = \frac{m}{s}\,.
\] 
Bringing the estimates together, we obtain
\[
\|\prod_m (\partial^m u)^{\alpha_m}\;\big|\;{H^\gamma}\| \lesssim {\tilde C}(\|u\|_\infty)\,\|u\|_{H^s}\,,
\]
as claimed.

\end{proof}

\section{Estimates on Riesz transforms} \label{sec:riesz-est} 
The Riesz transform operators ${\cal R}_j$, $j = 1, \dots, d$, will be defined by the formula
\[
{\cal R}_j f = {\cal F}^{-1} \frac{\kappa^j}{|\kappa|} {\cal F} f\,.
\]
It is well known (see \cite[Theorem 10.2.1]{Meyer-Coifman}) that ${\cal R}_j$ is bounded in ${\dot B}^s_{\infty, \infty}$, when $s > 0$, but not 
in  ${ B}^s_{\infty, \infty}$. The Riesz transform is a special case of a more general class of pseudo-differential operators which we shall now discuss. 

Denote by $S^0_{ph}$ the space of all smooth functions $a: \R^d\setminus\{0\}\to \R$ which are positively homogeneous of degree $0$ and satisfy 
\[
|\partial^\alpha_\kappa \,a(\kappa)|\le C_\alpha \,|\kappa|^{-|\alpha|}
\]
for any multiindex $\alpha$.  Each symbol $a\in S^0_{ph}$ gives rise to a pseudodifferential operator ${\mathfrak a}:\,f\mapsto {\mathfrak a}[f]$, where 
\[
{\mathfrak a}[f](x) = {\cal F}^{-1}_{\kappa\to x} a {\cal F} f = \int e^{i x \kappa} a(\kappa)\,{\hat f}(\kappa)\,\dbar\kappa\,.
\]
Any finite composition of the Riesz transforms has its symbol in $S^0_{ph}$. 
An observation: if, as usual,  $q$ and $q^\prime$ are conjugate exponents, and 
$ 1\le q^\prime \le 2\le  q \le\infty$, then (change of variables)
\be\label{unif-Lp}
 \|\int e^{i y \kappa} \vp(2^{-n} \kappa)\,a(\kappa)\,\dbar\kappa\;\big|\;L^{q^\prime}(dy) \| =  
2^{n \frac{d}{q}}\,\|{\cal F}^{-1} \left(\vp\cdot a\right)\;\big|\; L^{q^\prime}\|\,.
\ee
In particular, the integrals 
$
 \int e^{i k y} \vp_n(k)\,a(k)\,\dbar k 
$
represent smooth functions with uniformly in $n$ bounded $L^1$ norms.

\begin{lemma}\label{HA} 
Let the symbol $a\in S^0_{ph}$ be given. 
\begin{enumerate}
\item For any $r \in \R$ and any $p\in [2, +\infty]$,  
\be\label{rho}
\|{\mathfrak a}[f]\;\big|\; {\dot B}^r_{p, p}\|\lesssim \|f\;\big|\; {\dot B}^r_{p, p}\|\,.
\ee
\item \label{point:D1:2} Assume the parameters $r, p$, and $q$ satisfy the conditions $1\le p, q \le \infty$ and $r > \frac{d}{p}$. 
Then there exists a constant 
$C_1$ such that  
\be \label{eq:D.3}
\|{\mathfrak a}[f]\|_\infty \le C_1\; [f]_{-1}^{\gamma_1}\;\|f\;\big|\;{\dot B}^r_{p, q}\|^{1 - \gamma_1} 
\ee
for every $f\in {\dot H}^{-1}\cap {\dot B}^r_{p, q}$, 
where 
\be\label{gamma1}
\gamma_1 = \gamma_1(d, r, p) = \frac{r - \frac{d}{p}}{r + 1 + d\,\frac{p-2}{2p}}\,.
\ee
\item Assume the parameters $r, p$, and $q$ satisfy the conditions $1\le p, q \le \infty$ and $r > \frac{d}{p} - 1$.
Then  there exists a constant $C_2 > 0$ such that  
\be
\|D^{-1}\,{\mathfrak a}[f]\|_\infty \le C_2\; [f]_{-1}^{\gamma_2}\;\|f\;\big|\;{\dot B}^r_{p, q}\|^{1 - \gamma_2} 
\ee
for every $f\in {\dot H}^{-1}\cap {\dot B}^r_{p, q}$, where 
\be\label{gamma2}
\gamma_2 = \gamma_2(d, r, p) = \frac{r + 1 - \frac{d}{p}}{r + 1 + d\,\frac{p-2}{2p}}
\ee
\end{enumerate}
\end{lemma}
\begin{proof} To prove the first claim we need to show that 
\[
\sum_{n\in \Z}\,2^{r n p}\,\|\vp_n(D) {\mathfrak a}[f]\|_p^p \lesssim \sum_{m\in \Z}\,2^{r m p}\,\| f_m\|_p^p\,,
\]
where $f_m = \vp_m(D) f$. 
This follows from the estimate
\be\label{base-Bes}
\|\vp_n(D) {\mathfrak a}[f]\|_p \lesssim \|f_n\|_p\,,
\ee
which is easy to prove. Indeed,
\[
\left(\vp_n(D) {\mathfrak a}[f]\right)(x) = \sum_{\ell = - 1}^{1} \left(\vp_{n+\ell}(D) {\mathfrak a}[\vp_{n}(D)f]\right)(x)
\]
and
\[
\left(\vp_{n+\ell}(D) {\mathfrak a}[\vp_{n}(D)f]\right)(x) = \int\int e^{i k (x - y)} \vp_{n+\ell}(k)\,a(k)\,\dbar k\;f_{n}(y)\,dy\,.
\]
Then (see \eqref{unif-Lp})
\[
\|\vp_{n+\ell}(D) {\mathfrak a}[\vp_{n}(D)f]\,\|_p \le \|\int e^{i k y} \vp_{n+\ell}(k)\,a(k)\,\dbar k\;\big|\; L^1(dy)\|\;\|f_{n}\|_p \lesssim \|f_{n}\|_p\,.
\]
Thus we have \eqref{base-Bes}, and this implies \eqref{rho}.

To prove the remaining two claims, 
consider the expansion \eqref{LP} for $f$, $f = g_N + h_N$, where
\[
g_N = {\cal F}^{-1}\psi_N(\kappa)\,{\hat f}(\kappa)\quad \text{and}\quad h_N = {\cal F}^{-1}\sum_{n = N+1}^\infty \vp_n(\kappa) \,{\hat f}(\kappa)\,.
\]
the integer $N$ will be chosen later. We have
\[
\|{\mathfrak a}[g_N]\|_\infty \lesssim \|\psi(2^{-N}\kappa)\,a(\kappa)\,{\hat f}(\kappa)\|_{L^{1}(d\kappa)}\le [f]_{-1}\;\|\psi(2^{-N}\kappa)\,a(\kappa)\,|\kappa|\|_{L^{2}(d\kappa)} 
\]
and
\[
\|\psi(2^{-N}\kappa)\,a(\kappa)\,|\kappa|\|_{L^{2}(d\kappa)} = 2^{N(1 + \frac{d}{2})}\,\|\psi(\kappa)\,a(\kappa)\,|\kappa|\|_{L^{2}(d\kappa)}\,.
\] 
Therefore,
\be\label{g-11}
\|{\mathfrak a}[g_N]\|_\infty \lesssim 2^{N(1 + \frac{d}{2})}\,[f]_{-1}\;.
\ee
Next,
\[
\begin{aligned}
& {\mathfrak a}[h_N] = {\cal F}^{-1}\sum_{n = N+1}^\infty \vp_n(\kappa)\,a(\kappa) \,{\hat f}(\kappa) = {\cal F}^{-1}\sum_{n = N+1}^\infty \sum_{\ell = -1}^1 \vp_{n+\ell}(\kappa)\,a(\kappa)  \, \vp_{n}(\kappa)\,{\hat f}(\kappa) \\ 
& = \sum_{\ell = -1}^1 \sum_{n = N+1}^\infty {\cal F}^{-1}\left[\vp_{n+\ell}(\kappa)\,a(\kappa) \, {\hat f}_{n}(\kappa)\right] 
=  \sum_{\ell = -1}^1 \sum_{n = N+1}^\infty {\cal F}^{-1} \vp_{n+\ell} a\,{\cal F}  f_{n}\,,
\end{aligned}
\]
and hence, 
\[
\|{\mathfrak a}[h_N]\|_\infty \le \sum_{\ell = -1}^1 \sum_{n = N+1}^\infty \|{\cal F}^{-1} \vp_{n+\ell}\, a\,{\cal F}  f_{n}\|_\infty\,.
\]
Pick a $p\ge 1$ and observe that
\[
\begin{aligned}
& |\;[{\cal F}^{-1} \vp_j \,a\,{\cal F}  f_{n}](x)\;| = |\int \left(\int e^{i (x - y) \kappa} \vp(2^{-j} \kappa)\,a(\kappa)\,\dbar\kappa \right)\,f_{n}(y)\,dy| \le \\ 
& \|\int e^{i (x - y) \kappa} \vp(2^{-j} \kappa)\,a(\kappa)\,\dbar\kappa\|_{L^{p^\prime}(dy)}\,\|f_{n}\|_p\,.
\end{aligned}
\]
Applying \eqref{unif-Lp}, we obtain  
\be\label{temp-11-D}
\|{\mathfrak a}[h_N]\|_\infty \lesssim_p \sum_{n = N}^\infty 2^{n \frac{d}{p}} \;\|f_n\|_p\;.
\ee
If 
$
r > \frac{d}{p}\,,
$
then, for any $q\ge 1$,  
\begin{align*}
\sum_{n = N}^\infty 2^{n \frac{d}{p}} \;\|f_n\|_p ={}& \sum_{n = N}^\infty 2^{n (\frac{d}{p} - r)} \;2^{n r}\|f_n\|_p \le  
 \left(\sum_{n = N}^\infty 2^{n (\frac{d}{p} - r) q^\prime}\right)^{1/q^\prime}\;\left(\sum_{m = N}^\infty 2^{n r q}\|f_n\|^q_p\right)^{1/q} \\
={}& 
2^{N (\frac{d}{p} - r)}\,\left(1 - 2^{(\frac{d}{p} - r) q^\prime}\right)^{1/q^\prime} \;\|f\;\big|\; {\dot B}^r_{p, q}\|
\end{align*}
Thus,
\be\label{temp-12-D}
\|{\mathfrak a}[h_N]\|_\infty \lesssim_{p, q, d, r}\,2^{N (\frac{d}{p} - r)}\;\|f\;\big|\; {\dot B}^r_{p, q}\|
\ee
Combine this with \eqref{g-11} to obtain
\be\label{temp-13}
\|{\mathfrak a}[f]\|_\infty \lesssim  2^{N(1 + \frac{d}{2})}\,[f]_{-1} + 2^{N (\frac{d}{p} - r)}\;\|f\;\big|\; {\dot B}^r_{p, q}\|\,.
\ee
Now choose $N$ so that the two terms have (almost) the same magnitudes:
\[
N = \lfloor\; \log_2 \left( \frac{\|f\;\big|\; {\dot B}^r_{p, q}\|}{[f]_{-1}} \right)^{1/(1 + r + d \frac{p-2}{2p})}\;\rfloor 
\]
( $\lfloor\cdot\rfloor$ is the floor function).
The resulting inequality is
\be\label{refin-riesz}
\|\rho[f]\|_\infty \lesssim [f]_{-1}^{\gamma_1(d, p, r)}\;\cdot\;\|f\;\big|\; {\dot B}^r_{p, q}\|^{1 - \gamma_1(d, p, r)}  
\ee
with 
\[
\gamma_1(d, p, r) = \frac{r - \frac{d}{p}}{r + 1 + d\,\frac{p-2}{2p}}\,.
\]
The second part is proved similarly. First we have
\[
\|D^{-1}{\mathfrak a}[g_N]\|_\infty \lesssim \|\psi(2^{-N}\kappa)\,a(\kappa)\,\frac{1}{|\kappa|}\,{\hat f}(\kappa)\|_{L^{1}(d\kappa)}\le [f]_{-1}\;\|\psi(2^{-N}\kappa)\,a(\kappa)\|_{L^{2}(d\kappa)}\lesssim 2^{N\,\frac{d}{2}}\,[f]_{-1}\,. 
\]
After that, consider $\|D^{-1}{\mathfrak a}[h_N]\|_\infty$. We have
\[
\|D^{-1}{\mathfrak a}[h_N]\|_\infty \le \sum_{\ell = -1}^1 \sum_{n = N+1}^\infty \|D^{-1}{\cal F}^{-1} \vp_{n+\ell}\, a\,{\cal F}  f_{n}\|_\infty\,.
\]
Now,
\[
\begin{aligned}
& |\;[D^{-1}{\cal F}^{-1} \vp_j \,a\,{\cal F}  f_{n}](x)\;| = |\int \left(\int e^{i (x - y) \kappa} \vp(2^{-j} \kappa)\,a(\kappa)\,\frac{1}{|\kappa|}\,\dbar\kappa \right)\,f_{n}(y)\,dy| \le \\ 
& \|\int e^{i (x - y) \kappa} \vp(2^{-j} \kappa)\,a(\kappa)\,\frac{1}{|\kappa|}\,\dbar\kappa\|_{L^{p^\prime}(dy)}\,\|f_{n}\|_p\,,
\end{aligned}
\]
and 
\[
\begin{aligned}
& \|\int e^{i (x - y) \kappa} \vp(2^{-j} \kappa)\,a(\kappa)\,\frac{1}{|\kappa|}\,\dbar\kappa\|_{L^{p^\prime}(dy)} = 
2^{j (\frac{d}{p} - 1)}\,\|D^{-1}{\cal F}^{-1} \left(\vp\cdot a\right)\|_{p^\prime}\,= C\,2^{j (\frac{d}{p} - 1)}.
\end{aligned}
\]
This leads to
\[
\begin{aligned}
\|D^{-1}\rho[h_N]\|_\infty \lesssim \sum_{n = N}^\infty 2^{n (\frac{d}{p} - 1)}\;\|f_{n}\|_p
\end{aligned}
\]
and we proceed as in the first part. Now the restriction on $r$ will be 
$
r > \frac{d}{p} - 1
$
and then
\begin{align*}
\sum_{n = N}^\infty 2^{n (\frac{d}{p} - 1)}\;\|f_{n}\|_p \le{}& \left(\sum_{n = N}^\infty 2^{n (\frac{d}{p} - 1 - r) q^\prime}\right)^{1/q^\prime}\,\left(\sum_{m = N}^\infty 2^{n r q}\,\|f_{n}\|_p^q\right)^{1/q}\\
 ={}& 
2^{N (\frac{d}{p} - 1 - r)}\,\left(1 - 2^{(\frac{d}{p} - 1 - r) q^\prime}\right)^{1/q^\prime}\,\|f\;\big|\;{\dot B}^r_{p, q}\|
\end{align*}
This time we obtain
\[
\|D^{-1} {\mathfrak a}[f]\|_\infty \lesssim  2^{N\,\frac{d}{2}}\,[f]_{-1} + 2^{N (\frac{d}{p} - 1 - r)}\,\|f\;\big|\;{\dot B}^r_{p, q}\|\,.
\]
Again, choose $N$ so that
\[
2^{N(1 + r + d\,\frac{p-2}{2p})} \approx \frac{\|f\;\big|\;{\dot B}^r_{p, q}\|}{[f]_{-1}}\,.
\]
i.e.,
\[
N = \lfloor\; \log_2\left(\|f\;\big|\;{\dot B}^r_{p, q}\|/[f]_{-1}\right)^{1/(1 + r + d\, \frac{p-2}{2p})}\;\rfloor\,.
\]
Plugging this value into the inequality above we obtain
\[
\|D^{-1} \rho[f]\,\|_\infty \lesssim [f]_{-1}^{\gamma_2}\;\|f\;\big|\;{\dot B}^r_{p, q}\|^{1 - \gamma_2}
\]
with
\[
\gamma_2 = \gamma_2(d, p, r)= \frac{1 + r - \frac{d}{p}}{1 + r + d \,\frac{p-2}{2p}}\;.
\]
\end{proof}

\section{
Strichartz estimates} \label{sec:stricharts-app}

In this section we derive the homogeneous $\R^d$ version of the Strichartz estimates in Theorem 2 \cite{Kap-1}. We work with functions of $x\in\R^d$ and the Fourier 
variables are $\kappa$. Note, that the estimates are applied in Section \ref{sec:below}  to functions of $\xi$ with the dual Fourier variables $k$. 

Consider the Cauchy problem
\be\label{wave-eq}
\partial_t w = i D w + f\,,\quad w(0) = w_0\,,
\ee
where  $D = \sqrt{- \Delta}$ in $\R^d$. Given $w_0\in {\dot H}^\theta$ and $f\in L^1_{loc}(\R\to {\dot H}^\theta)$,  there exists 
a unique solution $w$ of equation \eqref{wave-eq} such that $w(t)$ is a continuous functions of $t$ with values in ${\dot H}^\theta$, and 
\be
\sup_{[0, T]} \|w(t)\;\big|\;{\dot H}^\theta\| \le C_\theta \,\left(\|w_0\;\big|\; {\dot H}^\theta\| + \int_0^T \|f(t)\;\big|\; {\dot H}^\theta\|\,dt\right)
\ee 
with the positive constant $C_\theta$ independent of the particular choice of $w_0$ and $f$. 

\begin{theorem}\label{Strichartz} 
Let $2\le q \le \infty$ and let $p$ be such that if $d = 2$, then 
$2 \le p \le +\infty$, and if $d \ge 3$, then $2 \le p < \frac{2(d-1)}{d - 3}$.  
Assume the parameters $r$ and $\theta$ satisfy the conditions
\be
r\in\R,\quad \theta = r + \frac{d+1}{4}\,\frac{p-2}{p}\,.
\ee
There exists a constant $C > 0$ (dependent on $d$, $r$, and $p$) such that, for any $T > 0$, the following estimate is true 
for the solutions of \eqref{wave-eq}:
\be\label{Str}
\left(\int_0^T \|w(t)\;\big|\;{\dot B}^r_{p, q}\|^{\frac{4p}{(d-1)(p-2)}}\,dt\right)^{\frac{(d-1)(p-2)}{4p}} \le 
C\,\left(\|w_0\;\big|\; {\dot H}^\theta\| + \int_0^T \|f(t)\;\big|\; {\dot H}^\theta\|\,dt\right)
\ee
\end{theorem}

\begin{corollary}\label{cor:strichartz}
Two special cases: If $d = 2$, then take $p = q = \infty$, $0 < r < 1$, and $\theta = r + \frac34$, and obtain
\be\label{str-d2}
\left(\int_0^T \|w(t)\;\big|\;{\dot B}^r_{\infty, \infty}\|^4\,dt\right)^{1/4} \le 
C\,\left(\|w_0\;\big|\; {\dot H}^\theta\| + \int_0^T \|f(t)\;\big|\; {\dot H}^\theta\|\,dt\right)\,.
\ee
If $d = 3$, take $2\le p = q < \infty$, and  $\theta = r + \frac{p-2}{p}$, and obtain
\be\label{str-d3}
\left(\int_0^T \|w(t)\;\big|\;{\dot B}^r_{p, p}\|^{\frac{2p}{p-2}}\,dt\right)^{\frac{p-2}{2p}} \le 
C\,\left(\|w_0\;\big|\; {\dot H}^\theta\| + \int_0^T \|f(t)\;\big|\; {\dot H}^\theta\|\,dt\right)
\ee
\end{corollary} 

\begin{proof}[Proof of Theorem \ref{Strichartz}]
For fixed $w_0$ and $f$, the solution of the problem \eqref{wave-eq} is given by the formula
\[
w(t) = e^{i t\,D} w_0 + \int_0^t e^{i (t - \tau) D} f(\tau)\,d\tau\,.
\]
Each term on the right can be analyzed separately. Take a $g\in {\cal S}$ and consider $e^{i t\,D}g$. 
In fact, we need to look at the dyadic pieces $\vp_n(D)e^{i t\,D}g = e^{i t\,D}g_n$, where $g_n = \vp_n(D) g$. Since 
\[
e^{i t\,D}g_n = \sum_{j = n-1}^{n+1} \vp_{j}(D)\,e^{i t\,D}g_n\,,
\]
we examine the terms $\vp_j(D)\,e^{i t\,D}g_n$. The first important observation in the Strichartz analysis is 
the following bound on the operator $\vp_j(D)\,e^{i t\,D}$ as an operator from $L^{p^\prime}$ to  $L^p$, $2\le p\le \infty$:
\be\label{est:LpLpprime}
\| \vp_j(D)\,e^{i t\,D}\|_{L^{p^\prime}(\R^d) \to L^p(\R^d)} \lesssim \frac{1}{|t|^{(d-1)\,\frac{p-2}{2p}}}\,2^{j (d+1)\,\frac{p-2}{2p}}\,.
\ee
As a corollary, 
\be\label{n-LpLpprime}
\|e^{i t\,D}g_n \|_p \lesssim \frac{1}{|t|^{(d-1)\,\frac{p-2}{2p}}}\,2^{n (d+1)\,\frac{p-2}{2p}}\,\|g_n\|_{p^\prime}\,,
\ee
for any $p\in [2, +\infty]$. 
This estimate implies the following estimates in the homogeneous Besov spaces:
\be\label{besov-1}
\|e^{i t\,D}g\;\big|\; {\dot B}^r_{p, q}\| \lesssim \frac{1}{|t|^{(d-1)\,\frac{p-2}{2p}}}\;\|g\;\big|\;
{\dot B}^{r + (d+1)\,\frac{p-2}{2p}}_{p^\prime, q}\|\,,
\ee
$r\in \R$ and $1\le q \le \infty$.
The second important observation is the space-time estimate for the homogeneous term $e^{i t\,D} g$. The argument uses duality. Let $h(t, x)$ be a sufficiently smooth function. Then
\[
\begin{aligned}
& \int_0^T\langle e^{i t\,D} g\,,\;h(t)\rangle \,dt = \int_0^T\langle e^{i t\,D} g\,,\;h(t)\rangle \,dt = 
\int_0^T\langle g\,,\;e^{- i t\,D} h(t)\rangle \,dt = \langle g\,,\;\int_0^T e^{- i t\,D} h(t)\,dt\rangle \\ 
& \le \|g\;\big|\;{\dot H}^\theta\|\,\|D^{-\theta}\int_0^T e^{- i t\,D} h(t)\,dt\|_2
\end{aligned}
\]
Now, 
\[
\begin{aligned}
& \|D^{-\theta}\int_0^T e^{- i t\,D} h(t)\,dt\|_2^2 = \int_0^T \int_0^T \langle e^{- i t\,D}D^{-\theta}  h(t), e^{- i t^\prime\,D} D^{-\theta}h(t^\prime)\rangle \,dt^\prime\,dt = \\ 
&\int_0^T \int_0^T \langle e^{- i (t^\prime - t)\,D} \,D^{-2\theta}  h(t), h(t^\prime)\rangle\,dt^\prime\,dt \\ 
& \le \int_0^T \int_0^T \| e^{ i (t - t^\prime)\,D} \,D^{-2\theta}  h(t)\;\big|\; {\dot B}^r_{p, q}\|\cdot\|h(t^\prime)\;\big|\;{\dot B}^{-r}_{p^\prime, q^\prime}\|\;dt^\prime\,dt  \quad\text{(use \eqref{besov-1})} \\ 
& \lesssim 
\int_0^T \int_0^T \frac{1}{|t - t^\prime|^{(d-1)\,\frac{p-2}{2p}}}\;\|h(t)\;\big|\; {\dot B}^{r+(d+1)\,\frac{p-2}{2p}  - 2\theta}_{p^\prime, q}\|
\cdot\|h(t^\prime)\;\big|\;{\dot B}^{-r}_{p^\prime, q^\prime}\|\;dt^\prime\,dt
\end{aligned}
\]
Choose $\theta$ so that 
\[
r + (d+1)\,\frac{p-2}{2p} - 2\theta = - r
\]
i.e.,
\be\label{thetar}
\theta = r + \frac{d+1}{4}\,\frac{p-2}{p} 
\ee
and assume that 
\[
q \ge 2
\]
(then $q^\prime \le 2 \le q$ and ${\dot B}^{-r}_{p^\prime, q^\prime} \subset {\dot B}^{-r}_{p^\prime, q}$). Then 
\begin{align*} 
\|D^{-\theta}\int_0^T e^{- i t\,D} h(t)\,dt\|_2^2 \lesssim{}& \int_0^T \int_0^T \frac{\|h(t)\;\big|\; {\dot B}^{- r}_{p^\prime, q^\prime}\|
\cdot\|h(t^\prime)\;\big|\;{\dot B}^{-r}_{p^\prime, q^\prime}\|}{|t - t^\prime|^{(d - 1)\,\frac{p-2}{2p}}}\;dt^\prime\,dt \\
\lesssim{}& \left(\int_0^T\|h(t)\;\big|\;  {\dot B}^{-r}_{p^\prime, q^\prime}\|^m\right)^{1/m}
\end{align*} 
with 
\[
m = \frac{1}{1 - \frac{d-1}{4}\,\frac{p-2}{p}}\,.
\]
The last inequality is a consequence of
the Hardy-Littlewood-Sobolev inequality (see \cite[Theorem 4.5.3]{H}). It requires the following restrictions on $p$: 
\be\label{restrict-p}
0 \le \frac{d-1}{2}\,\frac{p-2}{p} < 1\,,
\ee 
i.e., if $d = 2$, then 
$2 \le p \le +\infty$, and if $d \ge 3$, then $2 \le p < \frac{2(d-1)}{d - 3}$. 
In any case, 
\be\label{test-1}
|\int_0^T\langle e^{i t\,D} g\,,\;h(t)\rangle \,dt|\lesssim \|g\;\big|\;{\dot H}^\theta\|\cdot \left(\int_0^T\|h(t)\;\big|\;  {\dot B}^{-r}_{p^\prime, q^\prime}\|^m\right)^{1/m}
\ee
Hence, by duality, 
\be
\left(\int_0^T \|e^{i t D} g\;\big|\; {\dot B}^r_{p, q}\|^{\frac{4 p}{(d-1)(p-2)}} \right)^{\frac{(d-1)(p-2)}{4 p}} \lesssim 
\|g\;\big|\; {\dot H}^\theta\|\,,
\ee
where the exponent $\frac{4 p}{(d-1)(p-2)}$ is the conjugate of $m$. This proves estimate \eqref{Str} for the homogeneous equation \eqref{wave-eq}. The estimate on $\int_0^t e^{i (t - \tau) D} f(\tau)\,d\tau$ is obtained as in the second part of the proof of Theorem 2 in \cite{Kap-1} 
(with necessary slight modifications). 

\end{proof}

\section{
Norms in Euler and Lagrange coordinates} \label{sec:Eul-Lag-App}

In this section $\xi \to x(\xi)$ is a $C^1$ volume preserving diffeomorphism of the form $x(\xi) = A \xi + \vp(\xi)$ satisfying the assumptions of Lemma~\ref{difff}. In particular, $A$ is a constant $SL(d, \R)$ matrix and $\vp\in H^{s+1}(\R^d_\xi)$ with $s > d/2$. 
 This transformation from the Lagrangian coordinates, $\xi$, to the Eulerian coordinates, $x$, pushes back the functions $f( x)$ to the functions ${\tilde f}( \xi) = f(x(\xi))$.  This is a linear isometry  from $L^p(\R^d, dx)$ to $L^p(\R^d, d\xi)$, $1\le p \le \infty$. To analyze other norms, we use the superscripts $L$ and $E$ on functions to indicate the coordinate system used. Denote by $v^i_a(x)$ the entries of the Jacobian matrix, $\partial x/\partial \xi$, expressed in Eulerian coordinates, i.e., as functions of $x$. We have $v^i_a(x) = A^i_a + u^i_a(x)$, where $u^i_a \in H^s(\R^d_x)$. 
Notation $v_a$ or $u_a$ is used to represent a generic $v_a$ or $u_a$, or when in a norm, 
the maximal over $a$ norm, e.g., $\|v_a\|_\infty = \max_a \|v_a\|_\infty$. For the norms in $L^p$ we skip the superscripts $E$ and $L$. 
  
\subsection{Homogenous Besov and Sobolev spaces}  
The following proposition contains inequalities between the homogeneous Besov norms $\{g\}_{r,p}$ and between the homogeneous $L^2$ Sobolev  norms $[g]_\theta$ in the Euler and Lagrange coordinates. The range of $r, p$, and $\theta$ is restricted to the demands of the main body of the paper. All the spaces are over $\R^d$, $d \ge 2$.

\begin{lemma}\label{E-L-E}
\ 
\begin{enumerate}
\item[a)] Assume $0 < r < 1$, $0 \le \theta\le 1$, and $2\le p \le \infty$. Then, for all $g\in {\cal S}(\R^d)$,
\begin{align} 
& \{g^L\}_{r,p} \lesssim \| v_a\|_\infty^r\;\{g^E\}_{r,p}\,,\quad & \{g^E\}_{r,p} \lesssim \| v_a\|_\infty^{(d-1)r}\;\{g^L\}_{r,p}\,,\label{ineq1} \\  
& [g^L]_\theta \lesssim \| v_a\|_\infty^\theta\;[g^E]_\theta\,,\quad 
& [g^E]_\theta \lesssim \| v_a\|_\infty^{(d - 1)\theta}\;[g^L]_\theta \label{ineq2}
\end{align}
\item[b)] Assume $1<\theta < 2$, then 
\be\label{ineq3}
[g^L]_\theta \lesssim  \| v_a\|_\infty^{\theta-1}\;\left(\| v_a\|_\infty + \|u_a\|_\infty^{2 - \theta}\;\|u_a\,\big|\;{\dot F}^1_{d, 2}(\R^d)\|^{\theta - 1}\right)\;[g^E]_\theta\,.
\ee
\end{enumerate}
\end{lemma}
\begin{proof} We first prove the inequalities for Sobolev norms in part a). Compare  the norms  $\|g^L\;\big|\;{\dot H}^1(\R^d_\xi)\|$ and $\|g^E\;\big|\;{\dot H}^1(\R^d_x)\|$. We have
\begin{align*} 
\|g^L\;\big|\;{\dot H}^1(\R^d_\xi)\|^2 ={}& \int |\nabla_\xi {\tilde g}(\xi)|^2\,d\xi \\
\le{}& \int |\frac{\partial f(t, x(\xi))}{\partial x^j}\;\frac{\partial x^j(\xi)}{\partial \xi}|^2\,d\xi \\\le{}& \|v_a(t)\|_\infty^2\;\int |\frac{\partial g(x(\xi))}{\partial x^j}|^2\,d\xi \\
={}& 
\|v_a(t)\|_\infty^2\;\int |\frac{\partial g(x)}{\partial x^j}|^2\,dx \\
={}& \|v_a(t)\|_\infty^2\;\|g^E\;\big|\;{\dot H}^1(\R^d_x)\|^2\,.
\end{align*} 
If $0 < \theta < 1$, then ${\dot H}^\theta$ is an interpolation space between $L^2$ and ${\dot H}^1$, 
${\dot H}^\theta = (L^2, {\dot H}^1)_{\theta, 2}$ (see \cite[Theorem 6.3.1]{MR0482275}). This explains why $[g^L]_\theta \lesssim \|v_a\|_\infty^\theta \,[g^E]_\theta$. Similarly, 
$[g^E]_\theta \lesssim \|v_a\|_\infty^{(d - 1)\theta} \,[g^L]_\theta$, where $(d - 1)$ appears because 
of the $L^\infty$ bound on $\partial \xi^a/\partial x^i$ in terms of $\|v_a\|_\infty$. 

The Besov norm inequalities in part a) can be obtained by interpolation between the inequalities for ${\dot B}^r_{2, 2} \simeq {\dot H}^r$, which we already have, and the inequalities for ${\dot B}^r_{\infty, \infty} \simeq {\dot C}^r$, the homogeneous H\"older spaces. For the H\"older seminorms we have
\begin{align*} 
\{g^L\}_r ={}& \sup_{\xi, \xi^\prime} \frac{|g^L(\xi) - g^L(\xi^\prime)|}{|\xi - \xi^\prime|^r} \\
={}& \sup_{\xi, \xi^\prime} \frac{|g^E( x(\xi)) - g^E( x(\xi^\prime))|}{|\xi - \xi^\prime|^r} \\ 
={}& \sup_{\xi, \xi^\prime} \frac{|g^E(x(\xi)) - g^E(x(\xi^\prime))|}{|x(\xi) - x(\xi^\prime)|^r}\;
 \frac{|x(\xi) - x(\xi^\prime)|^r}{|\xi - \xi^\prime|^r} \\
\lesssim{}& \sup_{x, x^\prime} \frac{|g^E(x) - g^E(x^\prime)|}{|x - x^\prime|^r}\;
\|v_a\|_\infty^r\,.
\end{align*} 
Thus, $\{g^L\}_r \le \|v_a\|_\infty^r\,\{g^E\}_r$. 
The norm $\{g^E\}_r$ is bounded similarly.

Now turn to part b). Assume $1 < \theta < 2$ and proceed with
\begin{align}\label{g-1}  
[g^L]_\theta \simeq{}& [\frac{\partial g^L}{\partial \xi}]_{\theta-1} \nonumber \\ \underset{\eqref{ineq1}}{\lesssim} {}& 
\| v_a\|_\infty^{\theta-1}\,[\left(\frac{\partial g^L}{\partial \xi}\right)^E]_{\theta-1} \nonumber \\  
\simeq{}& \| v_a\|_\infty^{\theta-1}\,
[\frac{\partial g^E}{\partial x} \cdot\frac{\partial x}{\partial \xi}]_{\theta-1} \nonumber \\ \lesssim{}&  
\| v_a\|_\infty^{\theta-1}\;  [\frac{\partial g^E}{\partial x} \cdot v_a]_{\theta-1}\,.
\end{align}
Apply the fractional product rule \eqref{product-H}: 
\[
[\frac{\partial g^E}{\partial x} \cdot v_a^E]_{\theta-1} \lesssim 
[\frac{\partial g^E}{\partial x}]_{\theta-1}\,\| v_a\|_\infty +  \|\frac{\partial g^E}{\partial x}\|_{q_1}\; \|v_a^E\,\big|\; {\dot F}^{\theta-1}_{q_2, 2}\|\,.
\]
Note that  the homogeneous Lizorkin-Triebel norms of $v_a$ and $u_a$ are the same. So, we have 
\be\label{g-2}
[\frac{\partial g^E}{\partial x} \cdot v_a^E]_{\theta-1} \lesssim 
[\frac{\partial g^E}{\partial x}]_{\theta-1}\,\| v_a\|_\infty + \| \frac{\partial g^E}{\partial x}\|_{q_1}\;\|u_a^E\,\big|\; {\dot F}^{\theta-1}_{q_2, 2}\|\,.
\ee
The parameters $q_1$ and $q_2$ 
must satisfy $2\le q_1, q_2 \le \infty$ and 
\[
\frac{1}{q_1} + \frac{1}{q_2} = \frac12\,.
\]
We choose $q_1$ and $q_2$ as follows
\be\label{q1q2}
\frac{1}{q_1} = \frac12 - \frac{\theta - 1}{d}\,,\quad \frac{1}{q_2} = \frac{\theta-1}{d} \,.
\ee
Then ${\dot H}^{\theta-1}(\R^d)\subset L^{q_1}(\R^d)$, and so
\[
\|\frac{\partial g^E}{\partial x}\|_{q_1} \lesssim [\frac{\partial g^E}{\partial x}]_{\theta-1} \lesssim [g^E]_\theta\,.
\]
As for the other factor, use \eqref{hrunst}:
\[
\|u_a^E\,\big|\; {\dot F}^{\theta-1}_{q_2, 2}(\R^d)\| \lesssim \|u_a\|_\infty^{2 - \theta}\;\|u_a^E\,\big|\;{\dot F}^1_{d, 2}(\R^d)\|^{\theta - 1}\,.
\]
Collecting the pieces we arrive at \eqref{ineq3}.

\end{proof}

\subsection{Vorticities}

Let $v$ be a vectorfield on $\R^d$ such that $\hbox{div}\,v = 0$. Denote $\omega^{mn} = \partial_m v^n - \partial_n v^m$. 
In terms of Fourier transform,
\[
{\hat \omega}^{mn} = i\,\left(\kappa^m {\hat v}^n - \kappa^n {\hat v}^m\right)\quad \text{and}\quad \kappa^n {\hat v}^n = 0\,.
\]
From $\omega^{mn}$ one recovers $v$ as follows:
\[
{\hat v}^n = - i \;\frac{\kappa^m}{|\kappa|^2}\,{\hat \omega}^{mn}\,.
\]
In our notation,
\[
v^n = - i D^{-1}{\cal R}_m \,\omega^{mn}\,.
\]
In dimension $d=2$, $\omega = \partial_1 v^2 - \partial_2 v^1$ and 
\[
{\hat v}^1 = i\,\frac{\kappa^2}{|\kappa|^2}\,{\hat \omega}\,,\quad {\hat v}^2 = - i\,\frac{\kappa^1}{|\kappa|^2}\,{\hat \omega}\,.
\]
We also have pseudovelocities $v_a$ with the components $v^i_a = u^i_a + A^i_a$, and the corresponding pseudovorticities 
$\omega_a = \partial_1 v^2_a - \partial_2 v^1_a = \partial_1 u^2_a - \partial_2 u^1_a$. 

\begin{lemma}\label{v-r}
\ 
\begin{enumerate}
\item In the case $d = 2$, 
assume that $u\in L^2$ and $\omega = \curl u\in {\dot C}^r$ for some $r\in (0, 1)$. Then $u\in L^\infty$ and $\nabla u\in L^\infty$, and the following inequalities are true:
\be\label{v-infty-r}
\|u\|_\infty  \lesssim \|u\|^{(r + 1)/(r + 2)}\;\{\omega\}_r^{1/(r + 2)}
\ee 
and
\be\label{nabla-v-infty-r}
\|\nabla u\|_\infty \lesssim \|u\|^{r/(r + 2)}\;\{\omega\}_r^{2/(r + 2)}\,.
\ee
\item In the case $d = 2$, 
assume that $u\in L^2$ and $\omega = \curl u\in {\dot H}^\theta$ for some $\theta\in(0, 1)$. Then $u\in L^\infty$  and the following inequality is true:
\be\label{v-infty-theta}
\|u\|_\infty  \lesssim \|u\|^{\theta/(\theta+1))}\;[\omega]_\theta^{1/(\theta+1)}\;.
\ee 
If $\omega\in {\dot H}^\theta$ with $\theta > 1$, then $\nabla u\in L^\infty$ and 
\be\label{nabla-v-infty-theta}
\|\nabla u\|_\infty \lesssim \|u\|^{(\theta - 1)/(\theta + 1)}\;[\omega]_\theta^{2/(\theta+1)}\,.
\ee

\item In the case $d \ge 3$, assume that $u\in L^2$ and $\omega\in {\dot B}^r_{p, p}$, where 
\[
r\in (0, 1),\;r > \frac{d}{p},\;\;1\le p \le \infty .
\]
Then $u\in L^\infty$, $\nabla u\in L^\infty$, and 
\be\label{d-v-infty-r}
\|u\|_\infty  \lesssim \|u\|^{(r + 1 - d/p)/(r + 1 - d/p + d/2)}\;\|\omega\;\big|\;{\dot B}^r_{p, p}\|^{(d/2)/(r + 1 - d/p + d/2)}
\ee
and
\be\label{d-nabla-v-infty-r}
\|\nabla u\|_\infty  \lesssim \|u\|^{(r - d/p)/(r + 1 - d/p + d/2)}\;\|\omega\;\big|\;{\dot B}^r_{p, p}\|^{(1 + d/2)/(r + 1 - d/p + d/2)}\,.
\ee
\item  In the case $d \ge 3$ assume that $u\in L^2$ and $\omega = \curl\, u \in {\dot H}^\theta$.  If $\theta > \frac{d}{2} - 1$, then $u\in L^\infty$ and 
\be
\|u\|_\infty \lesssim \|u\|^{(\theta + 1 - d/2)/(\theta + 1)}\;[\omega]_\theta^{d/2/(\theta + 1)}\,.
\ee
If $\theta > \frac{d}{2}$, then $\nabla u\in L^\infty$ and 
\be
\|\nabla u\|_\infty \lesssim \|u\|^{(\theta - d/2)/(\theta + 1)}\,[\omega]_\theta^{(d/2 + 1)/(\theta + 1)}\,.
\ee
\end{enumerate} 
\end{lemma}

The inequalities follow immediately from Lemma \ref{HA} part 2 and the observation that $[\omega]_{-1} \simeq \|u\|$ and 
$\nabla u \simeq {\cal R} \omega$.

\bibliography{CH}
\newcommand{\arxivref}[1]{\href{http://www.arxiv.org/abs/#1}{{arXiv.org:#1}}}
\newcommand{\prd}{Phys. Rev. D} 
\bibliographystyle{abbrv}

\end{document}